\documentclass[10pt]{article}

\usepackage{amssymb}
\usepackage{amsmath}
\usepackage{algpseudocode}
\usepackage{graphicx}
\usepackage{url,here}
\usepackage[backref,colorlinks=true]{hyperref}
\usepackage{multirow}
\usepackage{algorithm}
\allowdisplaybreaks

\begin{document}
\newenvironment {proof}{{\noindent\bf Proof.}}{\hfill $\Box$ \medskip}

\newtheorem{theorem}{Theorem}[section]
\newtheorem{lemma}[theorem]{Lemma}
\newtheorem{condition}[theorem]{Condition}
\newtheorem{proposition}[theorem]{Proposition}
\newtheorem{remark}[theorem]{Remark}
\newtheorem{definition}[theorem]{Definition}
\newtheorem{hypothesis}[theorem]{Hypothesis}
\newtheorem{corollary}[theorem]{Corollary}
\newtheorem{example}[theorem]{Example}
\newtheorem{descript}[theorem]{Description}
\newtheorem{assumption}[theorem]{Assumption}

\newcommand{\ba}{\begin{align}}
\newcommand{\ea}{\end{align}}

\def\P{\mathbb{P}}
\def\R{\mathbb{R}}
\def\E{\mathbb{E}}
\def\N{\mathbb{N}}
\def\Z{\mathbb{Z}}

\renewcommand {\theequation}{\arabic{section}.\arabic{equation}}
\def \non{{\nonumber}}
\def \hat{\widehat}
\def \tilde{\widetilde}
\def \bar{\overline}

\def\ind{{\mathchoice {\rm 1\mskip-4mu l} {\rm 1\mskip-4mu l}
{\rm 1\mskip-4.5mu l} {\rm 1\mskip-5mu l}}}

\title{\Large\ {\bf Sensitivity analysis for stochastic chemical reaction networks with multiple time-scales}}

\author{Ankit Gupta and Mustafa Khammash \\
Department of Biosystems Science and Engineering \\ ETH Zurich \\  Mattenstrasse 26 \\ 4058 Basel, Switzerland. 
}
\date{\today}
\maketitle

\begin{abstract}
Stochastic models for chemical reaction networks have become very popular in recent years. For such models, the estimation of parameter sensitivities is an important and challenging problem. Sensitivity values help in analyzing the network, understanding its robustness properties and also in identifying the key reactions for a given outcome. Most of the methods that exist in the literature for the estimation of parameter sensitivities, rely on Monte Carlo simulations using Gillespie's stochastic simulation algorithm or its variants. It is well-known that such simulation methods can be prohibitively expensive when the network contains reactions firing at different time-scales, which is a feature of many important biochemical networks. For such networks, it is often possible to exploit the time-scale separation and approximately capture the original dynamics by simulating a ``reduced" model, which is obtained by eliminating the fast reactions in a certain way. The aim of this paper is to tie these model reduction techniques with sensitivity analysis. We prove that under some conditions, the sensitivity values of the reduced model can be used to approximately recover the sensitivity values for the original model. Through an example we illustrate how our result can help in sharply reducing the computational costs for the estimation of parameter sensitivities for reaction networks with multiple time-scales. To prove our result, we use coupling arguments based on the random time change representation of Kurtz. We also exploit certain connections between the distributions of the occupation times of Markov chains and multi-dimensional wave equations. 
\end{abstract}



\medskip
\setcounter{equation}{0}

\section{Introduction}
\label{sec:intro}
Chemical reaction networks have traditionally been studied using deterministic models that express the dynamics as a set of ordinary differential equations. Such models ignore the randomness 
in the dynamics which is caused by the discrete nature of molecular interactions. It is now widely accepted that this randomness can have a significant impact on the macroscopic properties of the 
system \cite{Goutsias,McAdams,Levin}, when the molecules are present in low copy numbers. To account for this randomness and study its effects, a stochastic formulation of the dynamics is necessary, and the most common choice is to model the dynamics as a continuous time Markov process. Such stochastic models have been extensively used in many recent articles \cite{Elowitz,Rao,Kierzek,Arkin,Noise,Haseltine} to understand the biological implications of random dynamics. For a detailed survey of Markov models for chemical reaction networks we refer the readers to \cite{DASurvey}.

Typically, a chemical reaction network depends on various kinetic parameters whose values are uncertain or suffer from measurement error. To determine the effects of inaccuracies in the parameter values, one needs to estimate the sensitivities of a given output with respect to the parameter values. If an output is highly sensitive to a specific parameter value, then greater time and effort may be invested in determining that parameter precisely. Such sensitivity values can also be useful in fine-tuning a certain output (see \cite{Feng}) or understanding the robustness properties of a system (see \cite{Stelling}). 

Estimation of parameter sensitivities is fairly straightforward for deterministic models, but it poses a major challenge for stochastic models. Many methods have been proposed in the literature for tackling this problem \cite{IRN,KSR1,KSR2,DA,Gupta}. However all these methods reply on extensive simulations of the stochastic model, which is usually carried out using Gillespie's Stochastic Simulation Algorithm \cite{GP} or its variants \cite{NR,GPSurvey}. These simulation methods account for each and every reaction event, which makes them prohibitively expensive, when the network consists of reactions firing at different time-scales. In such a scenario, the ``fast" reactions take up most of the computational time causing the simulation method to become very inefficient. Since time-scale separation is a feature of many important biochemical networks \cite{PWOLDE}, a new class of methods have been designed to exploit this feature and efficiently simulate the stochastic model \cite{ssSSA,weinan1,weinan2}. These methods simulate a ``reduced" model which is obtained by eliminating the fast components of the dynamics through a quasi-steady state approximation \cite{HR2,RaoArkin2}. Such reduced models capture the original dynamics in an approximate sense and the error in approximation disappears as the time-scale separation gets larger and larger. In \cite{HWKang}, Kang and Kurtz develop a systematic theoretical framework for constructing these reduced models. As discussed in \cite{ssSSA} and elsewhere, simulations of reduced models are generally much faster than the original model. Since most sensitivity estimation algorithms are simulation-based, it is of interest to determine if the parameter sensitivities for the original model can be approximated by the parameter sensitivities for the reduced model. Our aim in this paper is to present a theoretical result which shows that can indeed be done under certain conditions. Therefore one can obtain enormous savings in the computational costs required for the estimation of parameter sensitivities for stochastic models of multiscale reaction networks. From now on, the term ``multiscale network" refers to a chemical reaction network which consist of reactions firing at different time-scales.

It is observed in \cite{HWKang} that variations in the reaction time-scales could be both due to variation in species numbers and due to variation in rate constants. However in this paper we will only consider the latter source of variation. We now describe our stochastic model of a multiscale chemical reaction network. Suppose we have a well-stirred system consisting of $d$ chemical species. Its state at any time can be described by a vector in $\N^{d}_0$ whose $i$-th component is the non-negative integer corresponding to the number of molecules of the $i$-th species. These chemical species interact through $K$ predefined reaction channels and every time the $k$-th reaction fires, the state of the system is displaced by the $d$-dimensional stoichiometric vector $\zeta_k  \in  \Z^d$. If the state of the system is $x$, the rate at which the $k$-th reaction fires is given by $N_0^{\beta_k} \lambda_k(x)$, where $N_0$ is assumed to be a ``large" normalization parameter and $\lambda_k : \N^d_0 \to [0,\infty)$ is the \emph{propensity function} for the $k$-th reaction. The powers of $N_0$ in front of the \emph{propensity functions}, determine the various time-scales at which different reactions act.
In a stochastic setting, such a chemical reaction network can be modeled as a continuous time Markov process $\{X^{N_0}(t) : t \geq 0\}$ over $\mathbb{N}^d_0$. Given such a reaction network we have the flexibility of selecting our reference time-scale as $\gamma$. This means that we observe the reaction dynamics at times that are scaled by the factor $N_0^\gamma$. 
In other words, we observe the process $\{X^{N_0}_{\gamma}(t) : t \geq0\}$ defined by
\begin{align*}
X^{N_0}_{\gamma}(t)  = X^{N_0}(t N^{\gamma}_0) \quad \textnormal{ for } t \geq 0.
\end{align*}
Note that in the process $X^{N_0}_{\gamma}$, each reaction $k$ fires at a rate of order $N_0^{\beta_k+\gamma}$.
Hence reactions can be termed as ``fast", ``slow" or ``natural" according to whether $\beta_k + \gamma  > 0$, $\beta_k + \gamma  < 0$ or $\beta_k + \gamma = 0$ respectively. Note that as the value of $N_0$ increases, the slow reactions get slower and the fast reactions get faster. 
On the other hand, the natural reactions remain unaffected by the increase in $N_0$.
If we simulate the process $X^{N_0}_{\gamma}$ using Gillespie's Stochastic Simulation Algorithm, then the fast reactions take up most of the computational time, making the simulation procedure extremely cumbersome.   

Fortunately in certain situations, we can obtain a fairly good approximation of the dynamics by simulating a reduced model which does not contain any fast reactions. The state variables in this reduced model correspond to \emph{linear combinations} of species numbers that are unaffected by the fast reactions (see \cite{ssSSA,weinan1}).  As described in \cite{HWKang}, such model reductions can be derived by replacing $N_0$ by $N$ and showing that for a certain projection map $\Pi$ on $\R^d$, the sequence of processes $\{ \Pi X^{N}_{\gamma} : N \in \N \}$ has a well-defined limit as $N \to \infty$. The limiting process $\hat{X}$ corresponds to the stochastic model of a reduced reaction network made up of only those reactions that are ``natural" for the reference time-scale $\gamma$, making its simulation far less computationally demanding than the original model.
 In Section \ref{sec:app_results} we present these model reduction results in greater detail. Now suppose that the output of interest is given by a real-valued function $f$ and we would like to estimate the expectation $ \E\left( f( X^{N_0}_{\gamma}(t)  ) \right)$ for some observation time $t \geq 0$. If $f$ is invariant under the projection $\Pi$ (that is, $f(x) = f( \Pi x)$ for all $x \in \N^d_0$) then we would expect that  
\begin{align}
\label{introd:approx}
\lim_{N \to \infty} \E\left( f( X^{N}_{\gamma}(t)  ) \right) =\lim_{N \to \infty} \E\left( f( \Pi X^{N}_{\gamma}(t)  ) \right)  = \E\left(  f( \hat{X}(t) ) \right).
\end{align}
This limit implies that for large values of $N_0$, the quantity $\E( f( X^{N_0}_\gamma(t)  ) )$ is ``close" to $ \E(  f( \hat{X}(t) )$. Hence instead of estimating the former quantity directly we can estimate the latter quantity through simulations of the reduced model, and save a significant amount of computational effort.

As stated before, our aim in this paper is to tie these model reduction results with sensitivity analysis. Suppose that the propensity functions $\lambda_1,\dots,\lambda_K$ depend on a scalar parameter $\theta$. Now when the state is $x$, the $k$-th reaction fires at rate $ N_0^{\beta_k} \lambda_k(x,\theta)$. With these propensity functions, we can define the processes $X^{N_0}_{\gamma,\theta}$ and $X^{N}_{\gamma,\theta}$ as before, where the subscript $\theta$ is introduced to make the parameter dependence explicit. For an output function $f$ chosen as above, we would like to estimate the sensitivity of the expectation $\E( f( X^{N_0}_{\gamma}(t)  ) )$ with respect to $\theta$. In other words, we are interested in estimating
\begin{align}
\label{paramsensn0}
S^{N_0}_{\gamma ,\theta }(f,t) = \frac{\partial }{\partial \theta} \E\left( f(   X^{N_0}_{\gamma ,\theta } (t)  ) \right).
\end{align}
We remarked before that most direct methods to estimate this quantity are simulation-based. Since simulations of the process $X^{N_0}_{\gamma ,\theta } $ are very expensive, it is worthwhile to explore the possibility of using reduced models to obtain a close approximation for $S^{N_0}_{\gamma ,\theta }(f,t) $. Suppose that for each $\theta$ we have a process $\hat{X}_\theta$ which corresponds to the reduced model. Moreover there exists a projection $\Pi$ (independent of $\theta$) such that $\Pi X^{N}_{\gamma,\theta}$ converges in distribution to $\hat{X}_\theta$ as $N \to \infty$. Then similar to \eqref{introd:approx} we would get
\begin{align*}
\lim_{N \to \infty} \E\left( f( X^{N}_{\gamma,\theta}(t)  ) \right)  = \E\left(  f( \hat{X}_\theta(t) ) \right).
\end{align*}
However this relation does not ensure that
\begin{align}
\label{senisitivitylimitscommute}
\lim_{N \to \infty} \frac{\partial }{\partial \theta} \E \left(  f( X^{N}_{\gamma, \theta} (t) ) \right) = \frac{\partial }{\partial \theta} \left( \lim_{N \to \infty} \E \left(  f( X^{N}_{\gamma,\theta} (t) ) \right) \right) = \frac{\partial }{\partial \theta} \E \left(  f( \hat{X}_\theta (t) ) \right) ,
\end{align}
because in general, limits and derivatives do not commute. Note that if \eqref{senisitivitylimitscommute} holds then for large values of $N_0$, the quantity $S^{N_0}_{\gamma ,\theta }(f,t) $ is close to the value
\begin{align*}
\hat{S}_{\theta }(f,t) = \frac{\partial }{\partial \theta} \E \left(  f( \hat{X}_\theta (t) ) \right),
\end{align*}
which can be easily estimated using any of the sensitivity estimation methods \cite{IRN,KSR1,KSR2,DA,Gupta}, since simulations of the reduced model is computationally much easier than the original model.
This motivates the main result of the paper which is essentially to show that \eqref{senisitivitylimitscommute} holds under certain conditions. In the above discussion we had assumed that the output function $f$ is invariant under the projection $\Pi$, which is a highly restrictive assumption. Therefore we will prove a relation analogous to \eqref{senisitivitylimitscommute} for a general function $f$.

Even though our result is easy to state, its proof is quite technical. The main complication comes from the fact that the dynamics at different time-scales, may interact with each other in non-linear ways. Due to this problem, the proof of our main result involves several steps which are loosely described below. We mentioned above that for a certain projection $\Pi$, the process $\Pi X^{N}_{\gamma,\theta}$ may have a well-defined limit as $N \to \infty$. In such a situation, the \emph{left-over} part of the process, $(I -\Pi) X^{N}_{\gamma,\theta}$\footnote{Here $I$ is the identity projection}, does not converge in the \emph{functional} sense but it converges in the sense of occupation measures (see \cite{HWKang} or Section \ref{sec:app_results}). As reported in \cite{Bruno}, the distribution of occupation measures of Markov processes is related to the evolution of a system of multi-dimensional wave equations. Using this relation we construct another process $W^N_\theta$ whose distribution has some regularity properties with respect to $\theta$. The process $W^N_\theta$ captures the one-dimensional distribution of the process $X^{N}_{\gamma,\theta}$, which means that for any function $f$ and time $t$, we can find a function $g$ such that
$$  \E \left(  f( X^{N}_{\gamma, \theta} (t) ) \right) =  \E \left(  g( W^{N}_{\theta} (t) ) \right).  $$
Furthermore, the fast components of the dynamics are \emph{averaged} out in the process $W^N_\theta$, making it simpler to analyze than the original process $X^{N}_{\gamma, \theta} $. 
Next we couple the processes $W^N_\theta$ and $W^N_{\theta+h}$ (for a small $h$) in such a way, that it allows us to take the limits $h \to 0$ and $N \to \infty$ (in this order) of an appropriate quantity and prove our main result. This coupling is constructed using the random time change representation of Kurtz (see Chapter 7 in \cite{EK}). 

As a corollary of our main result we obtain an important relationship which can be useful in estimating steady-state parameter sensitivities. Let $X_\theta$ be a stochastic process which models the dynamics of the reaction network described above, with $\beta_k = 0$ for each $k$ and $\gamma=0$. Assume that this process is ergodic with stationary distribution $\pi_\theta$ and this distribution is difficult to compute analytically. Ergodicity implies that for any \emph{output} function $f$ we have
\begin{align*}
\lim_{t \to \infty} \E\left( f(X_\theta(t))  \right) =  \left( \int f(y) \pi_\theta(dy) \right),
\end{align*}
where the integral is taken over the state space of $X_\theta$. Suppose we are interested in computing the steady-state parameter sensitivity given by
\begin{align*}
\frac{d}{d \theta } \left( \int f(y) \pi_\theta(dy) \right).
\end{align*}
Since $\pi_\theta$ is unknown, this quantity cannot be computed directly and one has to estimate it using simulations. This can be problematic because simulations can only be performed until a finite time, and in general one is not sure if the sensitivity value estimated at a finite (but large $t$) is close to the steady-state value. However using our main result, we can conclude that under certain conditions we have
\begin{align}
\label{intro:sssenss}
\lim_{t \to \infty} \frac{\partial }{\partial \theta} \E\left( f(X_\theta(t))  \right) =   \frac{d }{d\theta} \left( \int f(y) \pi_\theta(dy) \right).
\end{align}
The details are given in Section \ref{subsec:ssparamsens}. Relation \ref{intro:sssenss} proves that for a large (but finite) $t$, the steady-state parameter sensitivity is well-approximated by 
$$ \frac{\partial }{\partial \theta}\E\left( f(X_\theta(t))  \right)  $$
which can be estimated using known simulation-based methods \cite{IRN,KSR1,KSR2,DA,Gupta}. Note that \eqref{intro:sssenss} is sometimes implicitly assumed (see \cite{PW} for example) without proof.

All the results in the paper are stated for a scalar parameter $\theta$, but the extension of these results for vector-valued parameters is relatively straightforward.
Finally we would like to mention the even though our paper is written in the context of chemical reaction networks, our main result can be applied to any continuous time Markov process
over a discrete lattice with time-scale separation in the transition rates. Other than reaction networks, such processes arise naturally in queuing theory and population modeling.  
  
This paper is organized as follows. In Section \ref{sec:app_results} we discuss the model reduction results for multiscale networks. The results stated there are simple adaptations of the results in \cite{HWKang}. Our main result is presented in Section \ref{sec:mainresult} and its proof is given in Section \ref{sec:proof}. In Section \ref{sec:example} we provide an illustrative example to show how our result can be useful.

\section*{Notation}
We now introduce some notation that we will use throughout this paper. Let $\R$, $\R_+$, $\Z$, $\N$ and $\N_{0}$ denote the sets of all reals, nonnegative reals, integers, positive integers and nonnegative integers respectively. For any $a,b \in \R$, their minimum is given by $a \wedge b$. The positive and negative parts of $a$ are indicated by $a^{+}$ and $a^{-}$ respectively. The number of elements in any finite set $E$ is denoted by $|E|$. By $\textnormal{Unif}(0,1)$  we refer to the uniform distribution on $(0,1)$.
If $\Pi$ is a \emph{projection map} on $\R^n$ then we write $\Pi x$ instead of $\Pi(x)$ for any $x \in \R^n$ and for any $S \subset \R^n$, the set $\Pi S$ is given by
\begin{align*}
\Pi S = \{ \Pi x : x \in S \}.
\end{align*}

For any $n \in \N$, $\langle \cdot , \cdot\rangle$ is the standard inner product in $\R^n$. Moreover for any $v = (v_1,\dots,v_n)\in \R^n$, $\|v\|$ is the $1$-norm defined by
$\left\| v\right\|= \sum_{i=1}^n |v_i|$. The vectors of all zeros and all ones in $\R^n$ are denoted by $\bar{0}_n$ and $\bar{1}_n$ respectively.
Let $\mathbb{M}(n,n)$ be the space of all $n \times n$ matrices with real entries. For any 
$M \in \mathbb{M}(n,n)$, the entry at the $i$-th row and the $j$-th column is indicated by $M_{ij}$. The transpose and inverse of $M$ are indicated by $M^{T}$ and $M^{-1}$ respectively. 
The symbol $I_n$ refers to the identity matrix in $\mathbb{M}(n,n)$. For any $v = (v_1,\dots,v_n) \in \R^n$, $\textrm{Diag}(v)$ refers to the matrix in $\mathbb{M}(n,n)$ whose non-diagonal entries are all $0$ and 
whose diagonal entries are $v_1,\dots,v_n$. A matrix in $\mathbb{M}(n,n)$ is called \emph{stable} if all its eigenvalues have strictly negative real parts. While multiplying a matrix with a vector we always regard the vector as a column vector.

Let $(S,d)$ be a metric space. Then by $\mathcal{B}(S)$ we refer to the set of all bounded real-valued Borel measurable functions on $S$. By $\mathcal{P}(S)$ we denote the space of all Borel probability measures on $S$. This space is equipped with the weak topology. The space of cadlag functions (that is, right continuous functions with left limits) from $[0,\infty)$ to $S$ is denoted by $D_{S} [0,\infty)$ and it is endowed with the Skorohod topology (for details see Chapter 3, Ethier and Kurtz \cite{EK}). For any $f \in D_{S} [0,\infty)$ and $t>0$, $f(t-)$ refers to the left-limit $\lim_{s \to t^{-}} f(s)$.

An operator $A$ on $\mathcal{B}(S)$ is a linear mapping that maps any function in its domain $\mathcal{D}(A) \subset \mathcal{B}(S)$ to a function in $\mathcal{B}(S)$.
The notion of the \emph{martingale problem} associated to an operator $A$ is introduced and developed in Chapter 4, Ethier and Kurtz \cite{EK}. 
In this paper, by a solution of the martingale problem for $A$ we mean a measurable stochastic process $X$ with paths 
in $D_{S} [0,\infty)$ such that for any $f \in \mathcal{D}(A)$,
\[f(X(t)) - \int_{0}^{t} A f(X(s))ds\]
is a martingale with respect to the filtration generated by $X$. 
For a given initial distribution $\pi \in \mathcal{P}(S)$, a solution $X$ of the martingale problem for $A$ is a solution of the martingale problem for $(A,\pi)$ if $\pi = \P X(0)^{-1}$. 
If such a solution $X$ exists uniquely for all $\pi \in \mathcal{P}(S)$, then we say that the martingale problem for $A$ is well-posed. Additionally, we say that 
$A$ is the generator of the process $X$.

Throughout the paper $\Rightarrow$ denotes convergence in distribution.

\section{Model Reduction results for multiscale networks}
\label{sec:app_results}
In this section we present the model reduction results for multiscale networks. Recall the definition of the process $X^{N}_{\gamma}$ from Section \ref{sec:intro}. We shall soon see that this process is well-defined under some assumptions on the propensity functions. Our primary goal in this section, is to find the values of the reference time-scale $\gamma$ such that the process $X^{N}_{\gamma}$ has a well-behaved limit as $N \to \infty$. This limit may not exist for the whole process but only for a suitable projection of the process. When the limit exists, the limiting process can be viewed as the stochastic model of a reduced reaction network, which only has reactions firing at a \emph{single} time-scale. The results mentioned in this section are derived from the more general results in \cite{HWKang}. Before we proceed we define a property of real-valued functions.
\begin{definition}
\label{polynomialgrowth}
Let $U$ be a subset of $\R^m$, $f$ be a real-valued function on $U$ and $\Pi$ be a projection map on $\R^m$. We say that the function $f$ is polynomially growing with respect to projection $\Pi$ if there exist constants $C,r > 0$ such that
\begin{align}
\label{polynomial_growth_pi}
|f(x)| \leq C(1 + \| \Pi x\|^r) \ \textnormal{ for all } x \in U.
\end{align}
We say that a function $f$ in linearly growing with respect to projection $\Pi$ if \eqref{polynomial_growth_pi} is satisfied for $r =1$.
A sequence of real-valued functions $\{f^N: N \in \N\}$ on $U$ is said to be polynomially (linearly) growing with respect to projection $\Pi$ if for some $C>0$ and $r >0$ ($r=1$), the relation \eqref{polynomial_growth_pi} holds for each $f^N$. A function (or a sequence of functions) is called polynomially (linearly) growing if it is polynomially (linearly) growing with respect to the identity projection $I$.
\end{definition}

Our first task is to ensure that there is a well-defined process which describes the stochastic dynamics of our multiscale reaction network. For this purpose we make certain assumptions.
\begin{assumption}
\label{assumptions1}
The propensity functions $\lambda_1,\dots, \lambda_K$ satisfy the following conditions.
 \begin{itemize}
\item[(A)] For any $k$ and $x \in \N^d_0$, if $\lambda_k(x) >0$ then $(x+\zeta_k)$ has all non-negative components.
\item[(B)] Let P be the set of those reactions which have a net positive affect on the total population, that is,
\begin{align}
\label{defn_positivereaction}
P = \{ k =1,\dots,K : \langle \bar{1}_d , \zeta_k \rangle > 0\}.
\end{align}
Then the function $\lambda_P : \N^d_0 \to \R_+$ defined by $\lambda_P(x) = \sum_{k \in P}\lambda_k(x)$ is linearly growing.
 \end{itemize}
\end{assumption}
Parts (A) of this assumption prevents the reaction dynamics from leaving the state space $\N^d_0$. The significance of part (B) will become clear in the next paragraph. Informally, part (B) says that all the reactions that add molecules into the system have orders $0$ or $1$. If there is a compact set S such that for each $k$, $\lambda_k(x) = 0$ for all $ x \notin S$, then part (B) is trivially satisfied.

Let $x_0$ be a vector in $\N^d_0$. Throughout the paper, the initial state of the reaction dynamics is fixed to be $x_0 \in \N^d_0$ and the corresponding \emph{stoichiometric compatibility class} is given by
\begin{align*}
\mathcal{S} = \left\{ x_0 + \sum_{k=1}^K \eta_k \zeta_k  \in \N^d_0 : \eta_1,\dots,\eta_K \in \N_0 \right\}.
\end{align*}
Part (A) of Assumption \ref{assumptions1} ensures that the reaction dynamics is always inside $\mathcal{S}$. From the description of the multiscale network with reference time-scale $\gamma$ (see  Section \ref{sec:intro}), it is clear that the generator of the reaction dynamics should be given by the operator $\mathbb{A}^N_\gamma$ whose domain is $\mathcal{D}( \mathbb{A}^N_\gamma ) = \mathcal{B}( \mathcal{S})$ and its action on any $ f\in \mathcal{B}( \mathcal{S})$ is given by
\begin{align}
\label{defn_gen_aNgamma}
\mathbb{A}^N_\gamma f(x) = \sum_{k=1}^K  N^{ \beta_k+ \gamma} \lambda_k(x) ( f(x+\zeta_k) -f(x) ).
\end{align}
From Lemma \ref{lemma1:app} we can argue that under Assumption \ref{assumptions1}, the martingale problem for $\mathbb{A}^N_\gamma$ is well-posed. Hence we can define $X^N_\gamma$ as the Markov process with generator $\mathbb{A}^N_\gamma$ and initial state $x_0$. The random time change representation (see Chapter 7 in \cite{EK}) of this process is given by
\begin{align}
\label{rtcrep1}
X^{N}_{\gamma}(t) = x_0 + \sum_{k=1}^K Y_k \left(  N^{\beta_k + \gamma} \int_{0}^t \lambda_k( X^{N}_{\gamma} (s) )ds \right)\zeta_k,
\end{align}
where $\{Y_k : k=1,\dots,K\}$ is a family of independent unit rate Poisson processes.

\subsection{Convergence at the first time-scale} \label{sec:firsttimescaleconv}

From \eqref{rtcrep1}, it is immediate that if the reference time-scale $\gamma$ is such that $\beta_k + \gamma \leq 0$ for each $k$, then all the reactions are either ``slow" or ``natural" at this time-scale\footnote{ The jargon of ``slow" , ``fast" and ``natural" reactions was introduced in Section \ref{sec:intro}}. Therefore we would expect the dynamics to converge as $N \to \infty$ and the limiting dynamics will only consist of the natural reactions.

To make this precise, define
\begin{align}
\label{firsttimescale}
\gamma_1 = - \max \{ \beta_k : k =1,\dots,K \} \ \textnormal{ and } \  \Gamma_1 = \{  k =1,\dots.K : \beta_k = - \gamma_1  \}.
\end{align}
Then $\gamma_1$ is the first time-scale for which the process $X^{N}_{\gamma_1}$ has a non-trivial limit as $N \to \infty$ and $\Gamma_1$ is the set of natural reactions for this time-scale.
Note that
\begin{align*}
\beta_k + \gamma_1 
\left\{ 
\begin{array}{cc}
 = 0 & \textnormal{ if } k \in \Gamma_1 \\
< 0 &  \textnormal{ if } k \notin \Gamma_1, 
\end{array}
\right.
\end{align*}
and hence using \eqref{rtcrep1} we can show that $X^{N}_{\gamma_1} \Rightarrow \hat{X}$ as $N \to \infty$, where the process $\hat{X}$ satisfies
 \begin{align}
 \label{rtcrep2}
\hat{X}(t) = x_0 + \sum_{k \in \Gamma_1} Y_k \left( \int_{0}^t \lambda_k(  \hat{X}(s) )ds \right)\zeta_k.
\end{align} 
In other words, $\hat{X}$ is the process with initial state $x_0$ and generator $\mathbb{C}_0$ given by
 \begin{align}
\label{defngenmathbbc}
\mathbb{C}_0 f(x) = \sum_{k \in \Gamma_1} \lambda_k(x) \left( f(x + \zeta_k) -f(x)  \right)  \ \textnormal{ for } f \in  \mathcal{D}(\mathbb{C}_0 ) = \mathcal{B}(\mathcal{S}).
\end{align}
The well-posedness of the martingale problem for $\mathbb{C}_0$ can be verified from Lemma \ref{lemma1:app} and therefore the process $\hat{X}$ is well-defined. 
The precise statement of this convergence result is given below.
\begin{proposition}
\label{convergenceresult1} 
Suppose that the propensity functions $\lambda_1,\dots,\lambda_K$ satisfy Assumption \ref{assumptions1}. Then we have $X^{N}_{\gamma_1} \Rightarrow \hat{X}$ as $N \to \infty$ where the limiting process $\hat{X}$ satisfies \eqref{rtcrep2}.
\end{proposition}
\begin{proof}
The proof follows easily from Theorem 4.1 in \cite{HWKang}.
\end{proof} 

Observe that this proposition can be viewed as a model reduction result, which says that at the time-scale $\gamma_1$, the dynamics of the original model (given by $X^{N_0}_{\gamma_1} $) is well-approximated by the dynamics of a reduced model (given by $\hat{X}$) for large values of $N_0$. This reduced model is obtained by simply dropping the ``slow" reactions from the network. Such a model reduction result is \emph{trivial} because one can easily see from the reaction time-scales that the slow reactions will not participate in the limiting dynamics. In the next section we describe a non-trivial model reduction result which is more useful from the point of view of applications.

\subsection{Convergence at the second time-scale} \label{sec:secondtimescaleconv}

As discussed in several recent papers \cite{Ball,HWKang}, there may be a second time-scale $\gamma_2$ ($ >\gamma_1$) so that a certain projection $\Pi_2$ of the process $X^N_{\gamma_2}$ has a well-behaved limit as $N \to \infty$. At this second time-scale, the network has ``fast" reactions in addition to the ``slow" and ``natural" reactions. The projection $\Pi_2$ is such, that the fast reactions do not affect the projected process $\Pi_2 X^N_{\gamma_2} $. Assuming quasi-stationarity for the fast sub-network \cite{HR2,RaoArkin2} we can have a well-defined limit $\hat{X}$ for the process $\Pi_2 X^N_{\gamma_2}$. Moreover the limiting process $\hat{X}$ corresponds to the stochastic model of a reduced reaction network which only contains those reactions that are natural for the time-scale $\gamma_2$.

We now describe this convergence result formally. Suppose that the set
\begin{align*}
\mathbb{S}_2 = \{ v \in  \R_+^d : \langle v, \zeta_k \rangle =0 \ \textnormal{ for all } \ k \in \Gamma_1 \}
\end{align*}
is non-empty. Then for any $v \in \mathbb{S}_2$, the process $\{ \langle v , X^{N}_{\gamma_2}(t) \rangle  : t \geq 0 \}$ is unaffected by the reactions in $\Gamma_1$. 
Let  $\gamma_v = - \max \{ \beta_k : k =1,\dots,K \ \textnormal{ and } \ \langle v, \zeta_k \rangle \neq 0 \}$ and define
\begin{align}
\label{secondtimescale}
\gamma_2 = \inf \{ \gamma_v : v \in \mathbb{S}_2\} \ \textnormal{ and } \  \Gamma_2 = \{  k =1,\dots.K : \beta_k = - \gamma_2  \}.
\end{align}
Then $\gamma_2 > \gamma_1$ by definition and note that the reactions in $\Gamma_1$ are fast at the time-scale $\gamma_2$. Let $\mathbb{L}_2$ be the subspace spanned by the vectors in $\mathbb{S}_2$ and let $\Pi_2$ be the projection map from $\R^d$ to $\mathbb{L}_2$. The definition of $\mathbb{L}_2$ implies that
\begin{align}
\label{fastreactionsdonotaffect}
\Pi_2 \zeta_k = \bar{0}_d \ \textnormal{ for all } k \in \Gamma_1,
\end{align}
which means that the fast reactions would leave the process $\Pi_2 X^N_{\gamma_2} $ unchanged.
Let $\mathbb{L}_1$ be the space spanned by the vectors in $(I-\Pi_2)\mathcal{S} = \{ (I -\Pi_2) x : x \in \mathcal{S}\}$, where $I$ is the identity map. For any $v \in \Pi_2\mathcal{S}$ let
\begin{align}
\label{defn_ev}
\mathbb{H}_v = \{ y \in \mathbb{L}_1 : y = (I- \Pi_2)x , \ \Pi_2 x = v \ \textnormal{ and } x \in \mathcal{S} \}
\end{align}
and define the operator $\mathbb{C}^v$ by
\begin{align}
\label{defn_cv}
\mathbb{C}^v f(z) = \sum_{k \in \Gamma_1} \lambda_k(v+z) \left( f(z + \zeta_k) -f(z)  \right)  \ \textnormal{ for } f \in  \mathcal{D}(\mathbb{C}^v ) = \mathcal{B}(\mathbb{H}_v ).
\end{align}
The operator $\mathbb{C}^v$ can be seen as the generator of a Markov process with state space $\mathbb{H}_v$.

We now define the \emph{occupation measure} of the process $ (I - \Pi_2) X^{N}_{\gamma_2}$. This is a random measure on $\mathbb{L}_1 \times [0,\infty)$ given by
\begin{align*}
V^{N}_{\gamma_2}(C \times [0,t]) = \int_{0}^t \ind_{C}\left(  (I - \Pi_2) X^{N}_{\gamma_2}(s)  \right) ds,
\end{align*}
where $C$ is any Borel measurable subset of $\mathbb{L}_1$. Note that for any $k$
\begin{align*}
\int_{0}^t \lambda_k( X^{N}_{\gamma_2}(s)   )ds  =  \int_{0}^{t} \int_{\mathbb{L}_1} \lambda_k(  \Pi_2 X^{N}_{\gamma_2}(s)  +y  ) V^{N}_{\gamma_2}(dy \times ds). 
\end{align*}
Therefore using \eqref{rtcrep1} and \eqref{fastreactionsdonotaffect}, we can write the random time change representation for the process $\Pi_2 X^{N}_{\gamma_2}$ as 
\begin{align}
\label{rtc_rep3}
\Pi_2 X^{N}_{\gamma_2}(t)  & = \Pi_2 x_0 + \sum_{k \in \Gamma_1} Y_k \left( N^{\beta_k + \gamma} \int_{0}^t \lambda_k( X^{N}_{\gamma_2}(s)   )ds  \right) \Pi_2 \zeta_k  + \sum_{k \in \Gamma_2} Y_k \left( N^{\beta_k + \gamma}  \int_{0}^t \lambda_k( X^{N}_{\gamma_2}(s)   )ds  \right) \Pi_2 \zeta_k  \notag \notag \\
&+ \sum_{k \notin  \Gamma_1 \cup \Gamma_2} Y_k \left( N^{\beta_k + \gamma}  \int_{0}^t \lambda_k( X^{N}_{\gamma_2}(s)   )ds  \right) \Pi_2 \zeta_k  \notag \\
& = \Pi_2 x_0 + \sum_{k \in \Gamma_2} Y_k \left( N^{\beta_k + \gamma}   \int_{0}^{t} \int_{\mathbb{L}_1} \lambda_k(  \Pi_2 X^{N}_{\gamma_2}(s) +y  ) V^{N}_{\gamma_2}(dy \times ds)\right) \Pi_2 \zeta_k  \\
&+ \sum_{k \notin  \Gamma_1 \cup \Gamma_2} Y_k \left( N^{\beta_k + \gamma}  \int_{0}^{t}  \int_{\mathbb{L}_1} \lambda_k(  \Pi_2 X^{N}_{\gamma_2}(s)  +y  ) V^{N}_{\gamma_2}(dy \times ds) \right) \Pi_2 \zeta_k.  \notag
\end{align}
Suppose that $V^{N}_{\gamma_2} \Rightarrow V$ as $N \to \infty$. In other words, for any $f \in \mathcal{B}( \mathcal{S})$ and $t>0$
\begin{align*}
\int_{0}^t \int_{\mathbb{L}_1}  f(x)V^{N}_{\gamma_2}(dx \times ds) \Rightarrow  \int_{0}^t \int_{\mathbb{L}_1}  f(x) V(dx \times ds) \ \textnormal{ as } \ N \to \infty.
\end{align*}
Since
\begin{align*}
\beta_k + \gamma_2 
\left\{ 
\begin{array}{cc}
 = 0 & \textnormal{ if } k \in \Gamma_2 \\
< 0 &  \textnormal{ if } k \notin \Gamma_1 \cup \Gamma_2,
\end{array}
\right.
\end{align*}
we can expect from \eqref{rtc_rep3} that $\Pi_2 X^{N}_{\gamma_2} \Rightarrow \hat{X}$ as $N \to \infty$ where the process $\hat{X}$ satisfies
\begin{align*}
\hat{X}(t) = \Pi_2x_0 + \sum_{k \in \Gamma_2} Y_k \left(  \int_{0}^t \int_{ \mathbb{L}_1 } \lambda_k( \hat{X}(s) +y ) V(dy \times ds)  \right) \Pi_2 \zeta_k.
\end{align*}
It can be seen that between consecutive jump times of the process $\Pi_2 X^{N}_{\gamma_2}$, if the state of the process $\Pi_2 X^{N}_{\gamma_2}$ is $v$, then the process $(I - \Pi_2)X^{N}_{\gamma_2}$ evolves like a Markov process with generator $\mathbb{C}^v$. If the generator $\mathbb{C}^v$ corresponds to an ergodic Markov process with the unique stationary distribution as $\pi^v \in \mathcal{P}( \mathbb{H}_v )$, then the limiting measure $V$ has the form
\begin{align}
\label{limitingoccmeasure}
V(dy \times ds) = \pi^{\hat{X}(s)}(dy)ds.
\end{align}
Therefore the random time change representation of the process $\hat{X}$ becomes
\begin{align}
\label{rtc_limit_process2}
\hat{X}(t) = \Pi_2x_0 + \sum_{k \in \Gamma_2} Y_k \left(  \int_{0}^t \hat{ \lambda }_k( \hat{X}(s) ) ds \right) \Pi_2 \zeta_k,
\end{align}
where $\hat{\lambda}_k(v) = \int_{ \mathbb{H}_v} \lambda_k( v+z ) \pi^v(dz)$. Before we state the convergence result, we need to make some assumptions.
\begin{assumption}
\label{assumptions2}
 \begin{itemize}
\item[(A)] For any $v = \Pi_2 \mathcal{S}$, the space $\mathbb{H}_v$ (given by \eqref{defn_ev}) is finite.
\item[(B)] The Markov process with generator $\mathbb{C}^v$ is ergodic and its unique stationary distribution is $\pi^v \in \mathcal{P}( \mathbb{H}_v )$.
\item[(C)] Let P be the set of reactions given by
\begin{align}
\label{defn_positivereaction2}
P = \{ k  = 1,\dots,K : \langle \bar{1}_d , \Pi_2 \zeta_k \rangle > 0\}.
\end{align}
Then the function $\lambda_P : \N^d_0 \to \R_+$ defined by $\lambda_P(x) = \sum_{k \in P}\lambda_k(x)$ is linearly growing with respect to projection $\Pi_2$ (see Definition \ref{polynomialgrowth}).
 \end{itemize}
\end{assumption}
Observe that part (C) implies that the functions $\{\hat{\lambda}_k : k \in \Gamma_2\}$ satisfy part (B) of Assumption \ref{assumptions1}. Therefore the process $\hat{X}$ satisfying \eqref{rtc_limit_process2} is well-defined due to Lemma \ref{lemma1:app}. Note that the set $\mathbb{H}_v$ can either be finite or countably infinite. 
Our main result (Theorem \ref{mainsensitivityresult}) should hold in both the cases, but to simplify the proof we assume that $\mathbb{H}_v$ is finite (part (A) of Assumption \ref{assumptions2}). We later discuss how the proof changes when this is not the case (see Remark \ref{rem:extension}). In many important  biochemical multiscale networks, the fast reactions conserve some quantity that only depends on the natural dynamics (see \cite{ssSSA,weinan1,PWOLDE}). In such a scenario, the set $\mathbb{H}_v$ will be finite. We now state the convergence result at the second time-scale.
\begin{proposition}
\label{convergenceresult2} 
Suppose that Assumption \ref{assumptions1} and \ref{assumptions2} hold.  Then $( \Pi_2 X^{N}_{\gamma_2} , V^{N}_{\gamma_2} ) \Rightarrow (\hat{X} , V )$ as $N \to \infty$, where the process $\hat{X}$ satisfies \eqref{rtc_limit_process2} and $V$ satisfies \eqref{limitingoccmeasure}.
\end{proposition}
%\begin{remark}
%\label{rem:conv2isspecialcaseofconv1}
%Note that Proposition \ref{convergenceresult1} is a special case of Proposition \ref{convergenceresult2}. Replacing $\gamma_2$ by $\gamma_1$, $\Gamma_2$ by $\Gamma_1$ and $\Pi_2$ by $I$ (identity map), we recover  Proposition \ref{convergenceresult1} from Proposition \ref{convergenceresult2}.
%\end{remark}
\begin{proof}
The proof follows from Theorem 5.1 in \cite{HWKang}. 
\end{proof}

\subsection{Convergence at higher time-scales} \label{sec:highertimescaleconv}

In Section \ref{sec:secondtimescaleconv} we outlined a systematic procedure to obtain a \emph{single-step} model reduction for a multiscale reaction network. The main idea was to assume ergodicity for the ``fast" sub-network and incorporate its steady-state information in the propensities of the ``natural" reactions. Moreover the ``slow" reactions can be ignored completely.
This single-step reduction process can be carried over multiple steps to construct a hierarchy of reduced models. 
This is useful because many biochemical networks have reactions spanning several time-scales (see \cite{Heat}, for example). 
Hence for a given reference time-scale, many steps of model reduction may be required to a obtain a model which is \emph{simple enough}, to be amenable for extensive simulations that are required for sensitivity estimation.

For our main result, we will assume that we are in the situation of Proposition \ref{convergenceresult2}, which describes a single-step model reduction. In Section \ref{subsec:ssparamsensmultreductions}, we shall discuss how our result can be used to estimate parameter sensitivity using reduced models that are obtained after many steps of model reduction.

\section{The Main Result} \label{sec:mainresult}
In this section we present our main result on sensitivity analysis of multiscale networks. Suppose that the propensity functions $\lambda_1,\dots,\lambda_K$ depend on a real-valued parameter $\theta$ and Assumption \ref{assumptions1} are satisfied for each value of $\theta$. If the reference time-scale is $\gamma$, then the reaction dynamics will be captured by the generator 
\begin{align}
\label{defn_gen_aNgamma_theta}
\mathbb{A}^N_{\gamma,\theta} f(x) = \sum_{k=1}^K  N^{ \beta_k+ \gamma} \lambda_k(x,\theta) ( f(x+\zeta_k) -f(x) ) \ \textnormal{ for any } f\in  \mathcal{D}( \mathbb{A}^N_{\gamma,\theta} ) = \mathcal{B}( \mathcal{S}).
\end{align} 
Using Lemma \ref{lemma1:app} we can argue that the martingale problem corresponding to $\mathbb{A}^N_{\gamma,\theta}$ is well-posed. Let $X^N_{\gamma,\theta}$ be the process with generator $\mathbb{A}^N_{\gamma,\theta}$ and initial state $x_0$. 

We use the same notation as in Section \ref{sec:secondtimescaleconv}. Note that the definitions of $\gamma_i,\Gamma_i,\mathbb{S}_i$ and $\mathbb{L}_i$, for $i=1$ and $2$, only depend on the stoichiometry of the reaction network and are hence independent of $\theta$. Similarly the projection map $\Pi_2$ and the space $\mathbb{H}_v$ (see \eqref{defn_ev}) do not depend on $\theta$. The definition of the operator $\mathbb{C}^v$ (see \eqref{defn_cv}) changes to 
\begin{align}
\label{defn_cvtheta}
\mathbb{C}^v_\theta f(z) = \sum_{k \in \Gamma_1} \lambda_k(v+z,\theta) \left( f(z + \zeta_k) -f(z)  \right)  \ \textnormal{ for } f \in  \mathcal{D}(\mathbb{C}^v_\theta ) = \mathcal{B}(\mathbb{H}_v ).
\end{align}
For our main result we require the following assumptions.
\begin{assumption}
\label{assforsensitivityresult}
\begin{itemize}
\item[(A)] Parts (A) and (C) of Assumption \ref{assumptions2} are satisfied. In addition, the mapping $v \mapsto | \mathbb{H}_v|$ is polynomially growing (see Definition \ref{polynomialgrowth}). 
\item[(B)] A Markov process with generator $\mathbb{C}^v_\theta$ is ergodic and its unique stationary distribution is $\pi^v_\theta \in \mathcal{P}( \mathbb{H}_v )$.
\item[(C)] Let $x \in \mathcal{S}$ be fixed. Then for any $k=1,\dots, K$, the function $\lambda_k(x,\cdot)$ is twice-continuously differentiable in a neighbourhood of $\theta$.
\item[(D)] For each $k \in \Gamma_2$, the functions $\lambda_k(\cdot,\theta)$ and $\partial \lambda_k (\cdot,\theta)/ \partial \theta$ are polynomially growing with respect to projection $\Pi_2$. 
Moreover there exists an $\epsilon >0$ such that the function
\begin{align*}
\sup_{ \xi \in (\theta -\epsilon , \theta + \epsilon ) } \left| \frac{ \partial^2 \lambda_k (\cdot,\xi)  }{  \partial \theta^2 } \right|
\end{align*}
is also polynomially growing with respect to projection $\Pi_2$.
\item[(E)] The functions $\{ \lambda_k(\cdot,\theta) : k \in \Gamma_2 \}$ satisfy part (B) of Assumption \ref{assumptions1}.
\end{itemize}
\end{assumption}

Note that if Assumption \ref{assforsensitivityresult} hold then Assumption \ref{assumptions2} will also hold. Hence Proposition \ref{convergenceresult2} ensures that $ \Pi_2 X^{N}_{\gamma_2,\theta} \Rightarrow \hat{X}_\theta$ as $N \to \infty$. The process $\hat{X}_\theta$ has initial state $\Pi_2 x_0$ and generator $\hat{\mathbb{A} }_\theta$ given by
\begin{align}
\label{finallimitgen}
\hat{\mathbb{A} }_\theta f(x) = \sum_{k \in \Gamma_2} \hat{\lambda}_k(x,\theta) ( f(x + \Pi_2 \zeta_k) -f(x) ) \ \textnormal{ for any } \ f \in  \mathcal{D}(\hat{\mathbb{A} }_\theta  ) =\mathcal{B}( \Pi_2 \mathcal{S}),
\end{align}
where the function $\hat{\lambda}_k(\cdot,\theta ) : \Pi_2 \mathcal{S} \to \R_+$ is defined by
\begin{align}
\label{defn_lambdahattheta}
\hat{\lambda}_k(x,\theta) = \int_{ \mathbb{H}_x }\lambda_k( x+ y,\theta) \pi^{x}_\theta(dy).
\end{align}
We now state our main result whose proof is given in Section \ref{sec:maintheoremproof}.
\begin{theorem}
\label{mainsensitivityresult}
Suppose that Assumption \ref{assforsensitivityresult} hold and the function $f : \mathcal{S} \to \R$ is polynomially growing with respect to projection $\Pi_2$.  Then for any $t > 0$ we have
\begin{align}
\label{mainresultequation}
\lim_{N \to \infty} \frac{\partial }{\partial \theta}\E \left(  f( X^{N}_{\gamma_2, \theta} (t) ) \right)= \frac{\partial }{\partial \theta}\E \left(  f_{\theta}( \hat{X}_\theta (t) ) \right) ,
\end{align}
where $f_\theta : \Pi_2 \mathcal{S} \to \R$ is given by
\begin{align}
\label{defn_fthetalimit}
f_\theta(x) = \int_{ \mathbb{H}_x }f( x+ y) \pi^{x}_\theta(dy).
\end{align}
\end{theorem}
\begin{remark}
\label{rem:mildextension}
This theorem will also hold if the function $f$ depends on the parameter $\theta$, as long as the dependence is continuously differentiable. This will be evident from the proof of the theorem.
\end{remark}

Recall that the reaction dynamics for the orginal model in the reference time-scale $\gamma_2$ is given by $X^{N_0}_{\gamma_2,\theta}$. If the output of interest is captured by function $f$, then we are interested in estimating the parameter sensitivity $S^{N_0}_{\gamma_2, t}(f,t)$ defined by \eqref{paramsensn0}.
As explained in Section \ref{sec:intro}, direct estimation of $S^{N_0}_{\gamma_2, t}(f,t)$ is often infeasible because simulations of the process $X^{N_0}_{\gamma_2,\theta}$ are prohibitively expensive. However simulations of the reduced model dynamics $\hat{X}_\theta$ is much cheaper, allowing us to easily estimate the right side of \eqref{mainresultequation}, using known methods \cite{IRN,KSR1,KSR2,DA,Gupta}. The main message of Theorem \ref{mainsensitivityresult} is that for large values of $N_0$
\begin{align}
\label{thm:mainmessage}
S^{N_0}_{\gamma_2, t}(f,t) \approx  \hat{S}_{\theta}(f_\theta,t) := \frac{\partial }{\partial \theta}\E \left(  f_{\theta}( \hat{X}_\theta (t) ) \right),
\end{align}
which allows us to approximately estimate $S^{N_0}_{\gamma_2, t}(f,t)$, in a computationally efficient way.

Observe that in \eqref{mainresultequation}, the function $f_\theta$ may depend on $\theta$ even if the function $f$ does not. If the stationary distribution $\pi^x_\theta$ is known for each $x \in \Pi_2 \mathcal{S}$, then the function $f_\theta$ and the propensities $\hat{\lambda}_k$ can be computed analytically. In this case, the simulations of the process $\hat{X}_\theta$ that are needed for estimating $\hat{S}_{\theta}(f_\theta,t)$, can be carried out using the slow-scale Stochastic Simulation Algorithm \cite{ssSSA}. If $\pi^x_\theta$ is unknown, then one can use nested schemes \cite{weinan1,weinan2} to estimate $f_\theta$ and $\hat{\lambda}_k$ during the simulation runs. In many applications, the ``fast" reactions are uninteresting \cite{PWOLDE,RaoArkin2,HR2} and they do not alter the output function $f$. In such a scenario we can expect $f$ to be invariant under the projection $\Pi_2$ (that is, $f(x) =f(\Pi_2 x)$ for all $x \in \mathcal{S}$) which would imply that the functions $f_\theta$ and $f$ are the same on the space $\Pi_2 \mathcal{S}$. Hence we recover \eqref{senisitivitylimitscommute} from Theorem \ref{mainsensitivityresult}.

\subsection{Estimation of steady-state parameter sensitivities}\label{subsec:ssparamsens}
We now discuss how relation \eqref{intro:sssenss} can be derived using our main result.
In Section \ref{sec:intro} we mentioned the importance of this relation in the context of estimating steady-state parameter sensitivities. Let $\{X_\theta(t) : t \geq 0\}$ be an ergodic $\mathcal{S}$-valued Markov process with generator 
\begin{align*}
\mathbb{C}_\theta f(x) = \sum_{k=1}^K \lambda_k(x,\theta) (f(x + \zeta_k)  - f(x)) \textnormal{ for any } f \in \mathcal{D}(\mathbb{C}_\theta )  =\mathcal{B}( \mathcal{S}),
\end{align*}
and stationary distribution $\pi_\theta$. If we define another process $X^N_\theta$ by
\begin{align}
\label{defn:fastprocessfromslow}
X^N_\theta(t) = X_\theta(N t)  \textnormal{ for } t \geq 0,
\end{align}
then $X^N_\theta$ represents the dynamics of a multiscale network with $\beta_k=1$ for each $k=1,\dots,K$. For this network, clearly $\gamma_2 = 0, \Gamma_2 = \emptyset$ and $\Pi_2 \mathcal{S} =\{0\}$. From Theorem \ref{mainsensitivityresult} we obtain
\begin{align*}
\lim_{N \to \infty} \frac{\partial }{\partial \theta} \E \left(  f( X^N_{\theta} (t) ) \right)=  \frac{d}{ d \theta } \left( \int_{\mathcal{S} } f(x) \pi_\theta(dx)   \right), 
\end{align*}
for any $t>0$. Hence \eqref{intro:sssenss} immediately follows from \eqref{defn:fastprocessfromslow}.

\subsection{Sensitivity estimation with multiple reduction steps} \label{subsec:ssparamsensmultreductions}

We have presented Theorem \ref{mainsensitivityresult} in the setting of Section \ref{sec:secondtimescaleconv}, where a \emph{single-step} reduction procedure was described to obtain a ``reduced" model ( with dynamics $\hat{X}_\theta$) from the original model (with dynamics $X^{N_0}_{\gamma,\theta}$), in the reference time-scale $\gamma = \gamma_2$. As mentioned in Section \ref{sec:highertimescaleconv}, there are examples of multiscale networks where many steps of model reduction may be required to arrive at a \emph{sufficiently} simple model. It is interesting to know that even in such cases, the main approximation relationship \eqref{thm:mainmessage} that falls out of Theorem \ref{mainsensitivityresult}, will continue to hold. To illustrate this point, we now consider an example where two-steps of model reduction are needed for sensitivity estimation.

Recall the description of a multiscale network from Section \ref{sec:intro}.
Let $\gamma_1,\gamma_2$ and $\gamma_3$ be real numbers such that $\gamma_3 > \gamma_2 > \gamma_1$. Suppose that the sets $\Gamma_1,\Gamma_2$ and $\Gamma_3$ form a partition of the reaction set $\{1,\dots,K\}$, and for each $k \in \Gamma_i$, we have $\beta_k = -\gamma_i$ for $i=1,2,3$. The dynamics of the model in the reference time-scale $\gamma$ is given by the process $X^{N_0}_{\gamma,\theta}$ whose random time change representation is 
\begin{align}
\label{mas:approx1}
X^{N_0}_{ \gamma,\theta}(t) & = \sum_{k \in \Gamma_1} Y_k \left(  N_0^{\gamma -\gamma_1 }  \int_{0}^t \lambda_k \left(   X^{N_0}_{ \gamma,\theta}(s)  ,\theta \right) ds \right) \zeta_k  
+ \sum_{k \in \Gamma_2} Y_k \left(  N_0^{\gamma - \gamma_2}  \int_{0}^t \lambda_k \left(   X^{N_0}_{ \gamma,\theta}(s)  ,\theta \right) ds \right) \zeta_k \\
&+ \sum_{k \in \Gamma_3} Y_k \left( N_0^{\gamma -\gamma_3 }   \int_{0}^t \lambda_k \left(   X^{N_0}_{ \gamma,\theta}(s)  ,\theta \right) ds \right) \zeta_k, \notag
\end{align}
where $\{Y_k : k =1,\dots,K\}$ is a family of independent unit rate Poisson processes. Clearly this multiscale network has three time-scales $\gamma_1,\gamma_2$ and $\gamma_3$. Suppose we want to estimate the sensitivity value $S^{N_0}_{\gamma, t}(f,t)$ (given by \eqref{paramsensn0}) at the reference time-scale $\gamma= \gamma_3$. Observe that for this time-scale, the reactions in both the sets $\Gamma_1$ and $\Gamma_2$ are ``fast", but the reactions in $\Gamma_1$ are ``faster" than those in $\Gamma_2$. Ideally we would like to estimate $S^{N_0}_{\gamma_3, t}(f,t)$ using a reduced model which only involves reactions in $\Gamma_3$. It is possible to obtain such a reduced model by applying the reduction procedure \emph{twice}. We now demonstrate that even with this \emph{second-order} reduced model, the main approximation relationship \eqref{thm:mainmessage} will still hold.

Replacing $N_0^{\gamma - \gamma_1}$ by $N_0^{\gamma - \gamma_2} N^{\gamma_2 - \gamma_1}$ in \eqref{mas:approx1}, we get another process $X^{N_0,N}_{ \gamma,\theta}$ defined by 
\begin{align}
\label{mas:approx2}
X^{N_0,N}_{ \gamma,\theta}(t) & = \sum_{k \in \Gamma_1} Y_k \left(N^{\gamma_2 - \gamma_1}  N_0^{\gamma - \gamma_2}  \int_{0}^t \lambda_k \left( X^{N_0,N}_{ \gamma,\theta}(s)  ,\theta \right)ds \right) \zeta_k  
+ \sum_{k \in \Gamma_2} Y_k \left(  N_0^{\gamma - \gamma_2}  \int_{0}^t \lambda_k \left(   X^{N_0,N}_{ \gamma,\theta}(s)  ,\theta \right)  ds \right) \zeta_k \notag \\
&+ \sum_{k \in \Gamma_3} Y_k \left( N_0^{\gamma -\gamma_3 }   \int_{0}^t \lambda_k \left(   X^{N_0,N}_{ \gamma,\theta}(s)  ,\theta \right)  ds \right) \zeta_k. 
\end{align}
Certainly for large values of $N_0$ we have
\begin{align}
\label{sens:msapp1}
S^{N_0}_{\gamma_3, t}(f,t) \approx \lim_{N \to \infty} \frac{ \partial  }{\partial \theta} \E \left(  f\left(  X^{N_0,N}_{ \gamma_3,\theta}(t) \right) \right).
\end{align} 
Observe that the process $X^{N_0,N}_{ \gamma_3,\theta}$ can be treated in the same way as the process $X^N_{\gamma_2,\theta}$ in Theorem \ref{mainsensitivityresult}. Suppose that the conditions of this theorem are satisfied. We can construct a projection $\Pi_2$ satisfying \eqref{fastreactionsdonotaffect} such that the process $\Pi_2 X^{N_0,N}_{ \gamma_3,\theta}$ has a well-behaved limit as $N \to \infty$. For any $v \in \Pi_2 \mathcal{S}$ let $\pi^v_\theta$ be the stationary distribution for the Markov process with generator $\mathbb{C}^v_\theta$ (see \eqref{defn_cvtheta}). Define $\bar{f}_\theta$ by \eqref{defn_fthetalimit} and for each $k \in \Gamma_2 \cup \Gamma_3$ let $\bar{\lambda}_k$ be given by \eqref{defn_lambdahattheta}. Using Theorem \ref{mainsensitivityresult} we can conclude that
\begin{align}
\label{sens:msapp2}
 \lim_{N \to \infty} \frac{ \partial  }{\partial \theta} \E \left(  f\left(  X^{N_0,N}_{ \gamma_3,\theta}(t) \right) \right) = \frac{ \partial  }{\partial \theta} \E \left(  \bar{f}_\theta \left(  \bar{X}^{N_0}_{ \gamma_3,\theta}(t) \right) \right),
\end{align} 
where $\bar{X}^{N_0}_{ \gamma_3,\theta}$ is the $\Pi_2 \mathcal{S}$-valued process given by
\begin{align*}
\bar{X}^{N_0}_{ \gamma_3,\theta}(t) & = \sum_{k \in \Gamma_2} Y_k \left(  N_0^{\gamma_3 - \gamma_2}  \int_{0}^t \bar{\lambda}_k \left(  \bar{ X }^{N_0}_{ \gamma_3,\theta}(s)  ,\theta \right)  ds \right) \Pi_2  \zeta_k + \sum_{k \in \Gamma_3} Y_k \left(    \int_{0}^t \bar{\lambda}_k \left(   \bar{ X }^{N_0}_{ \gamma_3,\theta}(s)  ,\theta \right)  ds \right) \Pi_2  \zeta_k. 
\end{align*}
Substituting $N_0$ by $N$ we get another process $\bar{X}^{N}_{ \gamma_3,\theta}$ which can again be dealt in the same way as the process $X^N_{\gamma_2,\theta}$ in Theorem \ref{mainsensitivityresult}. Moreover for large values of $N_0$,
\begin{align}
\label{sens:msapp5}
 \frac{ \partial  }{\partial \theta} \E \left(  \bar{f}_\theta \left(  \bar{X}^{N_0}_{ \gamma_3,\theta}(t) \right) \right) \approx \lim_{N \to \infty} \frac{ \partial  }{\partial \theta} \E \left(  \bar{f}_\theta \left(  \bar{X}^{N}_{ \gamma_3,\theta}(t) \right) \right).
\end{align} 
Assuming that the conditions of Theorem \ref{mainsensitivityresult} hold, we can construct a projection $\Pi_3$, such that $\Pi_3 \Pi_2 \zeta_k = \bar{0}_d$ for all $k \in \Gamma_2$, and the process $\Pi_3\bar{X}^{N}_{ \gamma_3,\theta}$ has a well-behaved limit as $N \to \infty$. For any $w \in \Pi_3 \Pi_2 \mathcal{S}$, let $\mu^w_\theta$ be the stationary distribution for the Markov process with generator
\begin{align*}
\mathbb{C}^w_\theta g(z) = \sum_{k \in \Gamma_2} \bar{ \lambda}_k(w+z,\theta) \left( g(w + \Pi_2\zeta_k) -g(z)  \right)  \ \textnormal{ for } g \in  \mathcal{D}(\mathbb{C}^w_\theta ) = \mathcal{B}(\mathbb{H}_w ),
\end{align*}
where the definition of $\mathbb{H}_w$ is similar to \eqref{defn_ev}. Define
\begin{align*}
\hat{f}_{\theta}(w) = \int_{\mathbb{H}_w} \bar{f}_\theta(w+y) \mu^{w}_{\theta}(dy) \ \textnormal{ and } \ \hat{\lambda}_{k} ( w  ,\theta ) =  \int_{\mathbb{H}_w} \bar{\lambda}_k(w+y,\theta )  \mu^{w}_{\theta}(dy),
\end{align*}
for each $k \in \Gamma_3$. From Theorem \ref{mainsensitivityresult} we get
\begin{align}
\label{sens:msapp3}
 \lim_{N \to \infty}  \frac{ \partial  }{\partial \theta} \E \left(  \bar{f}_\theta \left(  \bar{X}^{N}_{ \gamma_3,\theta}(t) \right) \right) =  \frac{ \partial  }{\partial \theta} \E \left(  \hat{f}_\theta \left(  \hat{X}_{\theta}(t) \right) \right),
\end{align}   
where $\hat{X}_\theta$ is the process given by
\begin{align*}
\hat{X}_{ \theta}(t) & = \sum_{k \in \Gamma_3} Y_k \left(    \int_{0}^t \hat{\lambda}_k \left( \hat{X}_{ \theta}(s) ,\theta \right)  ds \right) \Pi_3 \Pi_2  \zeta_k. 
\end{align*}
Combining \eqref{sens:msapp1}, \eqref{sens:msapp2}, \eqref{sens:msapp5} and \eqref{sens:msapp3}, we get that for large values of $N_0$
\begin{align}
\label{sens:msapp3:final}
  S^{N_0}_{\gamma_3, t}(f,t) \approx \frac{ \partial  }{\partial \theta} \E \left(  \hat{f}_\theta \left(  \hat{X}_{\theta}(t) \right) \right).
\end{align}
This shows that the main approximation relationship (\eqref{thm:mainmessage}) that arises from Theorem \ref{mainsensitivityresult} will hold even with a reduced model obtained after two steps of model reduction.
Observe that the reactions in $\Gamma_3$ are ``natural" for the time-scale $\gamma_3$, and the reduced model corresponding to $\hat{X}_{\theta}$ only consists of these reactions. Hence the process $\hat{X}_{\theta}$ is easy to simulate and $ S^{N_0}_{\gamma_3, t}(f,t) $ can be easily estimated using \eqref{sens:msapp3:final}.   

\section{Proofs}\label{sec:proof}

We mentioned in Section \ref{sec:intro} that the proof of our main result, Theorem \ref{mainsensitivityresult}, will require many steps. We now describe these steps in detail. In Section \ref{sec:fmcs} we show some regularity properties of the distributions of weighted occupation times for finite Markov chains with \emph{fast} parameter-dependent rates. For this, we exploit certain connections between the distribution of weighted occupation times and multi-dimensional wave equations (see \cite{Bruno}).
These regularity properties allows us to later argue that the distribution of the weighted occupation times for the ``fast" sub-network of our multiscale network, is differentiable with respect to $\theta$, and the derivative operation commutes with the limt $N \to \infty$.
In Section \ref{sec:constrnewprocess}, we construct a ``new" process $W^N_\theta$, which captures the one-dimensional distribution of the process $X^N_{\gamma_2,\theta}$, in the sense described in Section \ref{sec:intro}. The main difference between $X^N_{\gamma_2,\theta}$ and $W^N_\theta$, is that the dynamics of the fast sub-network is \emph{averaged} out in the process $W^N_\theta$, making it easier to work with. In particular the process $W^N_\theta$ is well-behaved limit as $N \to \infty$ (see Proposition \ref{prop_limitresult}), unlike the process $X^N_{\gamma_2,\theta}$.  
The proof of Theorem \ref{mainsensitivityresult} is given in Section \ref{sec:maintheoremproof}. The main idea of the proof is to couple the processes $W^N_\theta$ and $W^N_{\theta+h}$, in such a way, that it allows us to compute a double-limit of the form
\begin{align*}
\lim_{N \to \infty} \lim_{h\to 0}\frac{ \E\left( f^N_{\theta +h} ( W^N_{\theta+h} (t) )  \right) - \E\left( f^N_{\theta} ( W^N_{\theta} (t) )  \right)  }{ h},
\end{align*}
for some functions $f^N_{\theta } $ and $f^N_{\theta+h} $ that depend on our output function $f$. The results from Section \ref{sec:constrnewprocess} will imply that this quantity is equal to the left-hand side of \eqref{mainresultequation}. On the other hand, using Dynkin's formula (see Lemma 19.21 in \cite{Kal}) and some coupling arguments, we will show that this quantity is also equal to the right-side of \eqref{mainresultequation}, thereby proving Theorem \ref{mainsensitivityresult}.

\subsection{Weighted occupation times of finite Markov chains}\label{sec:fmcs}

Let $\{ Z(t) : t \geq 0\}$ be a continuous time Markov chain on a finite state space $\mathcal{E} =\{e_1,\dots,e_m\}$ and with generator 
\begin{align*}
\mathbb{A} f(z) = \sum_{k =1 }^{K}  \lambda_k(z)  \left( f(z+ \zeta_k)  -f(z)\right) \ \textnormal{ for all } f \in \mathcal{D}(\mathbb{A}) = \mathcal{B}(\mathcal{E}).
\end{align*}
Here $\lambda_1,\dots, \lambda_K$ are positive functions on $\mathcal{E} $.
For this Markov chain the $Q$-matrix (matrix of transition rates) is given by 
\begin{align*}
Q_{ij} = \left\{
\begin{tabular}{cc}
$\lambda_k(e_i)$ & if $i \neq j$ and $e_j = e_i + \zeta_k$  \\  
$-\sum_{k = 1}^K \lambda_k(e_i)$ & if $i = j$ \\
$0$ & otherwise.
\end{tabular} \right.
\end{align*}

For a function $\Lambda : \mathcal{E}\to [0,\infty)$ define
\begin{align}
\label{defvt}
V(t) = \int_{0}^{t} \Lambda(Z(s))ds,
\end{align}
then $V(t)$ is essentially the weighted occupation time of the process $Z$, where the weight is given by the function $\Lambda$.
For each $i = 1,\dots,m$ define $p_i, \beta_i : \R_+  \to [0,1]$ by
\begin{align*}
\beta_i(t) = \E\left( \ind_{ \{ Z(t) = e_i  \}}   \exp( -V(t) ) \right) \ \textnormal{ and  } \ p_i(t)= \P(Z(t) =e_i ).
\end{align*}
Note that $\beta_i(t) $ can be seen as the \emph{Laplace Transform} of the distribution of $V(t)$ on the event $Z(t) =e_i$.
Let $p(t)$ and $\beta(t)$ denote the vectors
\begin{align*}
p(t) = (p_1(t),\dots, p_m(t))  \  \textnormal{ and } \ \beta(t) = (\beta_1(t),\dots, \beta_m(t)).
\end{align*}
The definition of matrix $Q$ implies that
\begin{align}
\label{prob_evolution}
\frac{d p (t) }{dt } =  Q^T p(t).
\end{align}
The next proposition describes the dynamics of $\beta$.
\begin{proposition}
\label{mainoccmeasureprop}
The function $\beta$ satisfies the following ordinary differential equation
\begin{align*}
\frac{d \beta (t) }{dt } = \left( Q^T -  D \right)\beta(t),
\end{align*}
where $D$ is the $m  \times m$ diagonal matrix with entries $\Lambda(e_1),\dots,\Lambda(e_m)$.
\end{proposition}
\begin{proof}
Let $r_1,\dots,r_l$ be $l$ distinct values in the set $\{  \Lambda(e_1) ,\dots, \Lambda(e_m)\}$, arranged in the ascending order. For each $i=1,\dots,(l-1)$ let $B_i = \{e \in \mathcal{E} : \Lambda( e ) = r_i\}$.
For each $i=1,\dots,m$ define $F_i : \R_+ \times \R  \to [0,1]$ by
\begin{align*}
F_i(t,x) = \P\left(  Z(t) = e_i , V(t) > x   \right).
\end{align*}
The random variable $V(t)$ (given by \eqref{defvt}) can only take values between $r_1t$ and $r_l t$. Hence
\begin{align}
\label{icfordistr}
F_i(t, r_l t)=0 \ \textnormal{ and } \  F_i(t,r_1 t -) = \lim_{h \to 0^{-} } F_i(t,r_1 t +h) = \P(  Z(t) = e_i ).
\end{align}
It has been shown in \cite{Bruno} that the distribution of the real-valued random variable $V(t)$ is continuous in the interval $[r_1t ,r_l t]$, except at 
points $r_1t ,\dots, r_l t$. Whenever $x = r_j t$ for some $j =1,\dots,l$, the function $F_i$ has a discontinuity of size
\begin{align*}
F_i(x,r_j t ) - F_i(x,r_j t -) = -\P\left( Z(t) = e_i , V(t) = r_j  t \right).
\end{align*}
Moreover, the event $\{ V(t) = r_j  t \}$ can only happen if $Z(s) \in B_j$ for all $s \in [0,t]$. Therefore $\P\left( Z(t) = e_i , V(t) = r_j  t \right)$ is non-zero only if $e_i \in B_j$ and hence
\begin{align}
\label{sumofdiscont.}
  \sum_{j=1}^l  g(r_j)(  F_i( t ,r_j t- ) - F_i( t ,r_j t) ) \notag &= \sum_{j=1}^l  g(r_j) \P\left( Z(t) = e_i , V(t) = r_j  t \right)  \notag \\
 & =  g( \Lambda(e_i) ) \P\left( Z(t) = e_i , V(t) = \Lambda(e_i) t \right),
\end{align}
for any $g: \R_+ \to  \R_+$. It is shown in \cite{Bruno} that on the set $\mathcal{R} = \{(t,x) : t > 0 \textnormal{ and } x \in (r_{j-1} t , r_j t ) , \  j =2,\dots,l \}$, each $F_i$ is continuously differentiable and the family of functions $\{F_i : i =1,\dots,m\}$ satisfies the following system of multi-dimensional wave equations
\begin{align}
\label{systemofpdes}
\frac{ \partial F_i(t,x) }{ \partial t } = -\Lambda(e_i) \frac{ \partial F_i(t,x) }{ \partial x } + \sum_{k =1}^{m} F_k(t,x)Q_{ki}, \ \textnormal{  for  } i =1,\dots,m.
\end{align}

For each $i=1,\dots,m$ we can write $\beta_i(t)$ as
\begin{align}
\label{formulaforbeta1}
\beta_i(t) & = \E\left( \ind_{ \{ Z(t) =e_i  \} }  e^{-V(t)} \right) = e^{-\Lambda(e_i) t} \P\left( Z(t) = e_i , V(t) = \Lambda(e_i)  t \right)  - \sum_{j=2}^l \int_{r_{j-1} t}^{r_j t} e^{-x} \left(  \frac{ \partial F_i(t,x)  }{ \partial x }  \right) dx.
\end{align}
Using integration by parts, \eqref{icfordistr} and \eqref{sumofdiscont.} we get
\begin{align*}
\sum_{j=2}^l \int_{r_{j-1} t}^{r_j t} e^{-x} \left(  \frac{ \partial F_i(t,x)  }{ \partial x }  \right) dx & = \sum_{j=2}^l \left(  e^{  -r_j t} F_i( t ,r_j t- ) -  
e^{  -r_{j-1} t} F_i( t ,r_{j-1} t  )  \right)
+  \sum_{j=2}^l \int_{r_{j-1} t }^{r_j t} e^{-x} F_i(t,x)dx \\
& =\sum_{j=2}^l \left(  e^{  -r_j t} F_i( t ,r_j t ) -  e^{  -r_{j-1} t} F_i( t ,r_{j-1} t  )  \right)
+  \sum_{j=2}^l \int_{r_{j-1} t }^{r_j t} e^{-x} F_i(t,x)dx\\
& + \sum_{j=2}^l  e^{  -r_j t} (  F_i( t ,r_j t- ) - F_i( t ,r_j t) ) \\
& =-  e^{  -r_{1} t} F_i( t ,r_{1} t  )
+  \sum_{j=2}^l \int_{r_{j-1} t }^{r_j t} e^{-x} F_i(t,x)dx + \sum_{j=2}^l  e^{  -r_j t} (  F_i( t ,r_j t- ) - F_i( t ,r_j t) ) \\
& =   e^{  -r_{1} t} \left(  F_i( t ,r_{1} t - )  - F_i( t ,r_{1} t  )  \right) -  e^{  -r_{1} t} F_i( t ,r_{1} t - )  \\&
+  \sum_{j=2}^m \int_{r_{j-1} t }^{r_j t} e^{-x} F_i(t,x)dx  + \sum_{j=2}^l  e^{  -r_j t} (  F_i( t ,r_j t- ) - F_i( t ,r_j t) )  \\
& =  -  e^{  -r_{1} t} \P(  Z(t) = e_i )+  \sum_{j=2}^l \int_{r_{j-1} t }^{r_j t} e^{-x} F_i(t,x)dx  \\
&+ e^{-\Lambda(e_i) t} \P\left( Z(t) = e_i , V(t) = \Lambda(s_i)  t \right) .
\end{align*}
Substituing the above expression in \eqref{formulaforbeta1} we obtain
\begin{align}
\label{mainformulaforbeta}
\beta_i(t) =  e^{  -r_{1} t} p_i(t) - \sum_{j=2}^l \int_{r_{j-1} t }^{r_j t} e^{-x} F_i(t,x)dx ,
\end{align} 
 where $p_i(t) = \P( Z(t) = e_i )$.
 
 For $i=1,\dots,m$, the functions $p_i$ and $F_i(\cdot,x)$ are differentiable (see \eqref{prob_evolution} and \eqref{systemofpdes}). Hence the function $\beta_i$ is also differentiable. Taking derivative with respect to $t$ in \eqref{mainformulaforbeta} yields
\begin{align*}
\frac{d \beta_i(t) }{ d t } & = -\sum_{j=2}^l \int_{r_{j-1} t }^{r_j t} e^{-x} \frac{ \partial F_i(t,x) }{ \partial t} dx 
- \sum_{j=2}^l \left(  r_j e^{-r_j t} F_i(t,r_j t - ) - r_{j-1} e^{-r_{j-1 } t} F_i(t,r_{j-1} t  )  \right) \\
& - r_1  e^{  -r_{1} t} p_i(t) + e^{  -r_{1} t} \frac{d p_i(t) }{dt}.
\end{align*}
From \eqref{icfordistr} and \eqref{sumofdiscont.} it follows that
\begin{align}
\label{simplformula1}
& \sum_{j=2}^l \left(  r_j e^{-r_j t} F_i(t,r_j t - ) - r_{j-1} e^{-r_{j-1 } t} F_i(t,r_{j-1} t  )  \right) \notag \\
 &= \sum_{j=2}^l \left(  r_j e^{-r_j t} F_i(t,r_j t  ) - r_{j-1} e^{-r_{j-1 } t} F_i(t,r_{j-1} t )  \right) \notag  \\
& +\sum_{j=2}^l  r_j e^{-r_j t} \left( F_i(t,r_j t - ) - F_i(t,r_j t  ) \right) \notag \\
& =- r_1 e^{-r_1 t} F_i(t,r_1 t )  +\sum_{j=2}^l  r_j e^{-r_j t}\P\left( Z(t) = e_i , V(t) = r_j  t \right) \notag  \\
& = - r_1 e^{-r_1 t} p_i(t) + \sum_{j=1}^l  r_j e^{-r_j t}\P\left( Z(t) = e_i , V(t) = r_j  t \right)  \notag \\
& = - r_1 e^{-r_1 t} p_i(t) + \Lambda(e_i) e^{- \Lambda(e_i) t}\P\left( Z(t) = e_i , V(t) = \Lambda(e_i) t \right).
\end{align}
Therefore
\begin{align*}
\frac{d \beta_i(t) }{ d t } & =-\sum_{j=2}^l \int_{r_{j-1} t }^{r_j t} e^{-x} \frac{ \partial F_i(t,x) }{ \partial t} dx 
-  \Lambda(e_i)e^{-\Lambda(e_i)t}\P\left( Z(t) = e_i , V(t) =\Lambda(e_i) t \right)  + e^{  -r_{1} t} \frac{d p_i(t) }{dt}.
\end{align*}
From \eqref{systemofpdes} we get
\begin{align*}
\frac{d \beta_i(t) }{ d t } & =\Lambda(e_i)  \sum_{j=2}^l \int_{r_{j-1} t }^{r_j t} e^{-x} \frac{ \partial F_i(t,x) }{ \partial x} dx
- \sum_{k=1}^m  \left( \sum_{j=2}^l \int_{r_{j-1} t }^{r_j t} e^{-x} F_k(t,x)dx \right) Q_{ki} \\
  &-  \Lambda(e_i)  e^{-\Lambda(e_i)  t}\P\left( Z(t) = e_i , V(t) = \Lambda(e_i)   t \right)  + e^{  -r_{1} t} \frac{d p_i(t) }{dt}  \\
  & = \Lambda(e_i)  \sum_{j=2}^l \int_{r_{j-1} t }^{r_j t} e^{-x} \frac{ \partial F_i(t,x) }{ \partial x} dx
+ \sum_{k=1}^m  \left( \beta_k(t) - e^{-r_1 t} p_k(t) \right) Q_{ki} \\
  &-  \Lambda(e_i) e^{-\Lambda(e_i)  t}\P\left( Z(t) = e_i , V(t) = r_i  t \right)  + e^{  -r_{1} t} \frac{d p_i(t)}{dt} \\
  & =  \Lambda(e_i)  \sum_{j=2}^l \int_{r_{j-1} t }^{r_j t} e^{-x} \frac{ \partial F_i(t,x) }{ \partial x} dx
+ \sum_{k=1}^m  \beta_k(t)  Q_{ki}   -  \Lambda(e_i) e^{-\Lambda(e_i) t}\P\left( Z(t) = e_i , V(t) = r_i  t \right) \\& + e^{  -r_{1} t} \left( \frac{d p_i(t)}{dt}  -  \sum_{k=1}^m  p_k(t) Q_{ki} \right).
\end{align*}
Due to \eqref{prob_evolution}, the last term is $0$ and hence
\begin{align}
\label{righteqnforbetaderivative}
\frac{d \beta_i(t) }{ d t }  =\Lambda(e_i)  \sum_{j=2}^l \int_{r_{j-1} t }^{r_j t} e^{-x} \frac{ \partial F_i(t,x) }{ \partial x} dx
+ \sum_{k=1}^m  \beta_k(t)  Q_{ki}   -  \Lambda(e_i)  e^{-\Lambda(e_i)   t}\P\left( Z(t) = e_i , V(t) =\Lambda(e_i)  t \right).
\end{align}
 Using integration by parts, \eqref{simplformula1} and \eqref{mainformulaforbeta} we obtain
\begin{align*}
\sum_{j=2}^l \int_{r_{j-1} t }^{r_j t} e^{-x} \frac{ \partial F_i(t,x) }{ \partial x} dx
& = \sum_{j=2}^l  \left( e^{-r_j t} F_i(t,r_j t-) -  e^{-r_{j-1} t} F_i(t,r_{j-1} t)    \right) +  \sum_{j=2}^l \int_{r_{j-1} t }^{r_j t} e^{-x} F_i(t,x) dx \\
& =  e^{  -\Lambda(e_i)  t} \P\left( Z(t) = e_i , V(t) = \Lambda(e_i)   t \right) - \beta_i(t).
\end{align*}
Substituting this expression in \eqref{righteqnforbetaderivative} yields
\begin{align*}
\frac{d \beta_i(t) }{ d t } = - \Lambda(e_i) \beta_i(t) + \sum_{k=1}^m  \beta_k(t)  Q_{ki}.  
\end{align*}
This completes the proof of the proposition.
\end{proof}

Using the above proposition, we now establish some regularity properties of the distributions of weighted occupation times for finite Markov chains with \emph{fast} parameter-dependent rates.
Let $\{ Z^N_\theta (t) : t \geq 0 \}$ be a continuous time Markov chain on $\mathcal{E} =\{e_1,\dots,e_m\}$ with generator given by 
\begin{align*}
\mathbb{C}^N_\theta f(z) = N \sum_{k =1 }^{K}  \lambda_k(z,\theta)  \left( f(z+ \zeta_k)  -f(z)\right) \ \textnormal{ for all } f \in \mathcal{D}( \mathbb{C}^N_\theta  ) = \mathcal{B}( \mathcal{E}),
\end{align*}
where the function $\theta \mapsto \lambda_k(z,\theta)$ is continuously differentiable for each $k$ and $z\in \mathcal{E}$. For this Markov chain, the matrix of transition rates is given by $N Q_\theta$ where  
\begin{align*}
Q_{\theta, ij} = \left\{
\begin{tabular}{cc}
$\lambda_k(e_i,\theta)$ & if $i \neq j$ and $e_j = e_i + \zeta_k$  \\  
$-\sum_{k = 1}^K \lambda_k(e_i,\theta )$ & if $i = j$ \\
$0$ & otherwise.
\end{tabular} \right.
\end{align*}
We assume that this Markov chain is ergodic. Then its unique stationary distribution $\pi_\theta$ is a left eigenvector for $Q_\theta$ corresponding to the eigenvalue $0$. Hence
\begin{align}
\label{condonstdistribution}
\pi_\theta Q_\theta = \bar{0}_m \ \textnormal{ and } \ \langle \bar{1}_m , \pi_\theta \rangle =\bar{1}^T_m \pi_\theta = 1.
\end{align}
\begin{remark}
\label{ergodicity}
Due to the ergodicity assumption, the matrix $Q_\theta$ has $0$ as a simple eigenvalue and all its other eigenvalues have strictly negative real parts.
\end{remark}

For a function $\Lambda : \mathcal{E} \times \R \to  [0,\infty)$ define
\begin{align}
\label{defnvt}
V^N_\theta(t) = \int_{0}^{t} \Lambda(Z^N_\theta(s) , \theta  )ds
\end{align}
and let
\begin{align*}
\beta^{N}_{\theta , i}(t) = \E\left( \ind_{ \{ Z^N_\theta(t) = e_i  \}}   \exp( -V^N_\theta(t) ) \right), 
\end{align*}
for each $i=1,\dots,m$. From Proposition \ref{mainoccmeasureprop} it follows that the function $\beta^N_\theta(t) = ( \beta^N_{\theta,1} (t),\dots,\beta^N_{\theta,m} (t) )$ satisfies
\begin{align}
\label{mainodeforbeta}
\frac{d \beta^N_\theta (t) }{dt } = \left( N Q_\theta^T -  D_\theta \right)\beta^N_\theta(t),
\end{align}
where $D_\theta$ is the $m  \times m$ diagonal matrix with entries $\Lambda(e_1,\theta),\dots,\Lambda(e_m,\theta)$. We now define a condition on sequences of functions on $\R_+$.
\begin{condition}
\label{regularitycondition}
For each $N \in \N$, let $f^N$ be a function from $\R_+$ to $\R^m$ and let $\epsilon_N = 1/ \sqrt{N}$. Then the sequence of functions $\{f^N : N \in \N \}$ satisfies this condition if for any $T > 0$
\begin{align*}
\lim_{N \to \infty} \sup_{t \in [\epsilon_N,T ]} \| f^N(t) \| = 0 \ \textnormal{ and } \ \lim_{N \to \infty} \int_{0}^T\| f^N(t) \|dt = 0.
\end{align*}
\end{condition}
The main result of this section is given as the next proposition.
\begin{proposition}
\label{prop_regularity}
Define $\hat{\beta}^N_\theta : [0,\infty) \to \R^m$ by
\begin{align*}
\hat{\beta}^N_{ \theta } (t) = \beta^N_\theta(t) -e^{ -\lambda_\theta t}  \pi_\theta,
\end{align*}
where
\begin{align}
\label{perturbedeigenvalue}
\lambda_\theta = \bar{1}_m^T D_\theta \pi_\theta.
\end{align}
Then the functions $\hat{\beta}^N_{ \theta }$ and $\partial \hat{\beta}^N_{ \theta } /\partial \theta$ satisfy Condition \ref{regularitycondition}.
\end{proposition}
\begin{remark}
Here $\partial \beta^N_\theta  / \partial \theta$ should be interpreted as the map $t \mapsto \partial \beta^N_\theta(t)  / \partial \theta$.
Of course this proposition can only be true if $\partial \beta^N_\theta(t)  / \partial \theta$ and $\partial \pi_\theta / \partial \theta$ exist. Note that entries of the matrices $Q_\theta$ and $D_\theta$ are differentiable in $\theta$. Hence \eqref{mainodeforbeta} implies the existence of $\partial \beta^N_\theta(t)  / \partial \theta$. Moreover due to the implicit mapping theorem and the relation $\pi_\theta Q_\theta = \bar{0}_d$ (see \eqref{condonstdistribution}) one can also conclude that $\partial \pi_\theta / \partial \theta$ exists.  
\end{remark}
\begin{proof}
We start by defining some notation that will be useful in the proof. We say that a $\R^m$-valued sequence $\{a_N : N \in \N\}$ belongs to class $O(N^{-m})$ for some $m \in \N_0$, if and only if
\begin{align*}
\sup_{N \in \N} N^m \|a_N\| < \infty.
\end{align*}
For two such sequences $\{a_N : N \in \N\}$ and $\{b_N : N \in \N\}$, we will say that $a_N = b_N + O(N^{-m})$ when the sequence $\{(a_N -b_N) : N \in \N\}$ belongs to class $O(N^{-m})$.

For the proof, we can assume without loss of generality, that for each $N$, $Z^N_\theta(0) = e_{i_0}$ for some $i_0 = 1,\dots,m$. This implies that $\beta^N_\theta(0) = (0,\dots,0,1,0,\dots,0)$, where the $1$ is in place $i_0$. Hence
\begin{align}
\label{initialregularity}
\langle \bar{1}_m , \beta^N_\theta(0) - \pi_\theta  \rangle = \langle \bar{1}_m, \beta^N_\theta(0)  \rangle - \langle \bar{1}_m, \pi_\theta \rangle = 0. 
\end{align}
Define a function $h^N_\theta : \R_+ \to \R^m$ by
\begin{align}
\label{defnhntheta}
h^N_\theta(t) = e^{\lambda_\theta t}  \beta^N_\theta(t) - \pi_\theta.
\end{align}
To prove the proposition it is sufficient to show that both $h^N_\theta$ and $\partial h^N_\theta / \partial \theta$ satisfy Condition \ref{regularitycondition}.

From \eqref{mainodeforbeta} we obtain
\begin{align}
\label{main_eqn_hn}
\frac{d h^N_\theta(t) }{dt} = \left( N Q^T_\theta  - D_\theta +  \lambda_\theta I_m\right) h^N_\theta(t) - D_\theta \pi_\theta + \lambda_\theta \pi_\theta, 
\end{align}
where $I_m$ is the $m \times m$ identity matrix. Consider the matrix $B^N_\theta =  Q^T_\theta - N^{-1}D_\theta$, which can be seen as a small perturbation of $Q^T_\theta$ for large values of $N$. The eigenvalues of $B^N_\theta$ is slighly perturbed with respect to the eigenvalues of $Q^T_\theta$ (see \cite{PerTheory}). We know that matrix $Q_\theta^T$ has $0$ as a simple eigenvalue (see Remark \ref{ergodicity}) and the corresponding left eigenvector is $\bar{1}_m$. From Theorem 2.7 in \cite{PerTheory}, we can conclude that $B^N_\theta$ has an eigenvalue at $\lambda^N_\theta$ with the corresponding left eigenvector at $v^N_\theta$, where $\lambda^N_\theta$ and $v^N_\theta$ have the form
\begin{align}
\label{perturbedeigenvalueeigenvector}
\lambda^N_\theta = - \frac{ \lambda_\theta }{N} + O(N^{-2}) \ \textnormal{ and } \ v^N_\theta = \bar{1}_m + O(N^{-1}).
\end{align}
Therefore 
\begin{align*}
(v^N_\theta)^T \left( N Q^T_\theta  - D_\theta +  \lambda_\theta I_m\right) & = N (v^N_\theta)^T B^N_\theta + \lambda_\theta (v^N_\theta)^T   = N \lambda^N_\theta   (v^N_\theta)^T +  \lambda_\theta (v^N_\theta)^T   = \left( N \lambda^N_\theta +  \lambda_\theta   \right) (v^N_\theta)^T.
\end{align*}
Let $S^N_\theta = \langle v^N_\theta , h^N_\theta(t)  \rangle$. Taking inner product with $v^N_\theta$ in \eqref{main_eqn_hn} we get
\begin{align*}
\frac{d S^N_\theta(t) }{dt} =\left( N \lambda^N_\theta +  \lambda_\theta   \right) S^N_\theta(t) + (v^N_\theta)^T \left( - D_\theta \pi_\theta + \lambda_\theta \pi_\theta\right).
\end{align*}
Note that $a^N_\theta := N \lambda^N_\theta +  \lambda_\theta  = O(N^{-1})$ due to \eqref{perturbedeigenvalueeigenvector}. From \eqref{perturbedeigenvalue} and \eqref{condonstdistribution} we can see that
$b^N_\theta := (v^N_\theta)^T \left( - D_\theta \pi_\theta + \lambda_\theta \pi_\theta\right) = \bar{1}_m^T \left( - D_\theta \pi_\theta + \lambda_\theta \pi_\theta\right) + O(N^{-1}) = O(N^{-1})$. Therefore we can write
\begin{align}
\label{sumeqn}
\frac{d S^N_\theta(t) }{dt} = a^N_\theta S^N_\theta(t) + b^N_\theta,
\end{align}
where $\{a^N_\theta\}, \{b^N_\theta\}$ are sequences in $O(N^{-1})$. Using \eqref{initialregularity} we obtain
\begin{align}
\label{sumic}
S^N_\theta(0) = \langle  v^N_\theta, h^N_\theta(0) \rangle =  \langle \bar{1}_m, h^N_\theta(0) \rangle +O(N^{-1}) =  
\langle  \bar{1}_m , \beta^N_\theta(0) \rangle  - \langle  \bar{1}_m, \pi_\theta \rangle +O(N^{-1}) =  O(N^{-1}) .
\end{align}
Pick any $T>0$. From \eqref{sumeqn}, \eqref{sumic} and Gronwall's inequality it follows that
\begin{align*}
\sup_{t \in [0,T]} |S^N_\theta(t) | = \sup_{t \in [0,T]} | \langle v^N_\theta(t) , h^N_\theta(t) \rangle | = O(N^{-1}),
\end{align*}
which also implies that
\begin{align}
\label{sumisorder1}
 \sup_{t \in [0,T]} | \langle \bar{1}_m , h^N_\theta(t) \rangle | = O(N^{-1}).
\end{align}
This allows us to write
\begin{align*}
h^N_{\theta, m}(t) = - \sum_{i=1}^{m-1}  h^N_{ \theta, i }(t) + O(N^{-1}). 
\end{align*}
Let $C_\theta$ be the $(m-1) \times (m-1)$ matrix whose $ij$-th entry is given by
\begin{align*}
C_{\theta, ij} = Q_{\theta,ji} -  Q_{\theta,mi}.
\end{align*}
 If we define
\begin{align*}
P= \left[
\begin{array}{cc}
I_{m-1} & \bar{1}_{m-1} \\
\bar{0}^T_{m-1} & 1 \\
\end{array}
\right] 
\hspace{5pt} \textrm{ and } \hspace{5pt} 
P^{-1}= \left[
\begin{array}{cc}
I_{m-1} & - \bar{1}_{m-1}  \\
\bar{0}^T_{m-1} & 1 \\
\end{array}
\right]
\end{align*}
then using $\bar{1}_m^T Q^T_\theta = \bar{0}^T_m$ we can write 
\begin{align}
\label{relationofgandgbar}
P^T Q^T_\theta  (P^T)^{-1}= \left[
\begin{array}{cc}
C_\theta & v \\
\bar{0}^T_{m-1} & 0 \\
\end{array}
\right], 
\end{align}
where $v$ is some vector in $\R^{m-1}$. The matrix $Q_\theta$ has a simple eigenvalue at $0$ and all its other eigenvalues have strictly negative real parts (see Remark \ref{ergodicity}). This shows that matrix $C_\theta$ is stable.

Let $\bar{h}^N_\theta(t)$ and $\bar{\pi}_\theta$ be vectors containing the first $(m-1)$ components of $h^N(t)$ and $\pi_\theta$. Also let $\bar{D}_\theta$ be the $(m-1) \times (m-1)$ diagonal matrix with entries $\lambda(e_2,\theta),\dots,\lambda(e_m,\theta)$. From \eqref{main_eqn_hn} we get
\begin{align}
\label{eqnforhbarn}
\frac{d \bar{h}^N_\theta(t) }{dt} = \left( N C_\theta  - \bar{D}_\theta +  \lambda_\theta I_{m-1}\right) \bar{h}^N_\theta(t) - \bar{D}_\theta \bar{\pi}_\theta + \lambda_\theta \bar{\pi}_\theta.
\end{align}
Let $C^N_\theta$ be the matrix given by
\begin{align}
\label{defncntheta}
C^N_\theta = C_\theta - \frac{1}{N} \left( \bar{D}_\theta - \lambda_\theta I_{m-1}\   \right).
\end{align}
The stability of matrix $C_\theta$ implies that there exists a $\alpha >0$ such that for any $t \geq 0$ and $N$
\begin{align}
\label{alphacondition}
\| \exp( N C^N_\theta t ) \| \leq \exp( -N \alpha t).
\end{align}
The exact solution of \eqref{eqnforhbarn} is
\begin{align*}
\bar{h}^N_\theta(t) = \exp( N C^N_\theta t ) \bar{h}^N_\theta(0) - \int_{0}^t  \exp( N C^N_\theta (t-s) ) \left(  \bar{D}_\theta \bar{\pi}_\theta - \lambda_\theta \bar{\pi}_\theta \right)ds,
\end{align*}
which implies that
\begin{align*}
\| \bar{h}^N_\theta(t) \|  & \leq  e^{-N \alpha t }  \| \bar{h}^N_\theta(0) \| +  \int_{0}^t  e^{- N \alpha (t-s) }   \|  \bar{D}_\theta \bar{\pi}_\theta - \lambda_\theta \bar{\pi}_\theta \| ds \\
& \leq  e^{-N \alpha t }  \| \bar{h}^N_\theta(0) \| + \frac{  \|  \bar{D}_\theta \bar{\pi}_\theta - \lambda_\theta \bar{\pi}_\theta \|  }{N \alpha}.
\end{align*}
This along with \eqref{sumisorder1} shows that the function $h^N_\theta$ satisfies Condition \ref{regularitycondition}. In fact for any $T >0$
\begin{align}
\label{condforhn}
 \sup_{t \in [\epsilon_N,T] }  \|h^N_\theta(t) \| = O(N^{-1}) \textnormal{ and } \ \int_{0}^T \|h^N_\theta(t)\|dt = O(N^{-1}), 
\end{align}
where $\epsilon_N = 1 /\sqrt{N}$.

Let $H^N_\theta : \R_+ \to \R^m$ be defined by
\begin{align*}
H^N_\theta(t) = \frac{ \partial h^N_\theta (t) }{ \partial \theta}. 
\end{align*}
Differentiating \eqref{main_eqn_hn} with respect to $\theta$ we get
\begin{align*}
\frac{d H^N_\theta(t) }{dt} = \left( N Q^T_\theta  - D_\theta +  \lambda_\theta I_m\right) H^N_\theta(t) 
+ \left( N \frac{ \partial  Q^T_\theta }{ \partial \theta }  - \frac{ \partial  D_\theta }{\partial \theta } +  \frac{ \partial \lambda_\theta }{ \partial \theta } I_m\right) h^N_\theta(t) 
  -  \frac{\partial ( D_\theta \pi_\theta)  }{ \partial \theta} + \frac{ \partial (\lambda_\theta \pi_\theta ) }{ \partial \theta }.
\end{align*}
Note that 
\begin{align*}
\left\langle v^N_\theta ,  N \frac{ \partial  Q^T_\theta }{ \partial \theta } \right\rangle =  N \left\langle v^N_\theta ,  \frac{ \partial  Q^T_\theta }{ \partial \theta } \right\rangle  =  N \left\langle \bar{1}_m,  \frac{ \partial  Q^T_\theta }{ \partial \theta } \right\rangle + O(1) =N  \frac{ \partial (\bar{1}_m  Q^T_\theta) }{ \partial \theta }  + O(1) = O(1) ,
\end{align*}
where the last equality is true because $Q_\theta \bar{1}_m = \bar{0}_d$.
Let $G^N_\theta (t) = \langle v^N_\theta , H^N_\theta(t) \rangle$. Then $G^N_\theta$ satisfies an ordinary differential equation of the form
\begin{align*}
\frac{d G^N_\theta(t) }{dt} = e^N_\theta G^N_\theta(t) + f^N_\theta h^N_\theta(t) + g^N_\theta,
\end{align*}
where the sequences $\{e^N_\theta\}, \{g^N_\theta\}$ are in $O(N^{-1})$ and the sequence $f^N_\theta$ is in $O(1)$. Gronwall's inequality along with \eqref{condforhn} and \eqref{perturbedeigenvalueeigenvector} imply that
\begin{align}
\label{sumisorder2}
 \sup_{t \in [0,T]} | G^N_\theta(t)  | = O(N^{-1}) \ \textnormal{ and }   \sup_{t \in [0,T]} |  \langle \bar{1}_m ,  H^N_\theta(t) \rangle | = O(N^{-1}) .
\end{align}

Let $\bar{H}^N_\theta(t)$ be the first $(m-1)$ components of $H^N_\theta(t)$. Differentiating \eqref{eqnforhbarn} with respect to $\theta$, we see that $\bar{H}^N_\theta$ satisfies an equation of the form
\begin{align*}
\frac{d \bar{H}^N_\theta(t) }{dt} = \left( N C_\theta  - \bar{D}_\theta +  \lambda_\theta I_{m-1}\right) \bar{H}^N_\theta(t)
+
\left( N \frac{ \partial  C_\theta }{ \partial \theta }  - \frac{ \partial  \bar{D}_\theta }{\partial \theta } +  \frac{ \partial \lambda_\theta }{ \partial \theta } I_{m-1} \right) h^N_\theta(t) 
  -  \frac{\partial (\bar{ D }_\theta \bar{\pi}_\theta)  }{ \partial \theta} + \frac{ \partial (\lambda_\theta \bar{\pi}_\theta ) }{ \partial \theta }.
\end{align*}
If $C^N_\theta$ is the matrix given by \eqref{defncntheta}, then we can solve for $\bar{H}^N_\theta$ as
\begin{align*}
\bar{H}^N_\theta(t) & =  \exp( N C^N_\theta t ) \bar{H}^N_\theta(0) - \int_{0}^t  \exp( N C^N_\theta (t-s) ) \left(  \frac{\partial (\bar{ D }_\theta \bar{\pi}_\theta)  }{ \partial \theta} - \frac{ \partial (\lambda_\theta \bar{\pi}_\theta ) }{ \partial \theta } \right)ds \\
& + \int_{0}^t  \exp( N C^N_\theta (t-s) ) \left( N \frac{ \partial  C_\theta }{ \partial \theta }  - \frac{ \partial  \bar{D}_\theta }{\partial \theta } +  \frac{ \partial \lambda_\theta }{ \partial \theta } I_{m-1} \right) h^N_\theta(s) ds. 
\end{align*}
From \eqref{alphacondition} and \eqref{condforhn} we can deduce that $\bar{H}^N_\theta$ satisfies Condition \ref{regularitycondition}. Using \eqref{sumisorder2} it can be seen that $H^N_\theta$ also satisfies Condition \ref{regularitycondition}. This completes the proof of the proposition.
\end{proof}
\begin{corollary}
\label{corr_regularity}
Let $\hat{\beta}^N_\theta$ be the function defined in Proposition \ref{prop_regularity}. Then for any $T > 0$
\begin{align*}
\lim_{N \to \infty}  \sup_{t \in [0,T]}  \left| \left\langle  \bar{1}_m,  \hat{\beta}^N_\theta(t)\right\rangle \right|  = 0 \ \textnormal{ and } \ 
\lim_{N \to \infty}  \sup_{t \in [0,T]}  \left| \left\langle  \bar{1}_m,  \frac{ \partial  \hat{\beta}^N_\theta(t) }{ \partial \theta} \right\rangle \right|  = 0 
\end{align*}
\end{corollary}
\begin{proof}
The proof is immediate from \eqref{sumisorder1} and \eqref{sumisorder2}.
\end{proof}

We end this section with an important observation.
\begin{remark}
 \label{rem:extensionfinitetocountable} 
 To prove Proposition \ref{prop_regularity} we used results from the theory of perturbation of finite matrices. Consider the situation where the state space $\mathcal{E}$ of the Markov chain is countably infinite. Now the matrix of transition rates $Q_\theta$ is infinite and it can be seen as a linear operator on $\mathcal{E}$. Proposition \ref{mainoccmeasureprop} will still hold in this case and assuming the existence of a suitable \emph{Lyapunov function} (see \cite{Meyn1}) for the Markov chain, one can use results from the perturbation theory of linear operators (see \cite{Farid}) to prove Proposition \ref{prop_regularity} in a similar way.
\end{remark}

\subsection{Construction of a new process} \label{sec:constrnewprocess}
In this section we construct a new process $W^N_\theta$ and study some of its properties. As mentioned before, this process captures the one-dimensional distribution of $X^{N}_{ \gamma_2,\theta } $ (see Section \ref{sec:intro}) and its dynamics does not involve any ``fast" transitions. We begin by making a remark which will simplify the proof of Theorem \ref{mainsensitivityresult}.
\begin{remark}
\label{main_remark_simpl2}
Recall the description of the limiting process $\hat{X}_\theta$ from the statement of Theorem \ref{mainsensitivityresult}. Note that this process corresponds to a reduced model which does not contain any reactions in the set $(\Gamma_1 \cup \Gamma_2)^c = \{k=1,\dots,K : k  \notin \Gamma_1 \cup \Gamma_2\}$. This suggests that we can prove Theorem \ref{mainsensitivityresult} with the assumption that $(\Gamma_1 \cup \Gamma_2)^c$ is empty. If this is not the case, then our proof can be adjusted easily to account for the reactions in $(\Gamma_1 \cup \Gamma_2)^c$. We will also set $\gamma_2 = \gamma_1 + 1$, which can be ensured by redefining $N$, if necessary. 
\end{remark}

From now on we will always assume that $\gamma_2 = \gamma_1+1$ and $(\Gamma_1 \cup \Gamma_2)^c = \emptyset$. Under these assumptions the random time change representation of  $\{X^{N}_{\gamma_2,\theta}(t) : t \geq 0\}$ is given by 
\begin{align*}
X^{N}_{\gamma , \theta} (t) & = x_0 + \sum_{k \in \Gamma_1}Y_k \left( N  \int_{0}^t \lambda_k \left( X^{N}_{\gamma_2,\theta} (s)     ,\theta \right)ds  \right) \zeta_k
+ \sum_{k \in \Gamma_2}Y_k \left(  \int_{0}^t \lambda_k \left( X^{N}_{\gamma_2,\theta} (s)     ,\theta \right)ds  \right) \zeta_k.
\end{align*}
For each $k \in \Gamma_1 \cup \Gamma_2$ we let $\zeta^s_k = \Pi_2  \zeta_k$ and $\zeta^f_k = (I- \Pi_2) \zeta_k$. From \eqref{fastreactionsdonotaffect} we know that $\zeta^s_k = \bar{0}_d$ for each $k \in \Gamma_1$. If we define two processes $X^N_{S,\theta}$ and $X^N_{F,\theta}$ by
\begin{align}
\label{defn_slowandfastprocesses}
X^N_{S,\theta} (t) = \Pi_2  X^{N}_{\gamma_2 ,\theta}(t)  \ \textnormal{ and } \ X^{N}_{F,\theta} (t) = (I -\Pi_2 ) X^{N}_{\gamma_2,\theta}(t),
\end{align}
then their random time change representations are given by
\begin{align}
\label{rtc_slowprocess}
X^{N}_{S,\theta} (t) & = \Pi_2 x_0  + \sum_{k \in \Gamma_2} Y_k \left(  \int_{0}^t \lambda_k \left(X^{N}_{S,\theta} (s) +X^{N}_{F,\theta} (s)    ,\theta \right)ds  \right) \zeta^s_k \\
\label{rtc_fastprocess}
X^{N}_{F,\theta} (t) & = ( I - \Pi_2) x_0 +\sum_{k \in \Gamma_2} Y_k \left(  \int_{0}^t \lambda_k \left(X^{N}_{S,\theta} (s) +X^{N}_{F,\theta} (s)    ,\theta \right)ds  \right) \zeta^s_k \notag 
\\& + \sum_{k \in \Gamma_1} Y_k \left(  N \int_{0}^t \lambda_k \left(X^{N}_{S,\theta} (s) +X^{N}_{F,\theta} (s)    ,\theta \right)ds  \right) \zeta^f_k.
\end{align}
\begin{remark}
\label{fastmarkovprocess}
These representations show that between the successive jump times of $X^N_{S,\theta}$, if the state of this process is $v$, then the process $X^N_{F,\theta}$ evolves like a Markov process with state space $\mathbb{H}_v$ and generator $N\mathbb{C}^{v}_\theta$, where $\mathbb{C}^{v}_\theta$ is given by \eqref{defn_cvtheta}.
\end{remark}

The above remark motivates the construction of the process $W^N_\theta$. Before we describe this construction we need to define certain quantities. Let $\lambda_0(x,\theta) = \sum_{k \in \Gamma_2} \lambda_k(x,\theta)$ and for any $k \in \Gamma_2$, $v \in \Pi_2 \mathcal{S}$, $z \in \mathbb{H}_v$ and $t \geq 0$ define 
\begin{align}
\label{defnrhonk}
\rho^N_{k,\theta}(t,v,z) = \frac{  \E \left( \lambda_k(v+Z^N_\theta(t) ,\theta  ) \exp\left( - \int_{0}^{t}  \lambda_0( v +Z^N_\theta(s) ,\theta ) ds  \right)  \right)   }{  \E \left( \exp\left( - \int_{0}^{t}  \lambda_0( v + Z^N_\theta(s) ,\theta ) ds  \right)   \right) },
\end{align}
where $\{ Z^N_\theta(t) :  t \geq 0 \}$ is an independent Markov process with initial state $z$ and generator $N \mathcal{C}^{v}_\theta$.
For any $e \in \mathbb{H}_v$ define
\begin{align}
\label{defnbetanl}
\beta^N_{\theta}(t,v,z,e) & = \E \left(  \ind_{ \{ Z^N_\theta(t) =e \} } \exp\left( - \int_{0}^{t}  \lambda_0( v + Z^N_\theta(s) ,\theta ) ds  \right) \right) \\
\textnormal{ and } \quad  \Theta^N_{k,\theta}(t,v,z,e) & = \frac{ \lambda_k(v+e,\theta  )  \beta^N_{\theta}(t,v,z,e)}{ \rho^N_{k,\theta}(t,v,z)  \exp \left(  -\int_{0}^{t} \rho^N_{0,\theta}(t,v,z)  ds\right) }, \label{defnthetanl}
\end{align}
where
\begin{align}
\label{defrho0n}
\rho^N_{0,\theta}(t,v,z) = \sum_{k \in \Gamma_2} \rho^N_{k,\theta}(t,v,z).
\end{align}
If $\rho^N_{k,\theta}(t,v,z) =0$ then instead of defining $\Theta^N_{k,\theta}(t,v,z,e) $ by \eqref{defnthetanl} we do the following. We set $ \Theta^N_{k,\theta}(t,v,z,z) =1$ and set $ \Theta^N_{k,\theta}(t,v,z,e) =0$ for all $e  \in \mathbb{H}_v - \{z\}$.

Recall that the set $\mathbb{H}_v$ is finite due to part (A) of Assumption \ref{assforsensitivityresult}. Proposition \ref{mainoccmeasureprop} shows that the mapping $t \mapsto \beta^N_{\theta}(t,v,z,e)$ is continuously differentiable, and hence the mappings $t \mapsto \rho^N_{k,\theta}(t,v,z)$ and $t \mapsto  \Theta^N_{k,\theta}(t,v,z,e)$ are also continuously differentiable.
\begin{lemma}
\label{lemma_usefulproperties}
Fix a $v \in \Pi_2 \mathcal{S}$, $z \in \mathbb{H}_v$ and $t \geq 0$. 
\begin{itemize}
\item[(A)] Let $\{ Z^N_\theta(t) :  t \geq 0 \}$ be an independent Markov process with initial state $z$ and generator $N \mathcal{C}^{v}_\theta$. Then
\begin{align*}
\exp \left(  -\int_{0}^{t} \rho^N_{0,\theta}(s,v,z)ds\right) = \E \left( \exp\left( - \int_{0}^{t}  \lambda_0( v + Z^N_\theta(s) ,\theta ) ds  \right) \right).
\end{align*}
\item[(B)] For any $k \in \Gamma_2$
\begin{align*}
 \sum_{e \in \mathbb{H}_v }\Theta^N_{k,\theta}(t,v,z,e)  =1.
\end{align*}
\end{itemize}
\end{lemma}
\begin{proof}
Observe that
\begin{align*}
\rho^N_{0,\theta}(s,v,z)= \sum_{k \in \Gamma_2} \rho^N_{k,\theta}(s,v,z) &  =  \frac{  \E \left( \lambda_0(v+Z^N_\theta(s) ,\theta  ) \exp\left( - \int_{0}^{t}  \lambda_0( v +Z^N_\theta(s) ,\theta ) ds  \right)  \right)   }{  \E \left( \exp\left( - \int_{0}^{t}  \lambda_0( v + Z^N_\theta(s) ,\theta ) ds  \right)   \right) } \\
& = - \frac{d}{dt} \log \left(  \E \left( \exp\left( - \int_{0}^{t}  \lambda_0( x + Z^N_\theta(s) ,\theta ) ds  \right)   \right) \right).
\end{align*}
Integrating both sides with respect to $t$ and then exponentiating proves part (A).
From \eqref{defnrhonk} we get
\begin{align}
\label{defnjumpdistr1}
\rho^N_{k,\theta}(s,v,z) \exp \left(  -\int_{0}^{t} \rho^N_{0,\theta}(s,v,z) ds\right) & = \E \left( \lambda_k(v+Z^N_\theta(t) ,\theta  ) \exp\left( - \int_{0}^{t}  \lambda_0( v +Z^N_\theta(s) ,\theta ) ds  \right)  \right)  \notag \\
& = \sum_{e \in \mathbb{H}_v }\lambda_k(v+e ,\theta  )   \E \left(  \ind_{ \{ Z^N_\theta(t) =e \} }  \exp\left( - \int_{0}^{t}  \lambda_0( v +Z^N_\theta(s) ,\theta ) ds  \right)  \right) \notag \\
& = \sum_{e \in \mathbb{H}_v } \lambda_k(v+e ,\theta  )  \beta^N_{\theta}(t,v,z,e).
\end{align}
Hence
\begin{align*}
 \sum_{e \in \mathbb{H}_v }\Theta^N_{k,\theta}(t,v,z,e) =  \sum_{e \in \mathbb{H}_v }  \frac{ \lambda_k(v+e ,\theta  )  \beta^N_{\theta}(t,v,z,e) }{ \rho^N_{k,\theta}(s,v,z) \exp \left(  -\int_{0}^{t} \rho^N_{0,\theta}(s,v,z) ds\right)} = 1,
\end{align*}
and this proves part (B).
\end{proof}

Part (B) of Lemma \ref{lemma_usefulproperties} shows that for any $k \in \Gamma_2$, $v \in \Pi_2 \mathcal{S}$, $z \in \mathbb{H}_v$ and $t \geq 0$, we can regard $\Theta^N_{k,\theta}(t,v,z,\cdot)$ as a probability measure on $\mathbb{H}_v$. We know that $\mathbb{H}_v$ is a finite set. 
From now on, whenever we write $\mathbb{H}_v =\{ e_1,\dots,e_m\}$, we will assume that the elements are arranged in the lexicographical order on $\R^d$. For any $u \in (0,1)$ define
\begin{align}
\label{defndigamma}
\digamma^N_{ k,\theta} (t,v,z,u) = e_i \textnormal{ where } i = \min \left\{ l =1,\dots,m : u \leq \sum_{n=1}^l \Theta^N_{k,\theta}(t,v,z,e_n) \right\}.
\end{align}
Then a $\mathbb{H}_v$-valued random variable with distribution $\Theta^N_{k,\theta}(t,v,z,\cdot)$ can be generated by transforming a $\textnormal{Unif}(0,1)$ random variable $u$ with the function $\digamma^N_{ k,\theta} (t,v,z,\cdot)$. The next lemma will be useful in proving the main result.
\begin{lemma}
\label{lemma_decoupling}
Fix a $v \in \Pi_2 \mathcal{S}$, $z \in \mathbb{H}_v$ and $t > 0$. Let $\mathbb{H}_v =\{ e_1,\dots,e_m\}$ and $u$ be a $\textnormal{Unif}(0,1)$ random variable. Pick $i,j \in \{1,\dots,m\}$ such that $i \neq j$.
Then 
\begin{align*}
\lim_{h \to 0} \frac{ \P\left( \digamma^N_{ k,\theta} (t,v,z,u) =e_i \textnormal{ and }  \digamma^N_{ k,\theta+h} (t,v,z,u) = e_j   \right) }{h} \leq \sum_{e \in \mathbb{H}_v  } \left|  \frac{ \partial  \Theta^N_{k,\theta}(t,v,z,e) }{ \partial \theta } \right| 
\end{align*}
\end{lemma}
\begin{proof} 
%\leq |\mathbb{H}_v| \sum_{e \in \mathbb{H}_v }   \left| \frac{ \partial \Theta^N_{k,\theta}(t,v,z,e)  }{ \partial \theta }  \right|.
For proving this lemma we can assume that $\Theta^N_{k,\theta}(t,v,z,e) > 0$ for each $e \in \mathbb{H}_v$. Let $\mathbb{H}_v =\{ e_1,\dots,e_m\}$ and for any $l=1,\dots,m$ define
\begin{align*}
A_l(\theta) = \sum_{n=1}^l \Theta^N_{k,\theta}(t,v,z,e_n). 
\end{align*}
Note that $A_m(\theta) =1 $ for any $\theta$ due to part(B) of Lemma \ref{lemma_usefulproperties}. For convenience let $A_0(\theta)  = 0$ for any $\theta$.
For small values of $h$ we can write
\begin{align*}
 \P\left(  \digamma^N_{ k,\theta} (t,v,z,u) =e_i \textnormal{ and }  \digamma^N_{ k,\theta+h} (t,v,z,u) = e_j      \right) = \P\left(  u \in  ( A_{i-1}(\theta) , A_{i}(\theta) ) \textnormal{ and } u \in ( A_{j-1}(\theta+h) , A_{j}(\theta+h) )     \right).
\end{align*}
Since $\Theta^N_{k,\theta}(t,v,z,e_l) > 0$ for each $l=1,\dots,m$, this probability is $0$ if $j > i+1$ or $j < i-1$. Assume that $j = i+1$ for $i < m$. Then for smal values of $h$ we can write
\begin{align*}
 \P\left(  \digamma^N_{ k,\theta} (t,v,z,u) =e_i \textnormal{ and }  \digamma^N_{ k,\theta+h} (t,v,z,u) = e_j      \right) & = \P\left(  u \in  ( A_{i}(\theta +h ) , A_{i}(\theta) ) \right) \\
 & = \left[  \frac{ \partial  A_{i}(\theta) }{ \partial \theta } \right]^{-} h +o(h).
\end{align*}
Therefore
\begin{align*}
\lim_{h \to 0} \frac{ \P\left( \digamma^N_{ k,\theta} (t,v,z,u) =e_i \textnormal{ and }  \digamma^N_{ k,\theta+h} (t,v,z,u) = e_j   \right) }{h} =  \left[  \frac{ \partial  A_{i}(\theta) }{ \partial \theta } \right]^{-} .
\end{align*} 
Similarly for $j=i-1$ and $i>1$ we can show that
\begin{align*}
\lim_{h \to 0} \frac{ \P\left( \digamma^N_{ k,\theta} (t,v,z,u) =e_i \textnormal{ and }  \digamma^N_{ k,\theta+h} (t,v,z,u) = e_j   \right) }{h} =  \left[  \frac{ \partial  A_{i-1}(\theta) }{ \partial \theta } \right]^{+} .
\end{align*} 
Combining the last two relations proves the lemma. 
\end{proof}   

The new process $W^N_\theta$ will be a Markov process on state space $\hat{S}$ given by
\begin{align}
\label{defn:hats}
\hat{ \mathcal{S} }= \{ ( t,v,z ) \in \R_+ \times \R^d \times \R^d : v \in \Pi_2 \mathcal{S} \textnormal{ and } z \in \mathbb{H}_v \}.
\end{align}
Let $\Pi_{ \hat{ \mathcal{S} }}$ be the projection map from $\hat{ \mathcal{S} }$ to $\Pi_2 \mathcal{S} $ defined by
\begin{align}
\label{defn:projhats}
\Pi_{\hat{ \mathcal{S} }}(t,v,z) = v.
\end{align}
We now define a class $\mathcal{C}$ of bounded real-valued functions over $\hat{ \mathcal{S} }$ by
\begin{align}
\label{classc}
\mathcal{C} = & \left\{ f    \in  \mathcal{B}(\hat{ \mathcal{S} }) :  f(\cdot,v,z) \textnormal{ is continuously differentiable for each } v \in \Pi_2 \mathcal{S} \textnormal{ and } z \in \mathbb{H}_v\right\}.
\end{align}  
Let $\{W^N_\theta (t) : t \geq 0\}$ be the $\hat{ \mathcal{S}}$-valued Markov process with initial state $(0,v_0,z_0) = (0, \Pi_2 x_0 , (I- \Pi_2) x_0)$ and generator given by
\begin{align}
\label{genbntheta}
\mathbb{B}^N_\theta f(t,v,z) = \frac{ \partial f(t,v,z)}{ \partial t }+ \sum_{k \in \Gamma_2} \rho^N_{k,\theta}(t,v,z)  \sum_{e \in \mathbb{H}_v} \left( f(0,v+ \zeta^s_k, e+ \zeta^f_k ) -f(t,v,z) \right)  \Theta^N_{k,\theta}(t,v,z,e),
\end{align}
for all $f \in \mathcal{D}( \mathbb{B}^N_\theta  ) =\mathcal{C}$. The existence and uniqueness of the process $W^N_\theta$ is a direct consequence of the well-posedness of the martingale problem for $\mathbb{B}^N_\theta$, which is verified in Lemma \ref{mainlemmapp}. 

In the rest of this section we study some properties of the process $W^N_\theta$. Observe that the definition of $\hat{ \mathcal{S} }$ (see \eqref{defn:hats}) allows us to write
\begin{align}
\label{decompwntheta}
W^N_\theta(t) = \left(  \tau^N_\theta(t) ,  V^N_\theta(t) ,  Z^N_\theta(t)\right) \textnormal{ for all } t \geq 0,
\end{align}
where $\tau^N_\theta$, $V^N_\theta$ and $Z^N_\theta$ are processes with state spaces $\R_+, \Pi_2 \mathcal{S}$ and $\cup_{v \in  \Pi_2 \mathcal{S}}  \mathbb{H}_v$ respectively. Let $\sigma^N_i$ denote the $i$-th jump time of the process $W^N_\theta$ for $i=1,\dots$. We define $\sigma^N_0 = 0$ for convenience. From the form of the generator $\mathbb{B}^N_\theta$ it is immediate that between the jump times, $\tau^N_\theta$ increases linearly at rate $1$ while $V^N_\theta$ and $Z^N_\theta$ remain constant. Hence 
\begin{align}
\label{whatistau}
\left( \tau^N_\theta(t) ,V^N_\theta(t) , Z^N_\theta(t) \right) = \left( t - \sigma^N_{i-1}  ,  V^N_\theta(\sigma^N_{i-1})  , Z^N_\theta(\sigma^N_{i-1})  \right) \textnormal{ for any } i \in \N \textnormal{ and } t \in [\sigma^N_{i-1} , \sigma^N_i). 
\end{align}
Let $\eta_i$ be the $\Gamma_2$-valued random variable that denotes the direction of the jump at time $\sigma^N_i$ and let $\xi_i$ be the random variable given by $Z^N_\theta( \sigma^N_i - )$. The form of $\mathbb{B}^N_\theta$ allows us to compute the distributions of the random variables $(\sigma^N_i  - \sigma^N_{i-1})$, $\eta_i$ and $\xi_i$ from the values of $ V^N_\theta(\sigma^N_{i-1} )$ and $Z^N_\theta(\sigma^N_{i-1})$. Let $E_{i}(v,z)$ denote the event
\begin{align*}
E_{i}(v,z) = \{  V^N_\theta(\sigma^N_i ) = v,  Z^N_\theta(\sigma^N_i ) = z  \}. 
\end{align*}
Then given $E_{i-1}(v,z)$, $(\sigma^N_i  - \sigma^N_{i-1})$ is a $\R_+$-valued random variable with density
\begin{align}
\label{densityoftimeincerement}
\rho^N_{0,\theta}(t,v,z) \exp \left( -\int_{0}^{t} \rho^N_{0,\theta}(s,v,z) ds\right)dt.
\end{align}
Given $E_{i-1}(v,z)$ and $(\sigma^N_i  - \sigma^N_{i-1}) = t$, $\eta_i$ is a $\Gamma_2$-valued random variable with distribution
\begin{align}
\label{distr:etai}
\P\left(  \eta_i = k \vert E_{i-1}(v,z), (\sigma^N_i  - \sigma^N_{i-1}) = t  \right) = \frac{  \rho^N_{k,\theta}(t,v,z)  }{  \rho^N_{0,\theta}(t,v,z)  }.
\end{align}
Moreover conditioned on $E_{i-1}(v,z)$, $(\sigma^N_i  - \sigma^N_{i-1}) = t$ and $\eta_i= k$, the $\mathbb{H}_v$-valued random variable $\xi_i$ has distribution $\Theta^N_{k,\theta}(t,v,z,\cdot)$. Using \eqref{densityoftimeincerement} and \eqref{distr:etai} we can deduce that
\begin{align}
\label{distrequ1}
& \lim_{h \to 0}  \frac{ \P \left(  \sigma^N_i \in ( \sigma^{N}_{i-1} +t ,\sigma^{N}_{i-1}+ t+h) , V^N_{\theta}( \sigma^N_i) =  v +\zeta^s_k ,  Z^N_{\theta}( \sigma^N_i ) = e+\zeta^f_k  \middle\vert E_{i-1}(v,z)   \right)  }{h}   \\
& =  \rho^N_{k,\theta}(t, v,z) \exp \left( - \int_{0}^{t} \rho^N_{0,\theta}( u,v,z) du \right) \Theta^N_{k,\theta}( t,v,z,e), \notag
\end{align}
for any $i=1,2,\dots$.
\begin{remark}
\label{rem:construction}
The preceding discussion suggests a simple scheme to construct the process \newline $\{ W^N_\theta (t) = (\tau^N_\theta(t) ,  V^N_\theta(t) ,  Z^N_\theta(t)) : t \geq 0 \}$ with generator $\mathbb{B}^N_\theta$ and initial state $(0,v_0,z_0)$. Consider the random time change representation
\begin{align}
\label{rtcr:vntheta}
V^N_\theta(t) = v_0 + \sum_{k \in \Gamma_2} Y_k\left( \int_{0}^t \rho^N_{k,\theta}( \tau^N_\theta(s),V^N_\theta (s), Z^N_\theta(s)  ) ds  \right) \zeta^s_k,
\end{align}
where $\{ Y_k : k \in \Gamma_2 \}$ is a family of independent unit rate Poisson processes. The processes $\tau^N_\theta,V^N_\theta$ and $Z^N_\theta$ can be constructed as follows. For each $i \in \N_0$ let $\sigma^N_i$ be the $i$-th jump time of process $V^N_\theta$, where $\sigma^N_0 = 0$. Defining $(\tau^N_\theta(0) ,  V^N_\theta(0) ,  Z^N_\theta(0)) = (0,v_0,z_0)$ constructs the process $W^N_\theta$ until time $\sigma^N_0$. Assume that this process is constructed until time $\sigma^N_{i-1}$ for some $i =1,2,\dots$. Then the next jump time $\sigma^N_i$ can be evaluated from \eqref{rtcr:vntheta} and the process $W^N_\theta$ can be defined in the time interval $[\sigma^N_{i-1} ,\sigma^N_i )$ using \eqref{whatistau}. If $V^N_\theta(\sigma^N_i ) = v,  Z^N_\theta(\sigma^N_i ) = z $ and $\sigma^N_i -\sigma^N_{i-1} = t$ then we choose random variables $\eta_i$ and $\xi_i$ according to distributions \eqref{distr:etai} and $\Theta^N_{\eta_i,\theta}(t,v,z,\cdot)$ respectively and define
\begin{align*}
( \tau^N_\theta( \sigma^N_i )  ,V^N_\theta(\sigma^N_i ) ,  Z^N_\theta(\sigma^N_i )  = (0, v+ \zeta^s_{\eta_i} , \xi_i + \zeta^f_{\eta_i}).
\end{align*}
This completes the construction of the process until the next jump time $\sigma^N_i$. Proceeding this way we can define $W^N_\theta (t) = (\tau^N_\theta(t) ,  V^N_\theta(t) ,  Z^N_\theta(t))$ for all $t\geq 0$. The relation \eqref{distrequ1} ensures that the process $W^N_\theta $ has generator $\mathbb{B}^N_\theta$.
\end{remark}

In the next proposition we show that the one-dimensional distribution of the process $X^{N}_{\gamma_2,\theta}$ can be captured with the process $W^N_\theta$.
\begin{proposition}
\label{prop_preserve}
 For $i \in \N$, let $\delta^N_i$ and $\sigma^N_i$ denote the $i$-th jump time of the processes $\Pi_2 X^{N}_{\gamma_2,\theta}$ and $W^N_\theta$ respectively. We define $\delta^N_0 = \sigma^N_0 = 0$ for convenience. Then we have the following.
\begin{itemize}
\item[(A)] Let the processes $V^N_\theta$ and $Z^N_\theta$ be related to the process $W^N_\theta$ by \eqref{decompwntheta}. For each $i = 0,1,2,\dots$,
\begin{align}
\label{mainpropparta}
\left( \delta^N_i , \Pi_2 X^N_{\gamma_2,\theta}(\delta^N_i), (I-\Pi_2) X^N_{\gamma_2,\theta}(\delta^N_i) \right) \stackrel{d}{=} \left(  \sigma^N_{i}, V^N_\theta( \sigma^N_{i} ) , Z^N_\theta( \sigma^N_{i}) \right) , 
\end{align}
where $\stackrel{d}{=} $ denotes equality in distribution. 

\item[(B)] Let $f : \mathcal{S} \to \R$ be a polynomially growing function with respect to projection $\Pi_2$ (see Definition \ref{polynomialgrowth}). Then for any $t \geq 0$
\begin{align*}
\E\left( f( X^{N}_{\gamma_2, \theta} (t) ) \right) = \E\left( f^N_\theta( W^N_\theta(t) )  \right),
\end{align*}
where $f_\theta : \hat{ \mathcal{S} } \to \R$ is the function given by
\begin{align}
\label{defn_ftheta}
f^N_\theta(t,v,z) = \frac{\sum_{e \in \mathbb{H}_v } f(v+e) \beta^N_{\theta}(t,v,z,e)   }{ \exp \left( -\int_{0}^t  \rho^N_{0,\theta} (s,v,z)ds \right) }. 
\end{align}
\end{itemize}
\end{proposition}
\begin{remark}
\label{rem:functioninclassc}
Note that for any $v \in \Pi_2 \mathcal{S}$ and $z \in \mathbb{H}_v$, the mapping $t \mapsto f^N_\theta(t,v,z)$ is continuously differentiable with respect to $t$. Let  $\partial  f^N_\theta(t,v,z) /\partial t$ denote the derivative of this map. 
Since $f$ is polynomially growing with respect to projection $\Pi_2$, the sequences of functions $\{f^N_\theta : N \in \N\}$, $\{\partial  f^N_\theta /\partial t : N \in \N\}$ and $\{ \mathbb{B}^N_\theta f^N_\theta : N \in \N \}$ are also polynomially growing with respect to projection $\Pi_{ \hat{\mathcal{S}} }$.
%In particular, if $f$ is bounded then the sequence of functions $\{f^N_\theta : N \in \N \}$ belongs to class $\mathcal{C}$ (see \eqref{classc}) and this sequence is uniformly bounded in $N$.
\end{remark}
\begin{proof}
We prove part (A) by induction in $i$. Relation \eqref{mainpropparta} certainly holds for $i=0$. Suppose it holds for $(i-1)$ for some $i \in \N$. Then
\begin{align}
\label{induchyp}
\left( \delta^N_{i-1} , X^N_{S,\theta}(\delta^N_{i-1}), X^N_{F,\theta}(\delta^N_{i-1}) \right) \stackrel{d}{=} W^N_\theta( \sigma^N_{i-1} ),
\end{align}
where the processes $X^N_{S,\theta}$ and $X^N_{F,\theta}$ are given by \eqref{defn_slowandfastprocesses}.

For any $v \in \Pi_2 \mathcal{S}$ and $z \in \mathbb{H}_v$ let $E_{i-1}(v,z)$ denote the event
\begin{align}
\label{eventdefni_1}
E_{i-1}(v,z) = \{ X^N_{S,\theta}(\delta^N_{i-1} ) = v,  X^N_{F,\theta}(\delta^N_{i-1} ) = z  \}. 
\end{align}
Let $\eta_i$ be the $\Gamma_2$-valued random variable that gives the jump direction of the process $X^N_{S,\theta}$ at time $\delta^N_i$. For any $t >0$, $k \in \Gamma_2$ and $e \in \mathbb{H}_v$ we can write
\begin{align}
\label{limitrelations0}
& \lim_{h \to 0}  \frac{ \P \left(  \delta^N_i \in (\delta^N_{i-1} +t,\delta^N_{i-1} + t+h) , X^N_{S,\theta}( \delta^N_i) = v + \zeta^s_k, X^N_{F,\theta}( \delta^N_i ) = e + \zeta^f_k \middle\vert E_{i-1}(v,z)   \right)  }{h}  \notag \\
& =\lim_{h \to 0}  \frac{ \P \left(  \delta^N_i - \delta^N_{i-1}  \in (t,t+h) , \eta_i = k , X^N_{F,\theta}( \delta^N_i -)  =e \middle\vert E_{i-1}(v,z)   \right)  }{h}.
\end{align}
Let $\{ \bar{Z}^N_\theta(t) : t \geq 0\}$ be an independent Markov process with initial state $z$ and generator $N \mathbb{C}^{v}_\theta$. For each $k \in \Gamma_2$ let $u_k$ be an independent $\textnormal{Unif}(0,1)$ random variable. 
Using the observation made in Remark \eqref{fastmarkovprocess}, and the random time change representation \eqref{rtc_slowprocess} we can write
\begin{align}
\label{limitrelations1}
& \P \left(  \delta^N_i - \delta^N_{i-1}  \in (t,t+h) , \eta_i = k , X^N_{F,\theta}( \delta^N_i -)  =e \middle\vert E_{i-1}(v,z)   \right) \notag \\
& = \P\left(  \int_{0 }^{t+h} \lambda_k( v + \bar{Z}^N_\theta(u) ,\theta )du \geq -\log u_k \geq  \int_{0 }^{t} \lambda_k( v+  \bar{Z}^N_\theta(u),\theta  )du ,   \int_{0 }^{t} \lambda_j(  v + \bar{Z}^N_\theta(u) ,\theta  )du < -\log u_j  \right. \notag \\
& \left.  \  \  \ \  \textnormal{ for all } j \in \Gamma_2 - \{k\}  \textnormal{ and }  \bar{Z}^N_\theta(t) = e  \right) +o(h) \notag \\
& =   \lambda_k( v+ e , \theta)  \E\left(  \ind_{ \{ \bar{Z}^N_\theta(t) = e \} }  \exp \left( - \int_{0}^{t} \lambda_0( v+ \bar{Z}^N_\theta(u) ,\theta ) du   \right)     \right)h +o(h), 
\end{align}
where $o(h)$ denotes any quantity which upon division by $h$, goes to $0$ as $h \to 0$. To obtain \eqref{limitrelations1} we integrated with respect to the joint density of $\{ u_k : k \in \Gamma_2\}$. Note that due to \eqref{defnbetanl} and \eqref{defnthetanl} we get
\begin{align*}
&\lambda_k( v+ e , \theta)  \E\left(  \ind_{ \{ \bar{Z}^N_\theta(t) = e \} }  \exp \left( - \int_{0}^{t} \lambda_0( v+ \bar{Z}^N_\theta(u) ,\theta ) du    \right)   \right)\\ 
& =  \lambda_k( v+ e, \theta) \beta^N_{\theta}(t,v,z,e) \\
& = \rho^N_{k,\theta}(t,v,z) \exp \left(  -\int_{0}^{t} \rho^N_{0,\theta}(s,v,z) ds\right)\Theta^N_{k,\theta}(t,v,z,e) .
\end{align*}
Hence relations \eqref{limitrelations0} and \eqref{limitrelations1} yield
\begin{align*}
&\lim_{h \to 0}  \frac{ \P \left(  \delta^N_i \in ( \delta^N_{i-1} +t, \delta^N_{i-1} + t+h) , X^N_{S,\theta}( \delta^N_i) = v + \zeta^s_k, X^N_{F,\theta}( \delta^N_i ) = e + \zeta^f_k \middle\vert E_{i-1}(v,z)   \right)  }{h} \\
& =  \rho^N_{k,\theta}(t,v,z) \exp \left(  -\int_{0}^{t} \rho^N_{0,\theta}(s,v,z) ds\right)\Theta^N_{k,\theta}(t,v,z,e).
\end{align*}
From \eqref{distrequ1} it follows that for all $v \in \Pi_2 \mathcal{S}$ and $z \in \mathbb{H}_v$
\begin{align*}
& \lim_{h \to 0}  \frac{ \P \left(  \sigma^N_i \in (\sigma^N_{i-1}+ t, \sigma^N_{i-1}  t+h) , V^N_{\theta}( \sigma^N_i) =  v +\zeta^s_k ,  Z^N_{\theta}( \sigma^N_i ) = e +\zeta^f_k  \middle\vert E_{i-1}(v,z)   \right)  }{h}   \\
&=\lim_{h \to 0}  \frac{ \P \left(  \delta^N_i \in ( \delta^N_{i-1} +t, \delta^N_{i-1} + t+h), X^N_{S,\theta}( \delta^N_i) = v + \zeta^s_k, X^N_{F,\theta}( \delta^N_i ) = e+ \zeta^f_k \middle\vert E_{i-1}(v,z)   \right)  }{h}.
\end{align*}
This relation and \eqref{induchyp} imply that 
\begin{align*}
\left( \delta^N_{i} , X^N_{S,\theta}(\delta^N_{i}), X^N_{F,\theta}(\delta^N_{i}) \right) \stackrel{d}{=} \left(  \sigma^N_{i}, V^N_\theta( \sigma^N_{i} ) , Z^N_\theta( \sigma^N_{i}) \right),
\end{align*}
which completes the proof of part (A).

We now prove part (B). From Remark \ref{rem:functioninclassc} and Lemma \ref{mainlemmapp} we can conclude that for any $t \geq 0$
\begin{align*}
 \sup_{N \in \N} \E\left(  f^N_{\theta}( W^N_{\theta}(t)  )\right) < \infty .
\end{align*}
Moreover, one can rework the proof of part (C) of Lemma \ref{lemma1:app} to show that 
\begin{align*}
\sup_{N \in \N} \E\left(  f( X^N_{\gamma_2,\theta}(t)  )\right) < \infty \textnormal{ for any } t \geq 0.
\end{align*}
Let $\{\mathcal{F}_t\}$ be the filtration generated by the process $\{  X^{N}_{\gamma_2,\theta}(t)  : t \geq 0 \}$. Then we can write
\begin{align}
\label{conditionedxngamma}
\E \left(  f( X^{N}_{\gamma_2,\theta}(t) )  \right) & =  \sum_{i=1}^{\infty} \E \left( \ind_{ \{ \delta^N_{i-1} \leq t < \delta^N_{i} \} }  f( X^N_{\gamma_2, \theta}(t)  )    \right)  \notag \\ 
 & = \sum_{i=1}^{\infty} \E \left( \ind_{ \{ \delta^N_{i-1} \leq t < \delta^N_{i} \} }  f( X^N_{S, \theta}(t) + X^N_{F, \theta}(t)   )    \right) \notag \\
& = \sum_{i=1}^{\infty} \E \left(  \ind_{ \{ \delta^N_{i-1} \leq t \} }  \E\left(   \ind_{ \{ \delta^N_{i} -\delta^N_{i-1} > t -\delta^N_{i-1}    \}  }  f( X^N_{S, \theta}(\delta^N_{i-1}) + X^N_{F, \theta}(t)   )   \vert \mathcal{F}_{ \delta^N_{i-1} }  \right)    \right).
\end{align}
For any $v \in \Pi_2\mathcal{S}$ and $z \in \mathbb{H}_v$, let $E_{i-1}(v,z)$ be the event given by \eqref{eventdefni_1}. Suppose $\mathbb{H}_v= \{e_1,\dots,e_m\}$ and $\{ \bar{Z}^N_\theta(t)  : t \geq 0\}$ is an independent Markov process with initial state $z$ and generator $N\mathcal{C}^{v}_\theta$. For each $k \in \Gamma_2$ let $u_k$ be an independent $\textnormal{Unif}(0,1)$ random variable. 
Using the observation made in Remark \eqref{fastmarkovprocess}, and the random time change representation \eqref{rtc_slowprocess}, for any $s<t$ we can write
\begin{align*}
& \E\left(   \ind_{ \{ \delta^N_{i} -\delta^N_{i-1} > t -\delta^N_{i-1}    \}  }  f( X^N_{S, \theta}(\delta^N_{i-1}) + X^N_{F, \theta}(t)   )   \vert E_{i-1}(v,z)  , \delta^N_{i-1} = s\ \right) \\
&  =  \E\left(   \ind_{ \{ \delta^N_{i} -\delta^N_{i-1} > t -s   \}  }  f( v + \bar{Z}^N_{\theta}(t-s)   )  \right) \\
& =\sum_{e \in \mathbb{H}_v } \P\left(  \int_{0 }^{t-s} \lambda_k(  v + \bar{Z}^N_\theta(u) ,\theta  )du < -\log u_k \textnormal{ for all } k \in \Gamma_2, \ \bar{Z}^N_\theta(t-s) = e   \right) f(v+e) \notag \\
& = \sum_{e \in \mathbb{H}_v } \   \E\left(  \ind_{ \{ \bar{Z}^N_\theta(t-s) = e \} }  \exp \left( - \int_{0}^{t-s} \lambda_0( v+ \bar{Z}^N_\theta(u) ,\theta ) du    \right)    \right)f(v+e). 
\end{align*}
The last inequality is obtained by integrating with respect to the joint density of $\{ u_k : k \in \Gamma_2\}$. Due to \eqref{defnbetanl} and \eqref{defn_ftheta} we obtain
\begin{align*}
&\E\left(   \ind_{ \{ \delta^N_{i} -\delta^N_{i-1} > t -\delta^N_{i-1}    \}  }  f( X^N_{S, \theta}(\delta^N_{i-1}) + X^N_{F, \theta}(t)   )   \vert  E_{i-1}(v,z)  , \delta^N_{i-1} = s \right) \\&= \sum_{e \in \mathbb{H}_v } f(v+e) \beta^N_{\theta}(t-s,v,z,e) \\
& = \exp \left( -\int_{0}^{t-s}  \rho^N_{0,\theta} (u,v,z)du \right) f^N_\theta(t-s,v,z),
\end{align*}
which shows that
\begin{align*}
 &\E\left(   \ind_{ \{ \delta^N_{i} -\delta^N_{i-1} > t -\delta^N_{i-1}    \}  }  f( X^N_{S, \theta}(\delta^N_{i-1}) + X^N_{F, \theta}(t)   )   \vert \mathcal{F}_{ \delta^N_{i-1} }  \right) \\
 &  = \exp \left( -\int_{0}^{t-\delta^N_{i-1} }  \rho^N_{0,\theta} (u,X^N_{S, \theta}(\delta^N_{i-1} ) ,X^N_{F, \theta}(\delta^N_{i-1} ) )du \right) f^N_\theta(t-\delta^N_{i-1} ,X^N_{S, \theta}(\delta^N_{i-1} ) , X^N_{F, \theta}(\delta^N_{i-1} )   ).
\end{align*}
Substituting this relation in \eqref{conditionedxngamma} and using part (A) gives us
\begin{align*}
& \E \left(  f( X^{N}_{\gamma_2,\theta }(t) )  \right) \\
&= \sum_{i=1}^{\infty} \E \left(  \ind_{ \{ \delta^N_{i-1} \leq t \} }  \exp \left( -\int_{0}^{t-\delta^N_{i-1} }  \rho^N_{0,\theta} (u,X^N_{S, \theta}(\delta^N_{i-1} ) ,X^N_{F, \theta}(\delta^N_{i-1} ))du \right) f^N_\theta(t-\delta^N_{i-1} ,X^N_{S, \theta}(\delta^N_{i-1} ) , X^N_{F, \theta}(\delta^N_{i-1} )   )     \right) \\
&= \sum_{i=1}^{\infty} \E \left(  \ind_{ \{ \sigma^N_{i-1} \leq t \} }  \exp \left( -\int_{0}^{t-\sigma^N_{i-1} }  \rho^N_{0,\theta} (u,V^N_{\theta}(\sigma^N_{i-1} ) ,Z^N_{\theta}(\sigma^N_{i-1} ))du \right) f^N_\theta(t-\sigma^N_{i-1} ,V^N_{\theta}(\sigma^N_{i-1} ) , Z^N_{\theta}(\sigma^N_{i-1} )   )     \right).
\end{align*}
However from \eqref{whatistau} and \eqref{densityoftimeincerement} we can conclude that
\begin{align*}
& \sum_{i=1}^{\infty} \E \left(  \ind_{ \{ \sigma^N_{i-1} \leq t \} }  \exp \left( -\int_{0}^{t-\sigma^N_{i-1} }  \rho^N_{0,\theta} (u,V^N_{\theta}(\sigma^N_{i-1} ) ,Z^N_{\theta}(\sigma^N_{i-1} ))du \right) f^N_\theta(t-\sigma^N_{i-1} ,V^N_{\theta}(\sigma^N_{i-1} ) , Z^N_{\theta}(\sigma^N_{i-1} )   )     \right) \\
& =  \sum_{i=1}^{\infty} \E \left(  \ind_{ \{ \sigma^N_{i-1} \leq t  < \sigma^N_i \} }   f^N_\theta(t-\sigma^N_{i-1} ,V^N_{\theta}(\sigma^N_{i-1} ) , Z^N_{\theta}(\sigma^N_{i-1} )   )     \right) \\
& = \sum_{i=1}^{\infty} \E \left(  \ind_{ \{ \sigma^N_{i-1} \leq t  < \sigma^N_i \} }   f^N_\theta( \tau^N_\theta(t) ,V^N_{\theta}(t) , Z^N_{\theta}(t ) )     \right) \\
& =  \E \left(   f^N_\theta( \tau^N_\theta(t) ,V^N_{\theta}(t) , Z^N_{\theta}(t ) )     \right) \\
& = \E \left(   f^N_\theta( W^N_{\theta}(t )      \right) .
\end{align*}
This proves part (B) of the proposition.
\end{proof}

Part (B) of Assumption \ref{assforsensitivityresult} says that a Markov process with generator $\mathbb{C}^v_\theta$ is ergodic and its unique stationary distribution is $\pi^v_\theta \in \mathcal{P}( \mathbb{H}_v )$. Since $\mathbb{H}_v$ is finite, we can view $\pi^v_\theta $ as a vector in $\R^n$ where $n = | \mathbb{H}_v|$. The differentiability of $\pi^z_{\theta}$ with respect to $\theta$ follows from arguments given in Section \ref{sec:fmcs}.
Let $\{f^N : N \in \N \}$ be a sequence of real valued functions on $\R_+$ and let $c$ be a constant. In the next lemma we will use the notation $f^N \rightarrow c$ to denote that the sequence of functions $\{ (\hat{f}^N -c) :  N \in \N \}$ satisfies Condition \ref{regularitycondition}.

\begin{lemma}
\label{limitinglemma}
Fix a $v \in \Pi_2 \mathcal{S}$ and $z \in \mathbb{H}_v$. Then we have the following.
\begin{itemize}
\item[(A)] For any $k \in \Gamma_2$
\begin{align*}
\rho^N_{k,\theta}(\cdot,v,z) \rightarrow \hat{\lambda}_k(v, \theta) \  \textnormal{ and } \ \frac{ \partial \rho^N_{k,\theta}(\cdot,v,z) }{\partial \theta} \rightarrow \frac{ \partial \hat{\lambda}_k(v, \theta) }{\partial \theta},
\end{align*}
where $\hat{\lambda}_k$ is defined by \eqref{defn_lambdahattheta}.

\item[(B)] For any $k \in \Gamma_2$ and $e \in \mathbb{H}_v$
\begin{align*}
\Theta^N_{k,\theta  }(\cdot, v,z,e) \rightarrow \frac{ \lambda_k(v+e,\theta) \pi^v_\theta(e)}{ \hat{\lambda}_k(v, \theta) } \  \textnormal{ and } \ \frac{\partial \Theta^N_{k,\theta}(\cdot, v,z,e)   }{ \partial \theta} \rightarrow \frac{ \partial  }{\partial \theta} \left( \frac{ \lambda_k(v+e,\theta) \pi^v_{\theta}(e)   }{ \hat{\lambda}_k(v,\theta) } \right). 
\end{align*}
\item[(C)] Fix a function $f : \mathcal{S}  \to \R $. Let $f^N_\theta$ and $f_\theta$ be given by \eqref{defn_ftheta} and \eqref{defn_fthetalimit} respectively. Then 
\begin{align*}
f^N_\theta(\cdot,v,z)  \rightarrow f_\theta(v)  \  \textnormal{ and } \ \frac{\partial f^N_\theta(\cdot,v,z)  }{ \partial \theta} \rightarrow \frac{ \partial f_\theta(v)   }{\partial \theta}.
\end{align*}
\end{itemize}
\end{lemma}
\begin{proof}
Assume that $\mathbb{H}_v = \{e_1,\dots,e_m\}$. For each $l=1,\dots,m$, let $\hat{\beta}^N_{\theta, l} : \R_+ \to \R$ be given by
\begin{align*}
\hat{\beta}^N_{\theta, l}(t) = \beta^N_{\theta}(t,v,z,e_l) - \exp( - d_\theta (v) t)\pi^v_{\theta}(e_l),
\end{align*}
where $d_\theta(v) = \sum_{e \in \mathbb{H}_v }  \lambda_0(v+ e,\theta ) \pi^v_{\theta}(e)$.
Observe that
\begin{align*}
\exp\left( -\int_{0}^t \rho^N_{0,\theta}(s,v,z)ds  \right) = \sum_{e \in \mathbb{H}_v  }\beta^N_{\theta}(t,v,z,e) = \sum_{l=1}^m   \hat{\beta}^N_{\theta,l}(t) + \exp( - d_\theta (v) t). 
\end{align*}
From Corollary \ref{corr_regularity} we get that for any $T >  0$
\begin{align}
\label{conv_corr1}
& \lim_{N \to \infty} \sup_{t \in [0,T]} \left| \exp\left( -\int_{0}^t \rho^N_{0,\theta}(s,v,z)ds  \right)  -\exp( - d_\theta (v) t)  \right| = 0 \\
\label{conv_corr2}
 \textnormal{ and } \ & \lim_{N \to \infty} \sup_{t \in [0,T]} \left|  \frac{ \partial  }{\partial \theta}\exp\left( -\int_{0}^t \rho^N_{0,\theta}(s,v,z)ds  \right)  - \frac{ \partial }{ \partial \theta }\exp( - d_\theta (v) t)   \right| = 0 
\end{align}
Using part (A) of Lemma \ref{lemma_usefulproperties} we can write
\begin{align*}
\rho^N_{k,\theta}(t,v,z) = \frac{  \sum_{l=1}^m \lambda_k(v+e_l,\theta )  \hat{\beta}^N_{\theta,l}(t)  } { \sum_{l=1}^m  \hat{\beta}^N_{\theta,l}(t)   } . 
\end{align*}
 From Proposition \ref{prop_regularity} we can see that each $\hat{\beta}^N_{\theta, l}$ satisfies Condition \ref{regularitycondition}. This fact along with \eqref{conv_corr1} and \eqref{conv_corr2} proves part (A).

The proof of part (B) is immediate from the definition of $\Theta^N_{k,\theta}$ (see \eqref{defnthetanl}), part (A), \eqref{conv_corr1} and \eqref{conv_corr2}. Note that $f^N_\theta$ can be written as
\begin{align*}
f^N_\theta(t,v,z) = \frac{  \sum_{l=1}^m f(v+e_l) \hat{\beta}^N_{\theta,l}(t)   } { \sum_{l=1}^m  \hat{\beta}^N_{\theta,l}(t) } ,
\end{align*}
which enables us to prove part (C) in the same way as part (A).
\end{proof}
%Also let $\{ \hat{X}_\theta(t) : t \geq 0 \}$ be the process that appears in the statement of Theorem \ref{mainsensitivityresult}.

For the next proposition, recall the definition of the projection map $\Pi_{\hat{ \mathcal{S} }}$ from \eqref{defn:projhats} and the definition of $\hat{\mathbb{A} }_\theta$ from \eqref{finallimitgen}.

\begin{proposition}
\label{prop_limitresult}
Fix $(t_0,v_0,z_0) \in  \hat{ \mathcal{S} }$ and let $W^N_\theta$ be the Markov process with generator $\mathbb{B}^N_\theta$ and initial state $(t_0,v_0,z_0)$.
Then the sequence of processes $\{ W^N_\theta : N \in \N \}$ is tight in the space $D_{ \hat{\mathcal{S}}}[0,\infty)$. Let $W_\theta$ be a limit point of this sequence and let $\hat{X}_\theta$ be the process with generator $\hat{\mathbb{A} }_\theta$ and initial state $v_0$. Then the process $\Pi_{\hat{ \mathcal{S} }}W_\theta$ has the same distribution as the process $\hat{X}_\theta$.
\end{proposition}
\begin{remark}
\label{rem:convergenceofprojection}
Note that this proposition proves that $\Pi_{\hat{ \mathcal{S} }}W^N_\theta \Rightarrow \hat{X}_\theta$ as $N \to \infty$.
\end{remark}
\begin{proof}
The tightness of the sequence of processes $\{ W^N_\theta : N \in \N\}$ is argued in Lemma \ref{mainlemmapp}. Let the process $W_\theta$ be a limit point of this sequence.
For any function $g \in \mathcal{B}( \Pi_2  \mathcal{S}   )$, define another function $f : \hat{ \mathcal{S} }  \to \R$ by
\begin{align*}
f(t,v,z) = g(v).
\end{align*}
Then the function $f$ is in the class $\mathcal{C}$ (see \eqref{classc}) and the action of  $\mathbb{B}^N_\theta$ (see \eqref{genbntheta}) on $f$ is given by
\begin{align*}
\mathbb{B}^N_\theta f(t,v,z) = \sum_{k \in \Gamma_2} \rho^N_{k,\theta}(t,v,z) \left( g(v+ \zeta^s_k) -g(v) \right).
\end{align*}
This shows that the following is a martingale
\begin{align*}
&m^N_g(t) = f(W^N_\theta(t)) -  \sum_{k \in \Gamma_2} \int_{0}^t \rho^N_{k,\theta} ( W^N_\theta(s) ) \left( g(V^N_\theta(s)+ \zeta^s_k) -g( V^N_\theta(s) ) \right)ds \\
& = g\left( \Pi_{\hat{ \mathcal{S} }}W^N_\theta (t) \right) -  \sum_{k \in \Gamma_2} \int_{0}^t \rho^N_{k,\theta} ( W^N_\theta(s) ) \left( g\left(\Pi_{\hat{ \mathcal{S} }}W^N_\theta (s)+ \zeta^s_k\right) -g \left( \Pi_{\hat{ \mathcal{S} }}W^N_\theta (s) \right) \right)ds.
\end{align*}
Since $g$ is bounded, Lemma \ref{limitinglemma}, the continuous mapping theorem and Lemma \ref{mainlemmapp} imply that as $N \to \infty$, we have $m^N_g \Rightarrow m_g$ where
\begin{align*}
m_g(t) =  g\left( \Pi_{\hat{ \mathcal{S} }}W_\theta (t) \right) - \sum_{k \in \Gamma_2} \int_{0}^t \hat{\lambda}_k(  \Pi_{\hat{ \mathcal{S} }}W_\theta (s)  , \theta) \left( g \left( \Pi_{\hat{ \mathcal{S} }}W_\theta (s)+ \zeta^s_k\right) -g\left( \Pi_{\hat{ \mathcal{S} }}W_\theta (s) \right) \right)ds,
\end{align*}
is also a martingale. This shows that $\{ \Pi_{\hat{ \mathcal{S} }}W_\theta (t) : t \geq 0 \}$ satisfies the martingale problem for operator $\hat{ \mathbb{A} }_\theta$ (given by \eqref{finallimitgen}). Moreover $\Pi_{\hat{ \mathcal{S} }}W_\theta (0) = \hat{X}_\theta(0)=v_0 $.
Since the martingale problem for $\hat{ \mathbb{A} }_\theta$ is well-posed, the process $\Pi_{\hat{ \mathcal{S} }}W_\theta $ has the same distribution as the process $\hat{X}_\theta$ and this proves the the proposition.
\end{proof}

\subsection{Proof of Theorem \ref{mainsensitivityresult}} \label{sec:maintheoremproof}

We now have all the tools to prove our main result. But first we need to define some quantities and provide some preliminary results. For any function $f : \hat{\mathcal{S}} \to \R$, $(t_0,v_0,z_0) \in \hat{ \mathcal{S} }$ and $t \geq 0$ define
\begin{align}
\label{defn_psinlim}
\Psi^N_{f ,\theta}(t,t_0,v_0,z_0)= \E\left( f( W^N_\theta(t)  ) \right),
\end{align}
where $\{W^N_\theta(t) : t \geq 0\}$ is the process with generator $\mathbb{B}^N_\theta$ (see \eqref{genbntheta}) and initial state $(t_0,v_0,z_0)$. 
 Similarly for any function $g : \Pi_2 \mathcal{S} \to \R$ define
\begin{align}
\label{defn_psilimit}
\Psi_{g,\theta}(t,v_0)= \E\left( g(  \hat{X}_\theta(t)  ) \right),
\end{align}
where $\{  \hat{X}_\theta(t) : t \geq 0 \}$ is the process with generator $\hat{\mathbb{A} }_\theta $ (see \eqref{finallimitgen}) and initial state $v_0$. 
Now consider a function $f : \mathcal{S} \to \R$ which is polynomially growing with respect to projection $\Pi_2$. Corresponding to this function define $f^N_\theta  : \hat{\mathcal{S}} \to \R$ by \eqref{defn_ftheta} and $f_\theta : \Pi_2 \mathcal{S} \to \R$ by \eqref{defn_fthetalimit}. Remark \ref{rem:functioninclassc} and Lemma \ref{mainlemmapp} imply that for any $T>0$
\begin{align}
\label{mainp:bddnessoff}
\sup_{N \in \N} \sup_{t \in [0,T]} \E\left(  | f^N_\theta( W^N_\theta(t)   )  | \right) < \infty  \quad \textnormal{ and } \quad  \E\left(  \int_0^T | \mathbb{B}^N_\theta f^N_\theta( W^N_\theta(t)   )  | dt  \right) < \infty. 
\end{align}
If $\sigma$ is a stopping time with respect to the filtration generated by $W^N_\theta$, then due to part (E) of Lemma \ref{mainlemmapp} we have
\begin{align}
\label{mainpr:dynkinappl}
\E\left(  \int_{ 0 }^{\sigma \wedge t} \mathbb{B}^N_\theta f(W^N_\theta(s)   ) ds  \right) = \E\left( \Psi^N_{f ,\theta}(\sigma \wedge t,t_0,v_0,z_0) \right) -f(t_0,v_0,z_0).
\end{align}
Proposition \ref{prop_limitresult} shows that the sequence of processes $\{W^N_\theta : N \in \N\}$ is tight and $\Pi_{\hat{ \mathcal{S} }} W^N_\theta \Rightarrow \hat{X}_\theta$ as $N \to \infty$ (see Remark \ref{rem:convergenceofprojection}). This fact along with part (C) of Lemma  \ref{limitinglemma}  proves that for any $T >0$
\begin{align}
\label{mainproof:limitpsi}
\lim_{N \to \infty} \sup_{ t \in [\epsilon_N,T]} \left|\Psi^N_{f^N_\theta ,\theta}(t,t_0,v_0,z_0) - \Psi_{f_\theta ,\theta}(t,v_0) \right| =0 \ \textnormal{and} \
\lim_{N \to \infty} \int_{0}^T \left|\Psi^N_{f^N_\theta ,\theta}(t,t_0,v_0,z_0) - \Psi_{f_\theta ,\theta}(t,v_0) \right|dt =0
,
\end{align}
where $\epsilon_N = 1/\sqrt{N}$.

%Let $\{W^N_\theta(t) : t \geq 0\}$ is the process with generator $\mathbb{B}^N_\theta$ and initial state $(t_0,v_0,z_0) \in \hat{\mathcal{S}}$. 
%Now let the process $\hat{X}_\theta$ be as in the statement of Theorem \ref{mainsensitivityresult}.
Observe that the right side of \eqref{mainresultequation} can be written as 
\begin{align*}
\hat{S}_{\theta}(f_\theta,t) = \frac{ \partial }{\partial \theta} \E\left( f_\theta( \hat{X}_\theta(t)   )  \right) = \lim_{h \to 0} \frac{\E\left( f_{\theta+h}( \hat{X}_{\theta +h}(t)   )  \right) -\E\left( f_\theta( \hat{X}_\theta(t)   )  \right)    }{h},
\end{align*}
 where $\hat{X}_\theta$ and $\hat{X}_{\theta+h}$ are processes with initial state $v_0 = \Pi_2 x_0$ and generators $\hat{\mathbb{A} }_\theta $ and $\hat{\mathbb{A} }_{\theta+h} $ respectively.  This shows that we can write $\hat{S}_{\theta}(f_\theta,t) $ as
 \begin{align}
 \label{proof:simplrhs}
\hat{S}_{\theta}(f_\theta,t)  &= \lim_{h \to 0} \frac{\E\left( f_{\theta+h}( \hat{X}_{\theta +h }(t)   )  \right) -\E\left( f_\theta( \hat{X}_{\theta + h}(t)   )  \right)    }{h} + \lim_{h \to 0} \frac{\E\left( f_{\theta}( \hat{X}_{\theta +h}(t)   )  \right) -\E\left( f_\theta( \hat{X}_\theta(t)   )  \right)    }{h},
\end{align}
provided that the two limits exist. If $\partial f_\theta/ \partial \theta$ is the partial derivative of $f_\theta$ with respect to $\theta$, then for any $v \in \Pi_2 \mathcal{S}$
\begin{align*}
f_{\theta+h}(v) = f_{\theta}(v) + h \frac{\partial f_\theta}{   \partial \theta  } (v) +o(h).
\end{align*}
This shows that the first limit in \eqref{proof:simplrhs} is just
\begin{align}
 \label{proof:simplrhs1}
 \lim_{h \to 0} \frac{\E\left( f_{\theta+h}( \hat{X}_{\theta +h }(t)   )  \right) -\E\left( f_\theta( \hat{X}_{\theta + h}(t)   )  \right)    }{h}  = \E\left(\frac{\partial f_\theta}{   \partial \theta  } ( \hat{X}_{\theta }(t)   )   \right) .
\end{align} 
Using coupling arguments we proved in \cite{Gupta} that the second limit in \eqref{proof:simplrhs} is given by
\begin{align}
 \label{proof:simplrhs2}
 \lim_{h \to 0} \frac{\E\left( f_{\theta}( \hat{X}_{\theta +h}(t)   )  \right) -\E\left( f_\theta( \hat{X}_\theta(t)   )  \right)    }{h} & =  \sum_{k \in \Gamma_2}    \E \left[ \int_{0}^{t }  \frac{\partial  \hat{\lambda}_k(  \hat{X}_\theta(s)  ,\theta )   }{ \partial \theta} 
\left(  f_\theta( \hat{X}_\theta(s) + \zeta^s_k)  -  f_\theta(\hat{X}_\theta(s)) \right)ds \right] \notag  \\
& + \sum_{k \in \Gamma_2}    \E  \left[ \sum_{i = 0, \sigma_{i} <t  }^{\infty} \frac{ \partial \hat{\lambda}_{k}(\hat{X}_\theta(\sigma_{i}  ) ,\theta)  }{ \partial \theta }  R_{k,\theta}( \hat{X}_\theta(\sigma_{i}  ), f_\theta,t - \sigma_i \wedge t,k)   \right].
\end{align}
where $\zeta^s_k = \Pi_2 \zeta_k$, $\sigma_i$ is the $i$-th jump time\footnote{We define $\sigma_{0}=0$ for convenience} of the process $\hat{X}_\theta$ and 
\begin{align}
\label{mainp:defnrktheta}
R_{k,\theta}( x, f,t,k) = \int_{0}^{t}  \left( \Psi_{ f, \theta}(s, x +\zeta^s_k ) -  \Psi_{ f, \theta}(s, x ) - f(x+\zeta^s_k) + f(x) \right)\exp \left(-\hat{\lambda}_0( x,\theta)(t-s)  \right)ds. 
\end{align}
 From \eqref{proof:simplrhs1}, \eqref{proof:simplrhs2}, \eqref{proof:simplrhs} and \eqref{mainresultequation} we see that to prove Theorem \ref{mainsensitivityresult} is suffices to show that
\begin{align}
\label{mainproof:sufficientcondiytion0}
\lim_{N \to \infty} \frac{\partial  }{ \partial \theta } \E \left(  f( X^{N}_{\gamma_2, \theta} (t) ) \right)  & = \E\left(\frac{\partial f_\theta}{   \partial \theta  } ( \hat{X}_{\theta }(t)   )   \right) +  \sum_{k \in \Gamma_2}    \E \left[ \int_{0}^{t }  \frac{\partial  \hat{\lambda}_k(  \hat{X}_\theta(s)  ,\theta )   }{ \partial \theta} 
\left(  f_\theta( \hat{X}_\theta(s) + \zeta^s_k)  -  f_\theta(\hat{X}_\theta(s)) \right)ds \right] \notag  \\
& +\sum_{k \in \Gamma_2}    \E  \left[ \sum_{i = 0, \sigma_{i} <t  }^{\infty}   \frac{ \partial \hat{\lambda}_{k}(\hat{X}_\theta(\sigma_{i}  ) ,\theta)  }{ \partial \theta }  R_{k,\theta}( \hat{X}_\theta(\sigma_{i}  ), f_\theta,t - \sigma_i \wedge t,k)   \right].
\end{align}
We now come to the proof of our main result, where we establish \eqref{mainproof:sufficientcondiytion0}. The arguments used in the proof are motivated by the analysis in \cite{Gupta}.

 \begin{proof}[Proof of Theorem \ref{mainsensitivityresult}] 
 For the initial state $x_0$ let $v_0 = \Pi_2 x_0$ and $z_0 = (I- \Pi_2) x_0$. Let $X^{N}_{\theta}$ and $ X^{N}_{\theta+h}$ be Markov processes with initial state $x_0$ and generators  $\mathbb{A}^{N}_{\gamma_2, \theta}$ and $\mathbb{A}^{N}_{\gamma_2, \theta+h}$ respectively. Similarly let $W^N_\theta$ and $W^{N}_{\theta+h}$ be Markov processes with initial state $(0,v_0,z_0)$ and generators $\mathbb{B}^{N}_\theta$ and $\mathbb{B}^{N}_{\theta+h}$ respectively. 
From part (B) of Proposition \ref{prop_preserve} we know that
\begin{align}
\label{mainproofequiv}
\E\left( f( X^N_\theta (t) ) \right) & = \E\left( f_\theta^N( W^N_\theta(t)  ) \right)  \ \textnormal{ and } \ \E\left( f( X^N_{\theta+h} (t) ) \right) = \E\left( f_{\theta + h}^N( W^N_{\theta+h }(t)   ) \right). 
\end{align}
For any $(t,v,z) \in \hat{ \mathcal{S}}$, $f^N_\theta(t,v,z)$ is a continuously differentiable function of $\theta$. Hence we can write 
\begin{align*}
f^N_{\theta + h}(t,v,z) = f^N_{\theta}(t,v,z) + h \frac{\partial f^N_\theta }{ \partial \theta }(t,v,z)  + o(h).
\end{align*}
This expansion along with \eqref{mainproofequiv} gives us
\begin{align}
\label{snftheta_expn1}
S^N_\theta(f,t) &= \frac{\partial  }{ \partial \theta } \E \left(  f( X^{N}_{\gamma_2, \theta} (t) ) \right) \notag \\
 & = \lim_{h \to 0} \frac{ \E\left( f( X^N_{\theta+h} (t) ) \right) -  \E\left( f( X^N_\theta (t) ) \right)   }{h}  \notag \\
& = \lim_{h \to 0} \frac{ \E\left( f_{\theta + h}^N( W^N_{\theta +h}  (t)   ) \right) -  \E\left( f_\theta^N( W^N_{\theta}  (t)  ) \right)  }{h} \notag \\
& = \lim_{h \to 0} \frac{ \E\left( f_{\theta + h}^N( W^N_{\theta +h}  (t)   ) \right) -  \E\left( f_\theta^N( W^N_{\theta+h}  (t)  ) \right)  }{h} 
+ \lim_{h \to 0} \frac{ \E\left( f_{\theta}^N( W^N_{\theta +h}  (t)   ) \right) -  \E\left( f_\theta^N( W^N_{\theta}  (t)  ) \right)  }{h} \notag \\
& = S^{N,1}_\theta(f,t) + S^{N,2}_\theta(f,t), 
\end{align}
where
\begin{align}
\label{sn1ftheta}
S^{N,1}_\theta(f,t) & = \E \left( \frac{ \partial  f^N_\theta    }{ \partial \theta} (W^N_{\theta}  (t)  )  \right) \\ 
\textnormal{ and } \quad  S^{N,2}_\theta(f,t) & =  \lim_{h \to 0} \frac{ \E\left( f_{\theta}^N(  W^N_{\theta+h}  (t)    ) \right) -  
\E\left( f_\theta^N( W^N_{\theta+h}  (t)  ) \right)  }{h}. 
\end{align}
Proposition \ref{prop_limitresult} shows that the sequence of processes $\{W^N_\theta : N \in \N \}$ is tight and if $W_\theta$ is a limit point then the process $\Pi_{ \hat{ \mathcal{S} }} W_\theta$ has the same distribution as the process 
$\hat{X}_\theta$. This fact along with part (C) of Lemma \ref{limitinglemma} shows that for $t >0$
\begin{align}
\label{snithetalimit}
\lim_{N \to \infty} S^{N,1}_\theta(f,t)  = \E \left( \frac{ \partial f_\theta  }{ \partial \theta}  ( \hat{X}_\theta(t)  )  \right).
\end{align}

In order to compute the limit of $S^{N,2}_\theta(f,t)$ as $N \to \infty$, we will couple the processes $W^N_\theta$ and $W^N_{\theta+h}$ in a special way. We need to define certain quantities to describe the coupling. 
For any $(t_1,v_1,z_1), (t_2,v_2,z_2)  \in \hat{ \mathcal{S}}$ let
 \begin{align*}
\rho^N_{k , \theta,\textnormal{min}} (t_1,v_1,z_1,t_2,v_2,z_2,h) & = \rho^N_{k,\theta } (t_1,v_1,z_1) \wedge \rho^N_{k,\theta+h} (t_2,v_2,z_2), \\
r^{N,1}_{k,\theta } (t_1,v_1,z_1,t_2,v_2,z_2,h) & =\rho^N_{k,\theta}  (t_1,v_1,z_1) -  \rho^N_{k , \theta,\textnormal{min}} (t_1,v_1,z_1,t_2,v_2,z_2,h)  \\
\textnormal{ and } \quad  r^{N,2}_{k,\theta} (t_1,v_1,z_1,t_2,v_2,z_2,h)   &=\rho^N_{k ,\theta+h } (t_2,v_2,z_2)-  \rho^N_{k , \theta,\textnormal{min}} (t_1,v_1,z_1,t_2,v_2,z_2,h).
\end{align*}
We define the processes $V^N_{\theta}$ and $V^N_{\theta+h}$ by the following 
random time change representations 
\begin{align}
\label{mainproof:rtcrep1}
V^N_\theta(t) &= v_0 + \sum_{k \in \Gamma_2} Y_k \left(  \int_{0}^t \rho^N_{k , \theta, \textnormal{min}} \left( \tau^N_\theta(s), V^N_\theta(s), Z^N_\theta(s) , \tau^N_{\theta +h}(s), V^N_{\theta +h }(s), Z^N_{\theta +h}(s), h \right) ds \right)  \zeta^s_k \notag  \\
& +  \sum_{k \in \Gamma_2} Y^{(1)}_k \left(  \int_{0}^t r^{N ,1}_{k,\theta} \left( \tau^N_\theta(s), V^N_\theta(s), Z^N_\theta(s) , \tau^N_{\theta +h}(s), V^N_{\theta +h }(s), Z^N_{\theta +h}(s), h \right)  ds \right) \zeta^s_k  \\
\label{mainproof:rtcrep2}
V^N_{\theta+h}(t) &= v_0 + \sum_{k \in \Gamma_2} Y_k \left(  \int_{0}^t \rho^N_{k , \theta, \textnormal{min}} \left( \tau^N_\theta(s), V^N_\theta(s), Z^N_\theta(s) , \tau^N_{\theta +h}(s), V^N_{\theta +h }(s), Z^N_{\theta +h}(s), h \right)ds \right) \zeta^s_k \notag \\
& +  \sum_{k \in \Gamma_2} Y^{(2)}_k \left(  \int_{0}^t r^{N ,2}_{k,\theta} \left( \tau^N_\theta(s), V^N_\theta(s), Z^N_\theta(s) , \tau^N_{\theta +h}(s), V^N_{\theta +h }(s), Z^N_{\theta +h}(s), h \right) ds \right) \zeta^s_k,
\end{align}
where $\{  Y_k,  Y^{(1)}_k ,Y^{(2)}_k   : k \in \Gamma_2 \}$ is a family of independent unit rate Poisson processes. To $V^N_{\theta}$ ($V^N_{\theta+h}$) we associate processes $\tau^N_\theta$ ($\tau^N_{\theta+h}$)  and $Z^N_\theta$ ($Z^N_{\theta+h}$)  as in Remark \ref{rem:construction}. The above representations couple the processes $V^N_{\theta}$ and $V^N_{\theta+h}$. For each $i\in \N$, let $\sigma_{i}^1$ ($\sigma_{i}^{2}$) be the $i$-th jump time of the process $V^N_\theta$ ($V^N_{\theta+h}$ ) and let $\eta_{i}^{1}$ ($\eta_{i}^{2}$) be the jump direction of the process $V^N_\theta$ ($V^N_{\theta+h}$) at time $\sigma_{i}^{1}$ ($\sigma_{i}^{2}$). Define $\sigma^1_0 = \sigma^2_0 = 0$.
Fix a sequence $\{ u_i : i \in \N \}$ of independent $\textnormal{Unif}(0,1)$ random numbers. 
We couple the processes $Z^N_\theta$ and $Z^N_{\theta+h}$, by letting $Z^N_\theta(\sigma_{i}^{1}  ) = \digamma^N_{\eta_{i}^{1},\theta}( \sigma_{i}^{1} -\sigma_{i-1}^{1} ,  V^N_\theta( \sigma_{i-1}^{1})  , Z^N_\theta( \sigma_{i-1}^{1} ) ,u_i)$ and $Z^N_{\theta+h}(\sigma^{2}_i) = \digamma^N_{\eta^2_i,\theta +h}( \sigma^{2}_i -\sigma^{2}_{i-1} ,  V^N_{\theta+h}( \sigma^{2}_{i-1})  , Z^N_{\theta+h}( \sigma^{2}_{i-1} ) ,u_i)$ for each $i$, where the function $\digamma^N$ is defined by \eqref{defndigamma}. Note that we are using the same $u_i$ in the definition of $Z^N_\theta(\sigma_{i}^{1}  ) $ and $Z^N_{\theta+h}(\sigma^{2}_i) $. 
Define $W^N_{\theta}$ and $W^N_{\theta+h}$ by 
\begin{align*}
W^N_{\theta}(t) = \left(  \tau^N_{\theta}(t) ,  V^N_{\theta}(t) ,  Z^N_{\theta}(t)\right)  \quad \textnormal{ and } \quad  W^N_{\theta+h}(t) = \left(  \tau^N_{\theta+h}(t) ,  V^N_{\theta +h}(t) ,  Z^N_{\theta+h}(t)\right) \textnormal{ for all } t \geq 0.
\end{align*}
One can verify that the processes $W^N_\theta$ and $W^N_{\theta+h}$ have initial state $(0,v_0,z_0)$ and generators $\mathbb{B}^N_\theta$ and $\mathbb{B}^N_{\theta+h}$ respectively. 

Let $\gamma^N_h$ be the stopping time given by
\begin{align}
\label{defn_gamnh}
\gamma^N_h = \inf \{ t \geq 0 : W^N_\theta(t) \neq W^N_{\theta+h}(t)   \}.
\end{align}
Then the coupling of processes $W^N_\theta$ and $ W^N_{\theta+h}$ ensures that $\gamma^N_h \to \infty$ a.s. as $h \to 0$. Define
\begin{align}
\label{defnantheta}
A^N_\theta & = \lim_{h \to 0} \frac{1}{h} \E\left[  \int_{0}^{t \wedge \gamma^N_h}  \left( \mathbb{B}^N_{\theta+h} f_{\theta}^N( W^N_{\theta+h }(s) )  -  \mathbb{B}^N_{\theta} f_\theta^N( W^N_\theta(s)  )  \right)  ds \right] \\
\textnormal{ and } \quad
\label{defnbntheta}
B^N_\theta  & = \lim_{h \to 0} \frac{1}{h} \E\left[  \int_{t \wedge \gamma^N_h}^t  \left( \mathbb{B}^N_{\theta+h}  f_{\theta}^N( W^N_{\theta+h }(s) )   -  \mathbb{B}^N_{\theta}  f_{\theta}^N( W^N_{\theta }(s) )  \right)  ds \right].  
\end{align}
Note that $f^N_\theta(0,v_0,z_0) = f(x_0)$. Using \eqref{mainpr:dynkinappl} we can write
\begin{align*}
\E\left( f_\theta^N( W^N_\theta(t)  ) \right) & =f(x_0)+ \E \left( \int_{0}^t \mathbb{B}^N_\theta f^N_\theta ( W^N_{\theta  }(s) )  ds \right) \\
\textnormal{ and } \quad  \E\left( f_{\theta }^N( W^N_{\theta +h }(t)  ) \right) & =f(x_0) + \E \left( \int_{0}^t \mathbb{B}^N_{\theta +h} f^N_{\theta }( W^N_{\theta +h }(s)  ) ds \right).
\end{align*}
Therefore
\begin{align}
\label{defnsnitheraf}
S^{N,2}_\theta(f,t) & = \lim_{h \to 0} \frac{ \E\left( f_{\theta}^N( W^N_{\theta+h }(t)   ) \right) -  \E\left( f_\theta^N( W^N_\theta(t)  ) \right)  }{h} \notag \\
& =   \lim_{h \to 0} \frac{1}{h} \left[ \E\left(  \int_{0}^t  \left( \mathbb{B}^N_{\theta+h} f_{\theta}^N( W^N_{\theta+h }(s)   ) -   \mathbb{B}^N_{\theta} f_\theta^N( W^N_\theta(s) ) \right) ds \right) \right]  \notag \\
& = A^N_\theta  + B^N_\theta.
\end{align}
Using Taylor's expansion, for any $f \in \mathcal{C}$ and $(t,v,z) \in  \hat{ \mathcal{S}}$ we get
\begin{align}
\label{diffinbntheta}
& \mathbb{B}^N_{\theta +h  }f(t,v,z) - \mathbb{B}^N_{\theta }f(t,v,z)  \notag  \\
& =  \sum_{k \in \Gamma_2} \rho^N_{k,\theta+h}(t,v,z) 
  \sum_{e \in \mathbb{H}_v } \left( f(0,v+ \zeta^s_k, e+ \zeta^f_k ) -f(t,v,z) \right)  \Theta^N_{k,\theta+h}(t,v,z,e) \notag  \\
  & - \sum_{k \in \Gamma_2} \rho^N_{k,\theta}(t,v,z) 
  \sum_{e \in \mathbb{H}_v } \left( f(0,v+ \zeta^s_k, e+ \zeta^f_k ) -f(t,v,z) \right)  \Theta^N_{k,\theta}(t,v,z,e)\notag  \\
  & = \sum_{k \in \Gamma_2}  \sum_{e \in \mathbb{H}_v } f(0,v+ \zeta^s_k, e+ \zeta^f_k )  \left( \rho^N_{k,\theta+h} (t,v,z)   \Theta^N_{k,\theta+h}(t,v,z,e)  - 
  \rho^N_{k,\theta}(t,v,z) \Theta^N_{k,\theta}(t,v,z,e)  \right) \notag  \\
  & -  \sum_{k \in \Gamma_2}f(t,v,z) \left(  \rho^N_{k ,\theta+h} (t,v,z) -\rho^N_{k,\theta }(t,v,z)   \right)  \notag  \\
  & = \sum_{k \in \Gamma_2}  \sum_{ e \in \mathbb{H}_v }  f(0,v+ \zeta^s_k, e+ \zeta^f_k )  \left(   \frac{   \partial \rho^N_{k,\theta }(t,v,z)  }{ \partial \theta} \Theta^N_{k,\theta}(t,v,z,e)    +  \rho^N_{k,\theta}(t,v,z)  \frac{\partial  \Theta^N_{k,\theta} (t,v,z,e)    }{ \partial \theta} \right) h \notag  \\
  & -  \sum_{k \in \Gamma_2}f(t,v,z)  \frac{ \partial  \rho^N_{k,\theta}(t,v,z)  }{ \partial \theta} h +o(h)  \notag  \\
    & =   \sum_{k \in \Gamma_2}  \frac{\partial \rho^N_{k,\theta} (t,v,z)  }{ \partial \theta}\left(   \sum_{e \in \mathbb{H}_v } f(0,v+ \zeta^s_k, e+ \zeta^f_k )\Theta^N_{k,\theta}(t,v,z,e)  - f(t,v,z) \right)   h \notag   \\
    & + \sum_{k \in \Gamma_2} \rho^N_{k,\theta} (t,v,z) \sum_{e \in \mathbb{H}_v }   f(0,v+ \zeta^s_k, e + \zeta^f_k )  \frac{\partial  \Theta^N_{k,\theta}(t,v,z,e)    }{ \partial \theta}  h +o(h).
\end{align}
Note that for any $t \in [0, \gamma^N_h)$ we have $W^N_{\theta+h }(t) = W^N_{\theta}(t)$. Relation \eqref{diffinbntheta} implies that
\begin{align*}
&\lim_{N \to \infty} A^N_\theta \\ & = \lim_{N \to \infty} \lim_{h \to 0} \frac{1}{h} \E\left[  \int_{0}^{t \wedge \gamma^N_h}  \left( \mathbb{B}^N_{\theta+h} f_{\theta}^N( W^N_{\theta+h }(s) )  -  \mathbb{B}^N_{\theta} f_\theta^N( W^N_\theta(s)  )  \right)  ds \right] \\
& = \lim_{N \to \infty} \sum_{k \in \Gamma_2 } \E \left[ \int_{0}^t \frac{\partial \rho^N_{k,\theta} (W^N_\theta(s) )  }{ \partial \theta}\left(   \sum_{e \in \mathbb{H}_v } f_{\theta}^N(0,\Pi_{\hat{ \mathcal{S} }} W^N_\theta(s)+ \zeta^s_k, e+ \zeta^f_k )\Theta^N_{k,\theta}(W^N_\theta(s) ,e)  - f_{\theta}^N(W^N_\theta(s) ) \right) ds \right] \\
&+ \lim_{N \to \infty} \sum_{k \in \Gamma_2 } \E\left[ \int_{0}^t \rho^N_{k,\theta} (W^N_\theta(s)) \sum_{e \in \mathbb{H}_v }   f_{\theta}^N(0, \Pi_{\hat{ \mathcal{S} }} W^N_\theta(s)+ \zeta^s_k, e + \zeta^f_k )  \frac{\partial  \Theta^N_{k,\theta}(W^N_\theta(s),e)    }{ \partial \theta} ds   \right].
\end{align*}
Proposition \ref{prop_limitresult} shows that the sequence of processes $\{W^N_\theta : N \in \N \}$ is tight and if $W_\theta$ is a limit point then the process $\Pi_{ \hat{ \mathcal{S} }} W_\theta$ has the same distribution as the process $\hat{X}_\theta$. This fact along with Lemma  \ref{limitinglemma}  implies that  
\begin{align}
\label{limitofantheta}
\lim_{N \to \infty} A^N_\theta & =
 \sum_{k \in \Gamma_2}    \E \left[ \int_{0}^{t }  \frac{\partial  \hat{\lambda}_k(  \hat{X}_\theta(s)  ,\theta )   }{ \partial \theta} 
\left(  f_\theta( \hat{X}_\theta(s) + \zeta^s_k)  -  f_\theta(\hat{X}_\theta(s)) \right)ds \right].
\end{align}
    
Our next goal is to compute $\lim_{N \to \infty} B^N_\theta$. Recall the definitions of $\Psi^N_{f ,\theta}$ and $\Psi_{f ,\theta}$ from \eqref{defn_psinlim} and \eqref{defn_psilimit} respectively. For $i=1,2$, let $(t_i,v_i,z_i) \in \hat{\mathcal{S}}$. Define an event
\begin{align}
\label{mainproof:maineventusedfortheproof}
E^N(t_1,v_1,z_1,t_2,v_2,z_2,s) = \{ W^N_{\theta}(\gamma^N_h ) = (t_1,v_1,z_1), W^N_{\theta+h}(\gamma^N_h ) =(t_2,v_2,z_2) \textnormal{ and } \gamma^N_h = s  \}
\end{align}
and let
\begin{align}
\label{mainp:defrnknew}
&R^N_{\theta, h }(t_1,v_1,z_1,t_2,v_2,z_2,s,t)  \notag \\
&=  \E\left[  \int_{t \wedge \gamma^N_h }^t  \left( \mathbb{B}^N_{\theta+h} f_{\theta}^N( W^N_{\theta+h }(u)  )  -  \mathbb{B}^N_{\theta} f_\theta^N( W^N_\theta(u)  )  \right)  du  \middle\vert E^N(t_1,v_1,z_1,t_2,v_2,z_2,s) \right].
\end{align}
Let $\epsilon_N = 1/\sqrt{N}$.
From \eqref{mainpr:dynkinappl} and the strong Markov property, we can deduce that for any $0<s < t$
\begin{align}
\label{limitwithdynkin1}
&\lim_{N \to \infty} \lim_{h \to 0} R^N_{\theta, h }(t_1,v_1,z_1,t_2,v_2,z_2,s,t)  \notag \\
&= \lim_{N \to \infty} \lim_{h \to 0} \E\left[  \int_{t \wedge \gamma^N_h }^t  \left( \mathbb{B}^N_{\theta+h} f_{\theta}^N( W^N_{\theta+h }(u)  )  -  \mathbb{B}^N_{\theta} f_\theta^N( W^N_\theta(u)  )  \right)  du  \middle\vert E^N(t_1,v_1,z_1,t_2,v_2,z_2,s) \right] \notag \\
& = \lim_{N \to \infty} \lim_{h \to 0} \E\left[  \int_{s + \epsilon_N }^t  \left( \mathbb{B}^N_{\theta+h} f_{\theta}^N( W^N_{\theta+h }(u)  )  -  \mathbb{B}^N_{\theta} f_\theta^N( W^N_\theta(u)  )  \right)  du  \middle\vert E^N(t_1,v_1,z_1,t_2,v_2,z_2,s) \right] \notag \\
& = \lim_{N \to \infty} \lim_{h \to 0} \left[  \Psi^N_{f^N_\theta,\theta+h}( t - s  ,t_2, v_2,z_2 ) -  \Psi^N_{f^N_\theta,\theta+h}(\epsilon_N ,t_2, v_2,z_2 )- \Psi^N_{f^N_\theta,\theta}( t - s, t_1,v_1,z_1 ) + \Psi^N_{f^N_\theta,\theta}( \epsilon_N, t_1,v_1,z_1 )  \right]  \notag \\
& =  \Psi_{f_\theta ,\theta}(t-s,v_2) -\Psi_{f_\theta ,\theta}(t-s,v_1) - f_\theta(v_2) + f_\theta(v_1), 
\end{align}
where the last equality is due to \eqref{mainproof:limitpsi}.

Recall the random time change representations \eqref{mainproof:rtcrep1} and \eqref{mainproof:rtcrep2}. 
For each $i \in \N$, let $\sigma^N_i$ be the $i$-th jump time of the process $C^N_\theta$ defined by
\begin{align*}
C^N_\theta(t) = \sum_{k \in \Gamma_2} Y_k \left(  \int_{0}^{t} \rho^N_{k , \theta, \textnormal{min}} \left( \tau^N_\theta(s),V^N_\theta(s), Z^N_\theta(s) ,\tau^N_{\theta +h}(s), V^N_{\theta +h }(s), Z^N_{\theta +h}(s),   h \right) ds \right)  \zeta^s_k. 
\end{align*}      
Set $\sigma^N_0 = 0$ and note that $\gamma^N_h > \sigma^N_0$.  For each $i \in \N$ define
\begin{align*}
B^{N,1}_{\theta ,i } & = \lim_{h \to 0} \frac{1}{h} \E\left[  \ind_{ \{ \sigma^N_i = \gamma^N_h  \} }\int_{t \wedge \gamma^N_h }^t  \left( \mathbb{B}^N_{\theta+h} f_{\theta}^N( W^N_{\theta+h }(s) )  -  \mathbb{B}^N_{\theta} f_\theta^N( W^N_\theta(s)   )  \right)  ds \right]  \\
\textnormal{ and } \quad  B^{N,2}_{\theta ,i } & = \lim_{h \to 0} \frac{1}{h} \E\left[  \ind_{ \{ \sigma^N_{i-1} < \gamma^N_h < \sigma^N_{i}  \} }\int_{t \wedge \gamma^N_h  }^t  \left( \mathbb{B}^N_{\theta+h} f_{\theta}^N( W^N_{\theta+h }(s) )  -  \mathbb{B}^N_{\theta} f_\theta^N( W^N_\theta(s)  )  \right)  ds \right] . 
\end{align*}
Since $\ind_{ \{ \sigma^N_{i-1} \leq \gamma^N_h  < \sigma^N_{i}   \} } = \ind_{ \{ \sigma^N_{i-1}  = \gamma^N_h    \} }+ \ind_{ \{ \sigma^N_{i-1}  < \gamma^N_h < \sigma^N_{i}   \} } $ we can write
\begin{align}
\label{splittingofbntheta}
B^N_\theta & = \lim_{h \to 0} \frac{1}{h} \E\left[  \int_{t \wedge \gamma^N_h}^t  \left( \mathbb{B}^N_{\theta+h} f_{\theta}^N( W^N_{\theta+h }(s)  )  -  \mathbb{B}^N_{\theta} f_\theta^N( W^N_\theta(s) )  \right)  ds \right]  \notag \\
& =\sum_{i=1}^\infty \lim_{h \to 0} \frac{1}{h} \E\left[  \ind_{ \{ \sigma^N_{i-1} \leq \gamma^N_h < \sigma^N_{i}   \} }\int_{t \wedge \gamma^N_h }^t  \left( \mathbb{B}^N_{\theta+h} f_{\theta}^N( W^N_{\theta+h }(s)  )  -  \mathbb{B}^N_{\theta} f_\theta^N( W^N_\theta(s)  )  \right)  ds \right] \notag   \\
& = \sum_{i=1}^\infty ( B^{N,1}_{\theta ,i } +  B^{N,2}_{\theta ,i } ). 
\end{align}

We now show that the term $B^{N,1}_{\theta ,i }$ converges to $0$ as $N \to \infty$. 
Note that the event $\{\sigma^N_i = \gamma^N_h   \}$ occurs if and only if the event
$\{Z^N_\theta(\sigma^N_i  - ) \neq Z^N_{\theta+h}(\sigma^N_i - ), V^N_\theta(\sigma^N_{i-1}   ) = V^N_{\theta+h}(\sigma^N_{i-1} )   ,Z^N_\theta(\sigma^N_{i-1}   ) = Z^N_{\theta+h}(\sigma^N_{i-1} )  \} $ occurs.
Let $\eta^N_i$ be the $\Gamma_2$-valued random variable which gives the direction of the jump in $C^N_\theta$ at time $\sigma^N_i$. Pick a $\delta \geq0$, $v \in \Pi_2 \mathcal{S}$, $z \in \mathbb{H}_v$ and $k \in \Gamma_2$. Define an event 
\begin{align*}
 L_i(\delta,v,z,k)  =\left\{ \gamma^N_h  \geq \sigma^N_i , (\sigma^N_i -\sigma^N_{i-1} ) =\delta , \eta^N_i = k,  V^N_{\theta}( \sigma^N_{i-1} ) =V^N_{\theta+h}( \sigma^N_{i-1} ) =v , Z^N_{\theta}( \sigma^N_{i-1} ) =Z^N_{\theta+h}( \sigma^N_{i-1} ) =z\right\}.
\end{align*} 
Conditioned on this event, $Z^N_\theta(\sigma^N_i  - )  = \digamma^N_{k,\theta}(t,v,z,u_i)  $ and $Z^N_{\theta+h}(\sigma^N_i - )  = \digamma^N_{k,\theta+h}(t,v,z,u_i)$ where the function $\digamma^N_{k,\theta}$ is given by \eqref{defndigamma}. For any distinct $z_1,z_2 \in \mathbb{H}_v$ define
\begin{align*}
G^N_\theta (z_1,z_2 , \delta,v,z,k) = \lim_{h \to 0} \frac{ \P\left(\sigma^N_i = \gamma^N_h ,  Z^N_\theta(\sigma^N_i  - )  =z_1 \textnormal{ and }  Z^N_{\theta+h}(\sigma^N_i  - )  =z_2 \vert   L_i(\delta,v,z,k)   \right) }{h}.
\end{align*}
Lemma \ref{lemma_decoupling} ensures that $G^N_\theta (z_1,z_2 , \delta,v,z,k) $ exists and 
\begin{align*}
G^N_\theta (z_1,z_2 , \delta,v,z,k) \leq \sum_{e \in \mathbb{H}_v  } \left|  \frac{ \partial  \Theta^N_{k,\theta}( \delta,v,z,e) }{ \partial \theta } \right|. 
\end{align*}
Assumptions \ref{assforsensitivityresult} imply that the right hand side is a polynomially growing function with respect to projection $\Pi_{ \hat{ \mathcal{S} }}$ (see Definition \ref{polynomialgrowth}). Given the events $ L_i(\delta,v,z,k)$ and $\{Z^N_\theta(\sigma^N_i  - )  =z_1 ,  Z^N_{\theta+h}(\sigma^N_i  - )  =z_2\}$ we have      
  \begin{align*}
\left( \tau^N_\theta( \gamma^N_h ), V^N_\theta( \gamma^N_h ),Z^N_\theta( \gamma^N_h ) ,  \tau^N_{\theta+h}( \gamma^N_h ), V^N_{\theta +h}( \gamma^N_h ),Z^N_{\theta +h}( \gamma^N_h )\right) = (0,v+ \zeta^s_k, z_1 +\zeta^f_k,0,v+ \zeta^s_k, z_2 +\zeta^f_k ).
\end{align*}
Recall the definition of $R^N_{\theta, h }$ from \eqref{mainp:defrnknew}.
For any $\delta < s < t$ we can write
\begin{align}
\label{mainproof:thisoneis01}
&\lim_{N \to \infty} \lim_{h \to 0} \frac{1}{h} \E\left[  \ind_{ \{ \sigma^N_i = \gamma^N_h  \} }\int_{t \wedge \gamma^N_h }^t  \left( \mathbb{B}^N_{\theta+h} f_{\theta}^N( W^N_{\theta+h }(s)  )  -  \mathbb{B}^N_{\theta} f_\theta^N( W^N_\theta(s) )  \right)  ds \middle\vert  L_i(\delta,v,z, k) ,  \sigma^N_{i-1} =s -\delta \right] \notag \\
& = \lim_{N \to \infty} \lim_{h \to 0} \sum_{ z_1 \neq z_2 \in \mathbb{H}_v  }  G^N_\theta (z_1,z_2 , \delta,v,z,k)  R^N_{\theta, h }( 0,v+ \zeta^s_k, z_1 +\zeta^f_k,0,v+ \zeta^s_k, z_2 +\zeta^f_k,s,t).
\end{align}
Using \eqref{limitwithdynkin1} we see that
\begin{align}
\label{mainproof:thisoneis0}
\lim_{N \to \infty} \lim_{h \to 0}R^N_{\theta, h }( 0,v+ \zeta^s_k, z_1 +\zeta^f_k,0,v+ \zeta^s_k, z_2 +\zeta^f_k,s,t)  = 0.
\end{align}
This relation along with \eqref{mainproof:thisoneis01} implies that
\begin{align}
\label{limitofbn1}
\lim_{N \to \infty} B^{N,1}_{\theta ,i } = 0.
\end{align}

Recall the random time change representations \eqref{mainproof:rtcrep1} and \eqref{mainproof:rtcrep2}. On the event $\{ \sigma^N_{i-1} < \gamma^N_h < \sigma^N_i \}$, the process $V^N_\theta$
(or $V^N_{\theta+h}$) jumps at time $\gamma^N_h$ due to a jump in the Poisson process $Y^{(1)}_k$ (or $Y^{(2)}_k$) for some $k \in \Gamma_2$. Let $\eta$ be the $\Gamma_2$-valued random variable which gives the direction of the jump in $V^N_\theta$ or $V^N_{\theta+h}$ at time $\gamma^N_h$. Define a random variable
\begin{align*}
\alpha^N_i =  ( \sigma^N_{i} -\sigma^N_{i-1} ) \wedge  ( \gamma^N_h -\sigma^N_{i-1} )
\end{align*}
and an event 
\begin{align*}
H_i(s,v,z)  =\left\{ \sigma^N_{i-1}  =s , V^N_{\theta}( \sigma^N_{i-1} ) =V^N_{\theta+h}( \sigma^N_{i-1} ) =v , Z^N_{\theta}( \sigma^N_{i-1} ) =Z^N_{\theta+h}( \sigma^N_{i-1} ) =z\right\},
\end{align*}
for $s \geq 0$, $v \in \Pi_2 \mathcal{S}$ and $z \in \mathbb{H}_v$. The event $\{ \sigma^N_{i-1} < \gamma^N_h < \sigma^N_i \}$ is equivalent to the event $\{ \gamma^N_h > \sigma^N_{i-1},  \alpha^N_i =  ( \gamma^N_h -\sigma^N_{i-1} )\}$.
Given $ \gamma^N_h > \sigma^N_{i-1} $ and $H_i(s,v,z) $, the density of the $\R_+$-valued random variable $\alpha^N_i$ on the event $\{ \eta = k,\alpha^N_i =  ( \gamma^N_h -\sigma^N_{i-1} )\}$ is given by
 \begin{align}
 \label{densityofalpha}
& \lim_{\epsilon \to 0}\frac{ \P\left( \alpha^N_i \in (t,t+\epsilon)  , \eta = k,  \alpha^N_i =  ( \gamma^N_h -\sigma^N_{i-1} )   \middle\vert H_i(s,v,z) , \gamma^N_h > \sigma^N_{i-1}  \right) }{\epsilon} 
\notag \\&= \left( \rho^N_{k,\theta}(t,v,z) +\rho^N_{k,\theta+h}(t, v,z ) - 2 \rho^N_{k,\theta}( t,v,z) \wedge \rho^N_{k,\theta+h}(t,v,z)\right) \notag \\ & \times \exp\left[  -\int_{0}^t  \left( \rho^N_{0,\theta}(u,v,z) +\rho^N_{0,\theta+h}(u,v,z ) - 2 \rho^N_{0,\theta}(u,v,z) \wedge \rho^N_{0,\theta+h}(u,v,z)\right) du\right] \notag \\
&= h \left| \frac{ \partial \rho^N_{k,\theta}(t,v,z)  }{ \partial \theta } \right|\exp\left( -\int_{0}^t  \rho^N_{0,\theta}(u,v,z) du\right) + o(h).
\end{align}
 On the event $H_i(s,v,z)  \cap \{\gamma^N_h > \sigma^N_{i-1}, \eta =k , \alpha^N_i =  ( \gamma^N_h -\sigma^N_{i-1} ) = \delta\}$ we have      
  \begin{align*}
\left( W^N_\theta( \gamma^N_h ), W^N_{\theta +h}( \gamma^N_h )\right) = 
\left\{
\begin{tabular}{cc}
$( \delta, v, z, 0, v + \zeta^s_k ,\xi_2 +\zeta^f_k  )$ & $\rho^N_{k,\theta+h}( \delta,v,z) > \rho^N_{k,\theta}( \delta,v,z)$  \\
$( 0, v + \zeta^s_k ,\xi_1 +\zeta^f_k, \delta, v ,z)$ & $\rho^N_{k,\theta+h}( \delta,v,z) < \rho^N_{k,\theta}( \delta,v,z)$,
\end{tabular} \right.
\end{align*}
where $\xi_1 = \digamma^N_{k,\theta}(\delta,v,z,u_i)$ and $\xi_2=\digamma^N_{k,\theta+h}(\delta,v,z,u_i)$ are $\mathbb{H}_v$-valued random variables with distributions $\Theta^N_{k,\theta}(\delta,v,z,\cdot)$ and $\Theta^N_{k,\theta+h}(\delta,v,z,\cdot)$ respectively. For small values of $h$, $\partial \rho^N_{k,\theta}(\delta,v,z) / \partial \theta  >0$ implies that $\rho^N_{k,\theta+h}( \delta,v,z) > \rho^N_{k,\theta}( \delta,v,z)$ and similarly $\partial \rho^N_{k,\theta}(\delta,v,z) / \partial \theta  < 0$ implies that $\rho^N_{k,\theta+h}( \delta,v,z) < \rho^N_{k,\theta}( \delta,v,z)$. Using the density of $\alpha^N_i$ on the event $\{ \eta = k,\alpha^N_i =  ( \gamma^N_h -\sigma^N_{i-1} )\}$ 
(see \eqref{densityofalpha}) we obtain
\begin{align}
\label{mainproof:mainlimitingrelation007}
&\lim_{N \to \infty} \lim_{h \to 0} \frac{1}{h} \E \left[ \ind_{ \{ \sigma^N_{i-1} < \gamma^N_h < \sigma^N_i \} } \int_{t \wedge \gamma^N_h}^t  \left( \mathbb{B}^N_{\theta+h} f_{\theta}^N( W^N_{\theta+h }(u) )  -  \mathbb{B}^N_{\theta} f_\theta^N( W^N_\theta(u)   )  \right)  du  \middle \vert H_i(s,v,z), \gamma^N_h > \sigma^N_{i-1}\right] \notag \\
&= \lim_{N \to \infty} \lim_{h \to 0} \sum_{z_2 \in  \mathbb{H}_v  }\sum_{k \in \Gamma_2} \int_{0}^{t-s} \left[ \frac{ \partial \rho^N_{k,\theta}(\delta,v,z)  }{ \partial \theta } \right]^{+} \exp{ \left(-\int_{0}^\delta  \rho^N_{0,\theta}(u,v,z) du \right) }  \notag \\ &  \quad \quad  \quad \quad \quad  \quad \quad \quad  \quad \quad \quad  \quad         \times R^N_{\theta,h}( \delta, v, z, 0, v + \zeta^s_k ,z_2 +\zeta^f_k ,s+\delta,t ) \Theta^N_{k,\theta+h}(\delta,v,z,z_2)  d\delta \notag \\
&+\lim_{N \to \infty} \lim_{h \to 0} \sum_{z_1 \in  \mathbb{H}_v  } \sum_{k \in \Gamma_2} \int_{0}^{t-s} \left[ \frac{ \partial \rho^N_{k,\theta}(\delta,v,z)  }{ \partial \theta } \right]^{-}  \exp{ \left(-\int_{0}^\delta  \rho^N_{0,\theta}(u,v,z) du \right) } \notag  \\ &  \quad \quad  \quad \quad \quad  \quad \quad \quad  \quad \quad \quad  \quad         \times
R^N_{\theta,h}(  0, v + \zeta^s_k ,z_1 +\zeta^f_k, \delta, v ,z,s+\delta,t ) \Theta^N_{k,\theta+h}(\delta,v,z,z_1)  d\delta .
\end{align}  
From \eqref{limitwithdynkin1} one can verify that
\begin{align}
\label{mainp:lastrelation1}
&\lim_{N \to \infty} \lim_{h \to 0} R^N_{\theta,h}( \delta, v, z, 0, v + \zeta^s_k ,z_2 +\zeta^f_k, s+\delta,t )   = - \lim_{N \to \infty} \lim_{h \to 0} R^N_{\theta,h}(  0, v + \zeta^s_k ,z_1 +\zeta^f_k, \delta, v ,z,s+\delta,t )   \notag \\
&=  \Psi_{f_\theta ,\theta}(t-s -\delta,v + \zeta^s_k) -\Psi_{f_\theta ,\theta}(t-s-\delta,v) - f_\theta(v + \zeta^s_k) + f_\theta(v).
\end{align}
Using part (A) of Lemma \ref{limitinglemma}, \eqref{mainp:lastrelation1} and \eqref{mainproof:mainlimitingrelation007} we can conclude that
\begin{align}
\label{mainp:lastrelation2}
& \lim_{N \to \infty} \lim_{h \to 0} \frac{1}{h} \E \left[ \ind_{ \{ \sigma^N_{i-1} < \gamma^N_h < \sigma^N_i \} } \int_{t \wedge \gamma^N_h}^t  \left( \mathbb{B}^N_{\theta+h} f_{\theta}^N( W^N_{\theta+h }(u) )  -  \mathbb{B}^N_{\theta} f_\theta^N( W^N_\theta(u)   )  \right)  du  \middle \vert H_i(s,v,z), \gamma^N_h > \sigma^N_{i-1}\right] \notag \\
& = \sum_{k \in \Gamma_2} \int_{0}^{t-s} \frac{ \partial \hat{ \lambda }_k(v,\theta)}{ \partial \theta } \exp\left( -\hat{ \lambda }_0(v,\theta) \delta \right)  \left(  \Psi_{f_\theta ,\theta}(t-s -\delta,v + \zeta^s_k) -\Psi_{f_\theta ,\theta}(t-s-\delta,v)  - f_\theta(v + \zeta^s_k) + f_\theta(v) \right) d\delta \notag  \\
& = \sum_{k \in \Gamma_2} \int_{0}^{t-s} \frac{ \partial \hat{ \lambda }_k(v,\theta)}{ \partial \theta } \exp\left( -\hat{ \lambda }_0(v,\theta) (t -s -u) \right)  \left(  \Psi_{f_\theta ,\theta}( u,v + \zeta^s_k) -\Psi_{f_\theta ,\theta}( u,v)  - f_\theta(v + \zeta^s_k) + f_\theta(v) \right) du,
\end{align}
where $\hat{ \lambda }_0(v,\theta) = \sum_{k \in \Gamma_2} \hat{ \lambda }_k(v,\theta)$. Due to our coupling, as $h \to 0$, the process $W^N_{\theta+h}$ converges a.s. to the process $W^N_\theta$ and hence $\gamma^N_h \to \infty$ a.s. 
Proposition \ref{prop_limitresult} and Remark \ref{rem:convergenceofprojection} show that as $N \to \infty$ we have $V^N_\theta \Rightarrow \hat{X}_\theta$, where $\hat{X}_\theta$ is the limiting process in Theorem \ref{mainsensitivityresult}.
This convergence and \eqref{mainp:lastrelation2} yield the following
\begin{align*}
\lim_{N \to \infty}  B^{N,2}_{\theta,i} & = \lim_{N \to \infty}  \lim_{h \to 0} \frac{1}{h} \E\left[  \ind_{ \{ \sigma^N_{i-1} < \gamma^N_h < \sigma^N_{i}  \} }\int_{t \wedge \gamma^N_h  }^t  \left( \mathbb{B}^N_{\theta+h} f_{\theta}^N( W^N_{\theta+h }(s) )  -  \mathbb{B}^N_{\theta} f_\theta^N( W^N_\theta(s)  )  \right)  ds \right] \\
& =\sum_{k \in \Gamma_2} \E \left[  \frac{ \partial \hat{\lambda}_{k}(\hat{X}_\theta(\sigma_{i-1}  ) ,\theta)  }{ \partial \theta }    R_{k,\theta}( \hat{X}_\theta(\sigma_{i-1}  ), f_\theta,t - \sigma_{i-1} \wedge t,k)  \right],
\end{align*}
where $\sigma_i$ is the $i$-th jump time of the process $\hat{X}_\theta$ (with $\sigma_0 = 0$) and the function $R_{k,\theta}$ be given by \eqref{mainp:defnrktheta}. Note that the quanity on the right hand side is $0$ if $\sigma_{i-1}\geq t$. Using \eqref{splittingofbntheta} and \eqref{limitofbn1} we get
\begin{align*}
\lim_{N \to \infty} B^N_\theta = \sum_{k \in \Gamma_2 }  \E\left[ \sum_{i = 0, \sigma_i <t }^{\infty}  \frac{ \partial \hat{\lambda}_{k}(\hat{X}_\theta(\sigma_{i}  ) ,\theta)  }{ \partial \theta }  R_{k,\theta}( \hat{X}_\theta(\sigma_{i}  ), f_\theta,t - \sigma_i \wedge t,k)   \right].
\end{align*}
This relation along with \eqref{snftheta_expn1}, \eqref{snithetalimit}, \eqref{defnsnitheraf} and \eqref{limitofantheta} gives us
\begin{align*}
\lim_{N \to \infty} S^N_\theta(f,t) & = \E \left( \frac{ \partial f_\theta  }{ \partial \theta}  ( \hat{X}_\theta(t)  )  \right) +  \sum_{k \in \Gamma_2}    \E \left[ \int_{0}^{t }  \frac{\partial  \hat{\lambda}_k(  \hat{X}_\theta(s)  ,\theta )   }{ \partial \theta} 
\left(  f_\theta( \hat{X}_\theta(s) + \zeta^s_k)  -  f_\theta(\hat{X}_\theta(s)) \right)ds \right] \\
& + \sum_{k \in \Gamma_2} \E  \left[  \sum_{i = 0, \sigma_i <t }^{\infty} \frac{ \partial \hat{\lambda}_{k}(\hat{X}_\theta(\sigma_{i}  ) ,\theta)  }{ \partial \theta }  R_{k,\theta}( \hat{X}_\theta(\sigma_{i}  ), f_\theta,t - \sigma_i \wedge t,k)   \right] ,
\end{align*}
which is same as \eqref{mainproof:sufficientcondiytion0} and this completes the proof of the theorem. 
\end{proof}

 \begin{remark}
 \label{rem:extension} 
In proving Theorem \ref{sec:maintheoremproof}, we assumed that the set $\mathbb{H}_v$ is finite for any $v \in \Pi_2 \mathcal{S}$ (see part (A) of Assumption \ref{assumptions2}). This means that if the state of the ``natural" dynamics is $v$ then the ``fast" dynamics is constrained within a compact set $\mathbb{H}_v$. This assumption can be relaxed at the expense of making the proof more technical. The only place where finiteness of $\mathbb{H}_v$ is crucial is in the proof of Proposition \ref{prop_regularity}. As explained in Remark \ref{rem:extensionfinitetocountable}, this proposition can be extended for Markov chains with countable state spaces. Assuming the existence of a suitable Lyapunov function for the fast dynamics, the proof of Theorem \ref{sec:maintheoremproof} goes through with minor modifications.
  \end{remark}

 \section{An Illustrative Example} \label{sec:example}
 
In this section we present a simple example to illustrate how our main result, Theorem \ref{mainsensitivityresult}, can be useful for the estimation of parameter sensitivity for multiscale networks. 
Consider a chemical reaction network with three species $S_1,S_2$ and $S_3$, and three reactions given by
\begin{align*}
S_1 \stackrel{ c_1 }{\longrightarrow} S_2 , \quad S_2  \stackrel{c_2}{\longrightarrow} S_1 \quad 
\textnormal{and} \quad S_2 \stackrel{c_3}{\longrightarrow} S_3.
\end{align*}
The rate constant of the $i$-th reaction is $c_i$, for $i=1,2,3$.
Such a network is used to model the cellular heat-shock response in \cite{HeatShockKhammash}, where $S_1$, $S_2$ and $S_3$ correspond to the $\sigma_{32}-$DnaK complex, the $\sigma_{32}$ heat shock regulator and the $\sigma_{32}$-RNAP complex, respectively. In this example, the first and second reactions are much faster than the third reaction. We assume that the rate constants are given by
\begin{align*}
c_1 = 1, \quad c_2 = 2 \quad \textnormal{and} \quad  c_3 = 5 \times 10^{-4}.
\end{align*}
We choose our sensitive parameter to be $\theta = c_1=1$ and the large \emph{normalization} parameter to be $N_0 = 10^{4}$. The three reactions along with their \emph{scaling} factors ($\beta_k$'s), \emph{propensity} functions ($\lambda_k$'s) and their \emph{stoichiometric} vectors ($\zeta_k$'s) are presented in Table \ref{table:hs}.
\begin{table}[ht]
\caption{Example of Heat Shock Response Model} 
\centering
\begin{tabular}{|c | c | c | c | c | }
\hline
Reaction Number &  Reaction  &  Scaling Factor & Propensity Function & Stoichiometric Vector \\ \hline 
$1$ & $S_1 \longrightarrow S_2$ & $ \beta_1 = 0 $  & $  \lambda_1(x_1,x_2,x_3) = \theta x_1$  & $ \zeta_1 = (-1,1,0)$  \\
$2$ & $S_2 \longrightarrow S_1$ & $ \beta_2 = 0 $  & $  \lambda_2(x_1,x_2,x_3) = 2 x_2$  &$ \zeta_2 = (1,-1,0)$  \\
$3$ & $S_2 \longrightarrow S_3$ & $ \beta_3 = -1  $  & $ \lambda_3(x_1,x_2,x_3) = 5 x_2$  &$ \zeta_3 = (0,-1,1)$  \\
\hline
\end{tabular}
\label{table:hs} 
\end{table}

Let $\left\{ X^{N_0}_\theta(t) = (X^{N_0}_{\theta,1}(t) ,X^{N_0}_{\theta,2}(t) , X^{N_0}_{\theta,3}(t) ) :  t \geq 0\right\}$ be the stochastic process representing the dynamics of this multiscale reaction network. Hence for any time $t \geq 0$ and $i=1,2,3$, $X^{N_0}_{\theta,i}(t)$ denotes the number of molecules of $S_i$. Suppose that the initial state of the system is $X^{N_0}_\theta(0) = (v_0,0,0)$ for $v_0 = 20$. Note that the sum of the three species numbers is preserved by all the reactions. 
Hence the state space for the process $X^{N_0}_\theta$ is
\begin{align*}
\mathcal{S} = \left\{ (x_1,x_2,x_3) \in \N^d_0 : \ x_1+x_2+x_3 = v_0 \right\}.
\end{align*}
Clearly for this multiscale network, the \emph{first} time-scale is $\gamma_1 = 0$ (see Section \ref{sec:firsttimescaleconv}) and the corresponding set of ``natural" reactions is $\Gamma_1 =\{1,2\}$. Similarly the \emph{second} time-scale is $\gamma_2 = -1$ (see Section \ref{sec:secondtimescaleconv}) and the corresponding set of ``natural" reactions is $\Gamma_2 =\{3\}$. If the time-scale of reference is $\gamma_2$ then the dynamics is given by the Markov process $X^N_{\gamma_2,\theta}$ with generator $\mathbb{A}^N_{\gamma_2,\theta}$ (see \eqref{defn_gen_aNgamma_theta}) with $N = N_0$. As described in Section \ref{sec:secondtimescaleconv}, under certain conditions we can construct a projection $\Pi_2$ for which the process $\Pi_2X^N_{\gamma_2,\theta}$ has a well-behaved limit as $N \to \infty$. In this example, this projection is given by
 \begin{align*}
\Pi_2(x_1,x_2,x_3) = (x_1 + x_2,x_3). 
\end{align*}
Note that $\Pi_2 \zeta_k = (0,0)$ for each $k \in \Gamma_1$ and $\Pi_2 \zeta_3 = (-1,1) $. 
For any $v = (v_1,v_2) \in \Pi_2 \mathcal{S}$, define the space $\mathbb{H}_v$ (see \eqref{defn_ev}) by
\begin{align*}
\mathbb{H}_v = \{(x,v_1-x) \in \N^2_0  : \ x=0,1,\dots,v_1\} 
\end{align*}
and let $\mathbb{C}^v_\theta $ be the generator given by \eqref{defn_cvtheta}. A Markov process with state space $\mathbb{H}_v$ and generator $\mathbb{C}^v_\theta $ is ergodic. The unique stationary distribution has the form of a \emph{binomial} distribution
\begin{align*}
\pi^\theta_{v}(x, y) = \frac{ v_1 !}{ x! y! } \left( \frac{\theta}{2+\theta} \right)^{y}\left( \frac{2}{2 + \theta} \right)^x \textnormal{ for } (x,y) \in \mathbb{H}_v .
\end{align*}
Define $\hat{\lambda}_3 :\Pi_2 \mathcal{S} \to \R_+$ by  
\begin{align*}
\hat{\lambda}_3(v_1,v_2) = \sum_{ (x,y) \in\mathbb{H}_v }  5 y \pi^\theta_{v}(x, y) =   \left( \frac{5 v_1 \theta}{2+ \theta} \right).
\end{align*}
Let $\{ \hat{X}_\theta(t) = ( \hat{X}_{\theta,1}(t), \hat{X}_{\theta,2}(t) ) : t \geq 0 \}$ be the $\Pi_2 \mathcal{S}$-valued process with the following random time change representation
\begin{align*}
\hat{X}_{\theta}(t) &=\left[
\begin{tabular}{c}
$v_0$ \\
0
\end{tabular} \right] + Y\left( \left( \frac{5\theta}{2+ \theta} \right)  \int_{0}^t  \hat{X}_{\theta,1}(s)
ds \right) 
\left[
\begin{tabular}{c}
-1 \\
1
\end{tabular} \right]
, 
\end{align*}
where $Y$ is a unit rate Poisson process. The due to Proposition \ref{convergenceresult2} we have $\Pi_2 X^N_{\gamma_2,\theta}\Rightarrow \hat{X}_{\theta}$ as $N \to \infty$. 

Let $f : \R^3 \to \R$ be the function given by $$f(x_1,x_2,x_3) = x_3,$$ and suppose we want to estimate
\begin{align*}
S^{N_0}_{\gamma_2,\theta} (f,t) = \frac{\partial}{\partial \theta} \E\left( f( X^{N_0}_{\gamma_2,\theta}(t) ) \right) = \frac{\partial}{\partial \theta} \E\left(  X^{N_0}_{\theta,3}(t)  \right)  .
\end{align*}
Note that $f(x) = f(\Pi_2 x)$ for all $x \in \mathcal{S}$, and hence the function $f_\theta$ (given by \eqref{defn_fthetalimit}) coincides with the function $f$ on the set $\Pi_2 \mathcal{S}$. Therefore from Theorem \ref{mainsensitivityresult} we obtain
\begin{align}
\label{example:sens1}
S^{N_0}_{\gamma_2,\theta} (f,t)  \approx \hat{S}_{\theta }(f,t) :=\frac{\partial  }{ \partial \theta } \E \left( f( \hat{X}_\theta(t) ) \right) 
= \frac{\partial  }{ \partial \theta } \E \left( \hat{X}_{\theta,2}(t)  \right) ,
\end{align}
for large values of $N_0$. We now demonstrate the usefulness of \eqref{example:sens1} in estimating $S^{N_0}_{\gamma_2,\theta} (f,t)  $. We will
numerically show that $S^{N_0}_{\gamma_2,\theta} (f,t) $ and $\hat{S}_{\theta }(f,t) $ are ``close" to each other and the estimation of $\hat{S}_{\theta }(f,t) $ is far less computationally demanding than the estimation of $S^{N_0}_{\gamma_2,\theta} (f,t)$.

To estimate parameter sensitivities we will use the \emph{coupled finite difference} (CFD) scheme developed in \cite{DA}. In this method, the sensitivity value $S^{N_0}_{\gamma_2,\theta} (f,t) $ is estimated by a \emph{finite-difference} of the form 
\begin{align*}
\frac{1}{h}\E \left(  f \left(X^{N_0}_{\gamma_2,\theta + h }(t)  \right)  - f \left( X^{N_0}_{\gamma_2,\theta  }(t)  \right) \right)
\end{align*}
for a \emph{small} h, and the processes $X^{N_0}_{\gamma_2,\theta + h }$ and $X^{N_0}_{\gamma_2,\theta}$ are coupled together in a special way to reduce the variance of the associated estimator. Replacing derivative by a finite-difference introduces a \emph{bias} in the sensitivity estimate, but we will ignore this issue here.
Using CFD, we estimate $S^{N_0}_{\gamma_2,\theta} (f,t)$ and $ \hat{S}_{\theta }(f,t)$, with $h=0.01$, $t=1$, $N_0 = 10^4$, $\theta=1$ and $v_0 =20$.  
The results are reported in Table \ref{table2}. The sensitivity values are written in the form $s\pm l$, which means that the $95\%$ confidence interval of the estimated value is $[s-l,s+l]$. For each estimation we use the minimum number of samples that is needed to ensure that $l \leq 0.05 |s|$, where $|\cdot|$ is the absolute value function. In the table, we also indicate the CPU time\footnote{All the computations in this paper were performed using C++ programs on an Apple machine with a 2.2 GHz Intel i7 processor.} (in seconds) that was needed for the estimation. The CPU time can be taken as a measure of the computational effort that was required to estimate the sensitivity value.
\begin{table}[h]
\caption{Estimation of sensitivity value for $f(x_1,x_2,x_3) = x_3$} 
\label{table2} 
\centering
\begin{tabular}{|c | c | c | c |}
\hline
   & Sensitivity Value & Number of Samples & CPU time (s)    \\ \hline 
   $S^{N_0}_{\gamma_2,\theta} (f,t)$ & $4.2138 \pm 0.2107 $  &34932 &1663.34   \\
$\hat{S}_{\theta }(f,t)$  & $4.2017 \pm 0.2100$ &  35056 & 0.2333   \\ \hline
\end{tabular}
\end{table}
Note that Table \ref{table2} shows that relation \eqref{example:sens1} holds but the time needed to estimate $\hat{S}_{\theta }(f,t) $ is approximately $7000$ times less than the time needed to estimate $S^{N_0}_{\gamma_2,\theta} (f,t)$ .

Now suppose we want to estimate $S^{N_0}_{\gamma_2,\theta} (f,t) $ for $f : \R^3 \to \R$ given by
$$f(x_1,x_2,x_3) = x_1.$$
In this case, $f_\theta : \Pi_2 \mathcal{S} \to \R$ can be computed as
\begin{align*}
f_\theta(v) = \sum_{(x,y) \in \mathbb{H}_v} x \pi^\theta_v (x,y) =\left( \frac{2v_1}{ 2 + \theta } \right)  \textnormal{ for any } v = (v_1,v_2) \in  \Pi_2 \mathcal{S} .  
\end{align*}
Hence Theorem \ref{mainsensitivityresult} implies that
\begin{align*}
S^{N_0}_{\gamma_2,\theta } (f,t)  \approx  \hat{S}_{\theta }(f_\theta,t)  = \frac{\partial  }{ \partial \theta } \E \left( f_\theta( \hat{X}_\theta(t) ) \right) 
= \frac{\partial  }{ \partial \theta } \left(  \frac{ 2 \E \left( \hat{X}_{\theta,1}(t)  \right)  }{2 + \theta} \right)
=\left( \frac{2}{2+\theta}  \right)  \left[  \frac{\partial  }{ \partial \theta }\E \left( \hat{X}_{\theta,1}(t)  \right) - \frac{ \E \left( \hat{X}_{\theta,1}(t)  \right)  }{2 + \theta }  \right].
\end{align*}
As before we estimate $S^{N_0}_{\gamma_2,\theta} (f,t)$ and $ \hat{S}_{\theta }(f_\theta,t)$ using CFD, with $h=0.01$, $t=1$, $N_0 = 10^4$, $\theta=1$ and $v_0 =20$.  
The results are reported in Table \ref{table3}.
\begin{table}[h]
\caption{Estimation of sensitivity value for $f(x_1,x_2,x_3) = x_1$} 
\label{table3} 
\centering
\begin{tabular}{|c | c | c | c |}
\hline
   & Sensitivity Value & Number of Samples & CPU time (s)    \\ \hline 
$S^{N_0}_{\gamma_2,\theta} (f,t)$ & $-3.3946 \pm 0.1697 $ & 43745 & 2181.5  \\
$\hat{S}_{\theta }(f_\theta,t)$ & $-3.6369 \pm 0.1818 $ & 20827 & 0.1396   \\ \hline
\end{tabular}
\end{table}
As before, Table \ref{table3} shows that $S^{N_0}_{\gamma_2,\theta } (f,t)  \approx  \hat{S}_{\theta }(f_\theta,t) $ but the estimation of $S^{N_0}_{\gamma_2,\theta } (f,t) $ is around $15000$ times slower than the estimation of $ \hat{S}_{\theta }(f_\theta,t) $.

This example clearly illustrates that our main result, Theorem \ref{mainsensitivityresult}, can be used to obtain enormous savings in the computational effort that is required for the estimation of parameter sensitivities for multiscale networks.

\renewcommand {\theequation}{A.\arabic{equation}}
\appendix
\setcounter{equation}{0}

\section{Appendix.}\label{sec:app}

Let $S_1$ and $S_2$ be open subsets of $\R^n_+$ and $\R^m$ respectively. Let $\mathbb{A} \subset \mathcal{B}(S_1 \times S_2) \times \mathcal{B}(S_1 \times S_2)$ be an operator whose domain $\mathcal{D}(\mathbb{A})$ include all functions $f : S_1 \times S_2 \to \R$ of the form
\begin{align}
\label{app:funcfg}
f(x,y) = g(x),
\end{align}
where $g$ is some function in $\mathcal{B}(S_1)$. Let $U \subset S_1 \times S_2$ be an open set and let $X$ be a stochastic process with initial distribution $\nu \in \mathcal{P}(S_1\times S_2)$ and sample paths in $D_{S_1 \times S_2}[0,\infty)$. Define a stopping time with respect to the filtration generated by the process $X$ as
\begin{align}
\label{app:sttimetau}
\tau = \inf \{ t \geq 0 : X(t) \notin U \textnormal{  or  } X(t-) \notin U  \}.
\end{align}
Then $X$ is a solution of the \emph{stopped martingale problem} (see Section 6, Chapter 4 in \cite{EK}) for $(\mathbb{A}, \nu,U)$ if $X(\cdot) = X(\cdot \wedge \tau)$ a.s. and 
\begin{align*}
f(X(t)) - \int_{0}^{t\wedge \tau} \mathbb{A} f( X(s) )ds 
\end{align*}
is a martingale for each $f \in \mathcal{D}(\mathbb{A})$.

  Let $\Pi : S_1 \times S_2 \to S_1$ be the projection map defined by $\Pi(x,y) =x$. 
Suppose that for any $g \in \mathcal{B}(S_1)$ and $f$ given by \eqref{app:funcfg} we have
\begin{align}
\label{app:gen_relation}
\mathbb{A} f(x,y) = \sum_{k=1}^K \lambda_k(x,y) \left( g(x + \zeta_k) -g(x) \right),
\end{align}
where $\zeta_1,\dots,\zeta_K$ are certain vectors in $\R^n$ and $\lambda_1,\dots,\lambda_K$ are positive functions on $S_1 \times S_2$ satisfying the following : if $\lambda_k(x,y) >0$ for some $(x,y) \in S_1 \times S_2$ then $(x+\zeta_k) \in S_1$. Furthermore we assume that the function
\begin{align}
\label{app:whatispolygrow}
\sum_{k=1 , \langle \bar{1}_d, \zeta_k  \rangle > 0   }^K \lambda_k(x,y)
\end{align}
is linearly growing with respect to projection $\Pi$ (see Definition \ref{polynomialgrowth}) .
\begin{lemma}
\label{lemma1:app}
Fix a $w_0 =(x_0,y_0) \in S_1 \times S_2$ and let $\delta_{w_0} \in \mathcal{P}( S_1 \times S_2)$ be the distribution that puts all the mass at $w_0$. For any $M \in \N$, let $U_M$ be the open set 
\begin{align*}
U_M = \{ (x,y) \in S_1\times S_2 : \|x\| <  M  \}.
\end{align*}
Assume that the stopped martingale problem for $(\mathbb{A}, \delta_{w_0} ,U_M)$ has a unique solution $W_M$ for each $M$. Let $\tau_M$ be the stopping time defined by \eqref{app:sttimetau} with $U$ replaced by $U_M$. Then we have the following.
\begin{itemize}
\item[(A)] For any $T >0$, $\lim_{M \to \infty} \P( \tau_M < T ) = 0$. 
\item[(B)] There exists a unique solution $W$ for the (unstopped) martingale problem for $(\mathbb{A}, \delta_{w_0} )$. Moreover for any positive integer $p$ and $T>0$ we have
\begin{align*}
 \sup_{t \in [0,T]} \E( \| \Pi W(t)   \|^p ) <\infty.
\end{align*}
\item[(C)] If a function $f : S_1 \times S_2 \to \R$ is polynomially growing with respect to projection $\Pi$, then for any $T \geq 0$ 
\begin{align*}
\sup_{t \in [0,T]} \E\left( \left|  f ( W(t) )  \right|\right) <\infty.
\end{align*}
\item[(D)]  The martingale problem for $\mathbb{A}$ is well-posed.
\end{itemize}
\end{lemma}
\begin{proof}
Suppose that $W_M(t) = (X_M(t),Y_M(t) )$ for all $t \geq 0$, where $X_M$ and $Y_M$ are processes with state spaces $S_1$ and $S_2$ respectively.
Let $q = \max\{ \langle \bar{1}_d, \zeta_k\rangle : k =1,\dots,K \}$. For a large $M$ and a positive integer $p$ define $g \in \mathcal{B}(S_1)$ by
\begin{align*}
g(x) = \|x\|^p \wedge (M + q)^p.
\end{align*} 
Assume that $\|x_0\|^p < M  $ and note that the definition of $g$ implies that for $t\leq \tau_M$ we have $g( X_M(t)) = \|X_M(t)\|$.
Let $f : S_1 \times S_2 \to \R$ be the function given by $f(x,y) =g(x)$. Then $f \in \mathcal{D}(\mathbb{A} )$ and hence
\begin{align*}
&f(W_M( t \wedge \tau_M )) -x_0^p- \int_{0}^{t \wedge \tau_M} \mathbb{A} f( W_M(s) ) ds \\
& = \|  X_M( t \wedge \tau_M )) \|^p  -x_0^p- \int_{0}^{t \wedge \tau_M} \sum_{k=1}^K \lambda_k( X_M(s),Y_M(s) ) \left(   \| X_M( s  ) +\zeta_k \|^p -\| X_M( s  ))  \|^p   \right) ds
\end{align*}
is a martingale starting at $0$. Taking expectations we get
\begin{align*}
\E\left(  \|  X_M( t \wedge \tau_M )) \|^p \right) = x_0^p + \E\left(  \int_{0}^{t \wedge \tau_M} \sum_{k=1}^K \lambda_k( X_M(s),Y_M(s) ) \left(   \| X_M( s  ) +\zeta_k \|^p -\| X_M( s  )) \|^p   \right) ds\right)
\end{align*}
Our assumption on the functions $\lambda_1,\dots,\lambda_K$ implies that when $\lambda_k(X_M(s),Y_M(s) )  > 0$, then $(X_M(s)+ \zeta_k) \in S_1 \subset \R_+^d$ and hence $\| X_M(s) + \zeta_k\| = \langle \bar{1}_d, X_M(s)\rangle + \langle \bar{1}_d, \zeta_k \rangle$. This gives us
 \begin{align*}
 \E\left(  \|  X_M( t \wedge \tau_M )) \|^p \right) &= x_0^p + \E\left(  \int_{0}^{t \wedge \tau_M} \sum_{k=1}^K \lambda_k( X_M(s),Y_M(s) ) \left(   \| X_M( s  ) +\zeta_k \|^p -\| X_M( s  )) \|^p  \right) ds\right) \\
 & =  x_0^p + \E\left(  \int_{0}^{t \wedge \tau_M} \sum_{k=1}^K \lambda_k( X_M(s),Y_M(s) ) \left(   ( \langle \bar{1}_d, X_M(s)\rangle + \langle \bar{1}_d, \zeta_k \rangle)^p -\langle \bar{1}_d, X_M( s  ) \rangle^p  \right) ds\right) \\
& \leq x_0^p + 2^p q^p  \E \left(  \int_{0}^{t }  \sum_{k \in P} \lambda_k( X_M(s\wedge \tau_M),Y_M(s\wedge \tau_M) ) \left(  \| X_M(s \wedge \tau_M ) \|^{p-1} + 1 \right) ds    \right),
\end{align*}
where $P = \{k=1,\dots,K : \langle \bar{1}_d ,\zeta_k \rangle > 0\}$. Since the function given by \eqref{app:whatispolygrow} is linearly growing with respect to projection $\Pi$, we can find a positive constant $C$ (independent of $M$) such that
\begin{align*}
\E \left( \| X_M(t \wedge \tau_M) \|^p )\right) &\leq x_0^p + C t + C  \int_{0}^t  \E \left( \|X_M(s \wedge \tau_M)\|^p \right) ds.   
\end{align*}•
Gronwall's inequality implies that
\begin{align}
\label{app:gronwall1}
\E \left( \| X_M(t \wedge \tau_M) \|^p \right) \leq \left( x_0^p  + C   t \right) e^{C t}.
\end{align}
Using Markov's inequality we obtain
\begin{align*}
\lim_{M \to \infty} \P \left(   \tau_M <   t  \right) &= \lim_{M \to \infty}\P \left( \| X_M(t \wedge \tau_M) \|^p    \geq M^p \right)  \leq  \lim_{M \to \infty}\frac{\E \left( \| X_M(t \wedge \tau_M) \|^p )\right) }{M^p}  = 0.
\end{align*}•
The last limit is $0$ due to \eqref{app:gronwall1}. This proves part (A) of the lemma. From Theorem 6.3 in Chapter 4 of \cite{EK} we can conclude that the martingale problem for $( \mathbb{A} ,\delta_{w_0} )$ has a unique solution $W$. In fact for any $M \in \N$, the process $W_M(\cdot \wedge \tau_M)$ has the same distribution as the process $W(\cdot  \wedge \tau_M)$. Therefore using \eqref{app:gronwall1} we get
\begin{align}
\label{app:gronwall2}
\E \left( \| \Pi W(t \wedge \tau_M) \|^p \right)  = \E \left(  \| \Pi W_M(t \wedge \tau_M) \|^p  \right) \leq \left( x_0^p  + C   t \right) e^{C t}.
\end{align}
Since $\tau_M$ is monotonically increasing with $M$, we must have that $\tau_M \to \infty$ a.s. as $M \to \infty$.
Letting $M \to \infty$ in \eqref{app:gronwall2} and using Fatou's lemma we obtain
\begin{align*}
 \E \left(  \| \Pi W(t) \|^p \right) \leq \lim_{M \to \infty} \E \left( \| \Pi W(t \wedge \tau_M) \|^p  \right) \leq \left( x_0^p  + C t \right) e^{C t}.
\end{align*}
Taking supremum over $t \in [0,T]$ proves part (B) of the lemma. The proof of part (C) is immediate from part (B). Since part (B) of this lemma holds for any $w_0$, the martingale problem for $\mathbb{A}$ is well-posed and this proves part (D).
\end{proof}

Using the above lemma we now prove the main result of this section.
\begin{lemma}
\label{mainlemmapp}
Recall the definition of operator $\mathbb{B}^N_\theta$ from \eqref{genbntheta}. 
\begin{itemize}
\item[(A)] The martingale corresponding to $\mathbb{B}^N_\theta$ is well-posed. 
\item[(B)] Let $W^N_\theta$ be the $\hat{\mathcal{S}}$-valued Markov process with generator $\mathbb{B}^N_\theta$ and initial state $(t_0,v_0,z_0) \in \hat{\mathcal{S}}$.
For any $M \in \N$, define a stopping time by
\begin{align}
\label{app:sttimetaumn}
\sigma^N_M =\inf\{ t \geq 0 : \|\Pi_{\hat{ \mathcal{S} }} W^N_\theta(t) \| > M  \}.
\end{align}
Then for any $T>0$
\begin{align}
\label{app:sttimegototinfinitye}
\lim_{M \to \infty} \sup_{N \in \N} \P\left(  \sigma^N_M < T\right) = 0.
\end{align}
\item[(C)] For any positive integer $p$ and any $T>0$
 \begin{align}
\label{app:mmomnetboundactualprocess}
\sup_{t \in [0,T]} \sup_{N \in \N } \E\left(  \|\Pi_{\hat{ \mathcal{S} }} W^N_\theta(t) \|^p \right) < \infty.
\end{align}
\item[(D)]  Let $f : \mathcal{S} \to \R$ be a function which is polynomially growing with respect to projection $\Pi_2$, and define $f^N_\theta$ by \eqref{defn_ftheta}. Then for any positive integer $p$ and $T>0$ 
\begin{align}
\label{app:polygrowth2}
\sup_{N \in  \N } \sup_{t \in [0,T]} \E\left( \left|  f^N_\theta ( W^N_\theta(t) )  \right|^p \right)  < \infty \quad \textnormal{ and } \quad \sup_{N \in \N }
 \E\left( \int_{0}^T \left| \mathbb{B}^N_\theta  f^N_\theta ( W^N_\theta  (t))  \right|^p dt \right) <\infty. 
 \end{align} 
 \item[(E)]  Let $f$ and $f^N_\theta$ be as in part (D). 
 For any $T \geq 0$ and any stopping time $\sigma$ we have
\begin{align}
\label{corr:appdynkin}
\E\left( f^N_\theta( W^N_\theta(T \wedge \sigma) ) \right) = f^N_\theta(t_0,v_0,z_0) + \E\left( \int_{0}^{T \wedge \sigma} \mathbb{B}^N_\theta f^N_\theta( W^N_\theta(t) ) dt\right).
\end{align}
\item[(F)] The sequence of processes $\{W^N_\theta : N \in \N\}$ is tight in the space $D_{ \hat{\mathcal{S}}}[0,\infty)$.
\end{itemize}
\end{lemma}
\begin{proof}
Note that on the set
\begin{align*}
U_M = \{ (t,v,z ) \in \hat{ \mathcal{S} }  : \|v\| < M\},
\end{align*}
the functions $\{ \rho^N_{k,\theta} : k \in \Gamma_2\}$ are bounded. If we define each $\rho^N_{k,\theta} $ to be $0$ outside the set $U_M$, then the resulting operator $\mathbb{B}^N_{M,\theta}$ can be seen as a bounded perturbation of the translation operator
\begin{align*}
\mathbb{T}f(t,v,z) = \frac{ \partial f(t,v,z)}{ \partial t },
\end{align*}
which certainly has a well-posed martingale problem. From Theorem 4.10.3 in \cite{EK} we can conclude that the martingale problem for $\mathbb{B}^N_{M,\theta}$ is well-posed. This implies that for any initial state $w_0 \in \hat{ \mathcal{S} } $, the stopped martingale problem for $(\mathbb{B}^N_\theta, \delta_{w_0} , U_M )$ is well-posed. Assumption \ref{assforsensitivityresult} imply that the function 
\begin{align}
\label{defn_hatrho}
\hat{\rho}^N_{\theta} (t,v,z) = \sum_{k \in \Gamma_2 , \langle \bar{1}_d, \zeta^s_k  \rangle > 0   }  \rho^N_{k,\theta} (t,v,z)
\end{align}
is linearly growing with respect to projection $\Pi_{\hat{ \mathcal{S} }}$ (given by \eqref{defn:projhats}). Therefore part (B) of Lemma \ref{lemma1:app} shows that there is a unique solution for the martingale problem for $(\mathbb{B}^N_\theta, \delta_{w_0})$. Hence the martingale problem for $\mathbb{B}^N_\theta$ is well-posed and this proves part (A).

Let $W^N_\theta$ be the $\hat{\mathcal{S}}$-valued Markov process with generator $\mathbb{B}^N_\theta$ and initial state $(t_0,v_0,z_0)$. If $\hat{\rho}^N_{\theta} $ is given by \eqref{defn_hatrho}, then due to Assumption \ref{assforsensitivityresult} we can find constants $C,r \geq 0$ such that
\begin{align*}
|\hat{\rho}^N_{\theta} (t,v,z) |\leq C(1+\|v\|^r) \textnormal{ for all }(t,v,z) \in \hat{\mathcal{S}} \textnormal{ and } N \in \N.
\end{align*}
Using this fact we can rework the proof of Lemma \ref{lemma1:app} to prove parts (B) and (C). 

Let $f : \mathcal{S} \to \R$ be a function which is polynomially growing with respect to projection $\Pi_2$ and define $f^N_\theta$ by \eqref{defn_ftheta}.
Remark \ref{rem:functioninclassc} implies that the sequences of functions $\{f^N_\theta : N\in \N\}$ and $\{ \mathbb{B}^N_\theta  f^N_\theta : N\in \N\}$ are polynomially growing with respect to projection $\Pi_{\hat{ \mathcal{S} }} $. Therefore part (D) is an easy consequence of part (C).

Corresponding to the function $f$ define a function $f_M : \mathcal{S} \to \R$ by
\begin{align*}
f_M(x) = f(x) \wedge M.
\end{align*}
Let $f^N_{M,\theta}$ be the given by \eqref{defn_ftheta}, with $f$ replaced by $f_M$. Since $f_M$ is bounded, the function $f^N_{M,\theta}$ is in class $\mathcal{C}$. Using Dynkin's theorem (see Lemma 19.21 in \cite{Kal}) we get
\begin{align*}
\E\left( f^N_{M,\theta}( W^N_\theta(T \wedge \sigma) ) \right) = f^N_{M,\theta}(t_0,v_0,z_0) + \E\left( \int_{0}^{T \wedge \sigma} \mathbb{B}^N_\theta f^N_{M,\theta}( W^N_\theta(t) ) dt\right).
\end{align*}
Taking the limit $M \to \infty$ and using the dominated convergence theorem proves part (E).

To show that the sequence $\{W^N_\theta : N \in \N\}$ is tight we first have to prove the compact containment criterion (see Chapter 3 in \cite{EK}). This means that for any $T,\epsilon >0$ we exhibit a compact set $K_{\epsilon,T} \subset \hat{ \mathcal{S} }$ such that
\begin{align}
\label{app:compactcontainment}
\inf_{N \in  \N } \P\left( W^N_\theta(t) \in K_{\epsilon,T}  \textnormal{ for all }t \in [0,T]  \right) \geq 1- \epsilon.
\end{align}
Let $\sigma^N_M $ be the stopping time given by \eqref{app:sttimetaumn}. For any $t \geq 0$, we can write $W^N_\theta(t) = ( \tau^N_\theta(t), V^N_\theta(t), Z^N_\theta(t) )$ (see \eqref{decompwntheta}). Fix an $\epsilon >0$ and $T>0$. Part (B) shows that we can find a $M >0$ large enough so that
\begin{align}
\label{supoftauleqeps}
 \sup_{N \in \N}  \P(\sigma^N_M \leq T) < \epsilon. 
\end{align}
Note that for any $t \geq 0$, if $V^N_\theta(t) = v$ then $\tau^N_\theta(t) \in [0,t+t_0]$ and $ Z^N_\theta(t)\in \mathbb{H}_v$ where $\mathbb{H}_v$ is a finite set.
This shows that for any $t  <  \sigma^N_M$ we have $W^N_\theta(t) \in K_{\epsilon,T} $ where $K_{\epsilon,T} $ is the compact set given by
\begin{align*}
K_{\epsilon,T} = \left\{ (t,v,z) \in  \hat{ \mathcal{S} } : t \in [0,T+t_0] , \|v\| \leq M \textnormal{ and } z \in \mathbb{H}_v \right\}.
\end{align*}
Hence
\begin{align*}
\P\left( W^N_\theta(t) \in K_{\epsilon,T}  \textnormal{ for all }t \in [0,T]  \right) \geq  \P(\sigma^N_M > T) =1 - \P(\sigma^N_M \leq T)  .
\end{align*}
Taking supremum over $N$ and using \eqref{supoftauleqeps} proves \eqref{app:compactcontainment}.

Now that we have shown the compact containment condition, Theorem 3.9.1 in \cite{EK} allows us to verify the tightness of $\{W^N_\theta : N \in \N\}$ by proving that for any $f  \in \mathcal{C}$ (see \eqref{classc}), the sequence of processes $\{ f( W^N_\theta (\cdot) ) : N \in \N\}$ is tight in the space $D_\R[0,\infty)$. Note that
\begin{align*}
f(W^N_\theta(t)) - \int_{0}^t \mathbb{B}^N_\theta f(W^N_\theta(s)  ) ds
\end{align*}
is a martingale and part (D) of the lemma shows that
\begin{align*}
\E\left( \int_{0}^t \left| \mathbb{B}^N_\theta  f^N_\theta ( W^N_\theta  (s))  \right|^2ds \right) <\infty 
\end{align*}
for any $t \geq 0$. The tightness of the sequence $\{ f( W^N_\theta (\cdot) ) : N \in \N\}$ is immediate from Theorem 3.9.4 in \cite{EK}. This completes the proof of part (E) of the lemma.
\end{proof}

\setcounter{equation}{0}

%in this file we prove some results about fast ergodic markov chains
\bibliographystyle{abbrv}
\bibliography{references}
\end{document}